\documentclass[12pt, reqno]{amsart} 
\usepackage{amssymb,amscd,amsfonts,amsbsy}
\usepackage{latexsym}
\usepackage{exscale}
\usepackage{amsmath,amsthm,amsfonts}
\usepackage{mathrsfs}
\usepackage{xcolor} 
\usepackage[colorlinks=true,linkcolor=blue,citecolor=blue,urlcolor=blue]{hyperref} 
\usepackage{esint} 
\usepackage{amssymb} 
\usepackage{stmaryrd}
\usepackage{cite} 
\usepackage{tikz}
\usepackage{enumerate}
\calclayout

\newcommand{\subRn}{X}
\newcommand{\Q}{{Q_{j}^{k}}}

\allowdisplaybreaks[3] 

\newtheorem{theorem}{Theorem}[section]

\newtheorem{lemma}[theorem]{Lemma}

\newtheorem{definition}[theorem]{Definition}
\newtheorem{remark}[theorem]{Remark}

\numberwithin{equation}{section}

\newtheorem{que}{Question}

\makeatletter
\@addtoreset{equation}{section}
\makeatother

\def\rn{{\mathbb R^n}}

\def\M{{\mathscr M}}

\def\A{{\mathscr A}}

\def\P{{\mathscr P}}

\def\pp{{p(\cdot)}}

\def\cpp{{p'(\cdot)}}

\def\XXint#1#2#3{{\setbox0=\hbox{$#1{#2#3}{\int}$}
     \vcenter{\hbox{$#2#3$}}\kern-.5\wd0}}

\DeclareMathOperator{\dist}{dist}

\newcommand{\avf}{{\langle f\rangle}}

\newcommand{\pap}{{p_{1}(\cdot)}}
\newcommand{\pbp}{{p_{2}(\cdot)}}
\newcommand{\cpap}{{p_{1}^{\prime}(\cdot)}}
\newcommand{\cpbp}{{p_{2}^{\prime}(\cdot)}}

\newcommand{\Lp}{L^{p(\cdot)}}

\newcommand{\qq}{{q(\cdot)}}

\newcommand{\Lq}{L^{q(\cdot)}}

\def\d{\mathcal{D}}
\def\v{\mathtt{v}}

\def\Z{\mathbb{Z}}

\def\R{\mathbb{R}}

\def\dist{\operatorname{dist}}

\begin{document}
\title[New variable weighted conditions...]
{\bf New variable weighted conditions for fractional maximal operators over spaces of homogeneous type}

\author[X. Cen]{Xi Cen}
\address{Xi Cen\\
	School of Mathematics and Physics\\
	Southwest University of Science and Technology\\
	Mianyang 621010 \\
	People's Republic of China}\email{xicenmath@gmail.com}

\date{August 7, 2024.}




\begin{abstract} 
Based on the rapid development of dyadic analysis and the theory of variable weighted function spaces over the spaces of homogeneous type $(X,d,\mu)$ in recent years, we systematically consider the quantitative variable weighted characterizations for fractional maximal operators. On the one hand, a new class of variable multiple weight $A_{\vec{p}(\cdot),q(\cdot)}(X)$ is established, which enables us to prove the strong and weak type variable multiple weighted estimates for multilinear fractional maximal operators ${{{\mathscr M}_{\eta }}}$. More precisely, 
\[
{\left[ {\vec \omega } \right]_{{A_{\vec p( \cdot ),q( \cdot )}}(X)}} \lesssim {\left\| \mathscr{M}_\eta \right\|_{\prod\limits_{i = 1}^m {{L^{p_i( \cdot )}}({X,\omega _i})}  \to {L^{q( \cdot )}}(X,\omega )({WL^{q( \cdot )}}(X,\omega ))}} \le  {C_{\vec \omega ,\eta ,m,\mu ,X,\vec p( \cdot )}}.
\]

On the other hand, on account of the classical Sawyer's condition $S_{p,q}(\mathbb{R}^n)$, a new variable testing condition $C_{{p}(\cdot),q(\cdot)}(X)$ also appears in here, which allows us to obtain quantitative two-weighted estimates for fractional maximal operators ${{{M}_{\eta }}}$. To be exact,
\begin{align*}
    \|M_{\eta}\|_{L^{p(\cdot)}(X,\omega)\rightarrow L^{q(\cdot)}(X,v)} \lesssim \sum\limits_{\theta  = \frac{1}{{{p_{\rm{ - }}}}},\frac{1}{{{p_{\rm{ + }}}}}} {{{\left( {{{[\omega ,v]}_{C_{p( \cdot ),q( \cdot )}^2(X)}} + {{[\omega ]}_{C_{p( \cdot ),q( \cdot )}^1(X)}}{{[\omega ,v]}_{C_{p( \cdot ),q( \cdot )}^2(X)}}} \right)}^\theta }}.
\end{align*}
The implicit constants mentioned above are independent on the weights. 
\end{abstract}

\maketitle
\tableofcontents

\section{\bf Introduction}

\subsection{The Hypotheses}
~~

First, let us look back the concepts related to the spaces of homogeneous type, which ones may also find in \cite{Cruz2022} and \cite{Cen2024}.

\begin{definition}
For a positive function $d: X \times X \rightarrow [0, \infty)$, $X$ is a set, the quasi-metric space $(X, d)$ satisfies the following conditions:
\begin{enumerate}
    \item  When $x=y$, $d(x, y)=0$.
    \item $d(x, y)=d(y, x)$ for all $x, y \in X$.
    \item  For all $x, y, z \in X$, there is a constant $A_0 \geq 1$ such that $d(x, y) \leq A_0(d(x, z)+d(z, y))$.
\end{enumerate}
\end{definition}

\begin{definition}
 Let $\mu$ be a measure of a space $X$. For a quasi-metric ball $B(x,  r)$ and any $r>0$, if $\mu$ satisfies doubling condition, then there exists a doubling constant $C_\mu \geq 1$, such that  
$$
0<\mu(B(x, 2 r)) \leq C_\mu \mu(B(x, r))<\infty.
$$
\end{definition}

\begin{definition}
For a non-empty set $X$ with a qusi-metric $d$, a triple $(X, d, \mu)$ is said to be a space of homogeneous type if $\mu$ is a regular measure which satisfies doubling condition on the $\sigma$-algebra, generated by open sets and quasi-metric balls.
\end{definition}

Considering a measurable function \(p: E \rightarrow [1, \infty)\) on a subset \(E \subseteq X\), we define \(p_-(E) = \operatorname*{ess\,inf}\limits_{x \in E} p(x)\) and \(p_+(E) = \operatorname*{ess\,sup}\limits_{x \in E} p(x)\), with \(p_-\) and \(p_+\) specifically denoting these quantities over the entire space \(X\). Furthermore, we introduce some sets of measurable functions based on these definitions.
\begin{align*}
	&{\P}\left( E \right) = \{ p( \cdot ) :{\rm{E}} \to \left[ {1,\infty } \right) \text{ is measurable: } 1 < {p_ - }(E) \le {p_{\rm{ + }}}(E) < \infty \};\\
	&{{\P}_1}\left( {E} \right) = \{ p( \cdot ) :{\rm{E}} \to \left[ {1,\infty }\right) \text{ is measurable: } 1 \le {p_ - }(E) \le {p_{\rm{ + }}}(E) < \infty \};\\
	&{{\P}_0}\left( {E} \right) = \{ p( \cdot ) :{\rm{E}} \to \left[ {0,\infty }\right) \text{ is measurable: } 0 < {p_ - }(E) \le {p_{\rm{ + }}}(E) < \infty \}.
\end{align*}
Obviously, ${\P}\left( E \right) \subseteq {{\P}_1}\left( {E} \right) \subseteq {{\P}_0}\left( {E} \right)$. When $E=X$,  we write $\mathscr{P}(X)$ by ${\P}$ for convenience.

\begin{definition}
Let $1\leq p_-\leq p_+\leq \infty$, the variable exponent Lebesgue spaces with Luxemburg norm is defined as
\begin{equation*}
L_{}^{p( \cdot )}(X) = \{ f:{\left\| f \right\|_{L_{}^{p( \cdot )}(X)}}: = \inf \{ \lambda  > 0:{\rho _{p( \cdot )}}(\frac{f}{\lambda }) \le 1\}  < \infty \},
\end{equation*}
where ${\rho _{p( \cdot )}}(f) = \int_X {{{\left| {f(x)} \right|}^{p(x)}}dx}+\|f\|_{L^{\infty}(X_\infty)}.$ We always abbreviate $\|\cdot\|_{L^{p(\cdot)}(X)}$ to $\|\cdot\|_{p(\cdot)}$. For every ball $B \subseteq X$, if $\rho_\pp(f\chi_B) < \infty$, then $f$ is said to be locally $\pp$-integrable.
\end{definition}
 
In fact, the above spaces are Banach spaces (precisely, Banach function spaces), to which readers can refer \cite{red}.

Let $\omega$ be a weight on $X$. The variable exponent weighted Lebesgue spaces are defined by
\begin{equation*}
	L_{}^{p( \cdot )}(X,\omega) = \{ f:{\left\| f \right\|_{L_{}^{p( \cdot )}(X,\omega)}} := {\left\| {\omega f} \right\|_{L_{}^{p( \cdot )}(X)}} < \infty \}.
\end{equation*}

\begin{definition}
For any $x, y \in X$ and $d(x,y)<\frac{1}{2}$, we say $p(\cdot) \in LH_{0}$, if  	\begin{equation}\label{LH0}
		|p(x)-p(y)| \lesssim \frac{1}{\log (e+1 /d(x,y))}.
	\end{equation}
We say $p(\cdot) \in LH_{\infty}$ (respect to a point $x_0 \in X$), if there exists $p_{\infty} \in X$, for any $x \in X$,
	\begin{equation}\label{LHinfty}
		\left|p(x)-p_{\infty}\right| \lesssim \frac{1}{\log (e+d(x,x_0))} .
	\end{equation}
We denote the globally log-H\"{o}lder continuous functions by $LH=LH_{0} \cap LH_{\infty}$.
\end{definition}
According to the above definition, it seems to relate to the choice of the point $x_0$. However, through \cite{Ad2015}, we can know that such a choice is immaterial.

\begin{lemma}\label{LHxy}
    For any $y_0\in X$, if $p(\cdot) \in L H_{\infty}$ with respect to $x_0\in X$, then $p(\cdot) \in L H_{\infty}$ with respect to $y_0$.
\end{lemma}

If $x_0$ is not chosen definitely, we always suppose that $X$ has an arbitrary given point $x_0$.

For $\eta  \in \left[ {0,m} \right)$, every ball $B \subseteq X$, we define the multilinear fractional averaging operator by
$${\A_{\eta,B}}(\vec f)(x): = {\mu(B)^{\eta}}\left( {\prod\limits_{i = 1}^m {{{\left\langle {{f_i}} \right\rangle }_B}} } \right){\chi _B}(x),$$
The multilinear fractional maximal operator $\M_\eta$ on the spaces of homogeneous type is defined as
$$
\M_\eta (\vec{f})(x) : =\sup _{B \subseteq X}{\A_{\eta,B}}(\vec f)(x).
$$
When $X=\rn$, we take $\eta=\frac{\alpha}{n}$ and write $\M_\eta$ by $\M_\alpha$. When $m=1$, we denote $\M_\eta$ by $M_\eta$. When $\eta=0$, we deonte $\M_\eta$ by $\M$.

We next set up a new class of variable multiple weight \(A_{\vec{p}(\cdot),q(\cdot)}(X)\) and its characterizations.

\begin{definition}\label{vweight3}
    Let  $p_i(\cdot) \in \P$, $i=1,\cdots,m$, with $\frac{1}{{p( \cdot )}} = \sum\limits_{i = 1}^m {\frac{1}{{{p_i}( \cdot )}}}$, $\eta{\rm{ = }}\frac{1}{{{p}( \cdot )}}{\rm{ - }}\frac{1}{{{q}( \cdot )}} \in [0,m)$, and $\vec{\omega}:=(\omega_1,\ldots,\omega_m,\omega)$ is a multiple weight. We say $\vec{\omega} \in A_{\vec{p}(\cdot),q(\cdot)}(X)$, if it satisfies
        \begin{align*}
        {\left[ {\vec \omega } \right]_{{A_{\vec p( \cdot ),q( \cdot )}}}(X)}: = \mathop {\sup }\limits_{B \subseteq X} {\mu(B)^{\eta - m}}{\left\| {\omega {\chi _B}} \right\|_{{L^{q( \cdot )}}}}\prod\limits_{i = 1}^m {{{\left\| {\omega _i^{ - 1}{\chi _B}} \right\|}_{{L^{{p_i}^\prime ( \cdot )}}}}}  < \infty.
        \end{align*}
        In generally, we always abbreviate ${A_{\vec p( \cdot ),q( \cdot )}}(X)$ to ${A_{\vec p( \cdot ),q( \cdot )}}$.
\end{definition}

\begin{figure}[!h]
	\begin{center}
		\begin{tikzpicture}
			\node (1) at(-4,0) {$A_p(\R^n)$};
			\node (2) at(-1,3.5) {$A_{p,q}(\R^n)$};
			\node (3) at(-1,1.5) {$A_{p}(X)$};
			\node (4) at(-1,-1.5) {$A_{\vec{p}}(\R^n)$};
			\node (5) at(-1,-3.5) {$A_{p(\cdot)}(\R^n)$};
			\node (6) at(2,5) {$A_{\vec{p},q}(\R^n)$};
			\node (7) at(2,3) {$A_{{p}, q}(X)$};
			\node (8) at(2,1) {$A_{\vec{p}}(X)$};
            \node (9) at(2,-1) {$A_{\vec{p}(\cdot)}(\R^n)$};
            \node (10) at(2,-3) {$A_{{p}(\cdot)}(X)$};
            \node (11) at(2,-5) {$A_{{p}(\cdot), q(\cdot)}(\R^n)$};
            \node (12) at(8,0) {$A_{\vec{p}(\cdot), q(\cdot)}(X)$};
            \node (13) at(5,3.5) {$A_{\vec{p}(\cdot), q(\cdot)}(\R^n)$};
            \node (14) at(5,1.5) {$A_{\vec{p},q}(X)$};
            \node (15) at(5,-1.5) {$A_{\vec{p}(\cdot)}(X)$};
            \node (16) at(5,-3.5) {$A_{{p}(\cdot), q(\cdot)}(X)$};
			\draw[->] (1)--(2);
			\draw[->] (1)--(3);
			\draw[->] (1)--(4);
			\draw[->] (1)--(5);
			\draw[->] (2)--(6);
			\draw[->] (2)--(7);
            \draw[->] (2)--(11);
			\draw[->] (3)--(7);
            \draw[->] (3)--(8);
            \draw[->] (3)--(10);
			\draw[->] (4)--(6);
            \draw[->] (4)--(8);
            \draw[->] (4)--(9);
			\draw[->] (4)--(7);
			\draw[->] (5)--(9);
            \draw[->] (5)--(10);
            \draw[->] (5)--(11);
			\draw[->] (13)--(12);
            \draw[->] (14)--(12);
            \draw[->] (15)--(12);
            \draw[->] (16)--(12);
            \draw[->] (6)--(13);
            \draw[->] (6)--(14);
            \draw[->] (7)--(14);
            \draw[->] (7)--(16);
            \draw[->] (8)--(14);
            \draw[->] (8)--(15);
            \draw[->] (9)--(13);
            \draw[->] (9)--(15);
            \draw[->] (10)--(15);
            \draw[->] (10)--(16);
            \draw[->] (11)--(13);
            \draw[->] (11)--(16);
		\end{tikzpicture}
		\end{center}
\caption{The relationships between weights}\label{figure1}
	\end{figure}
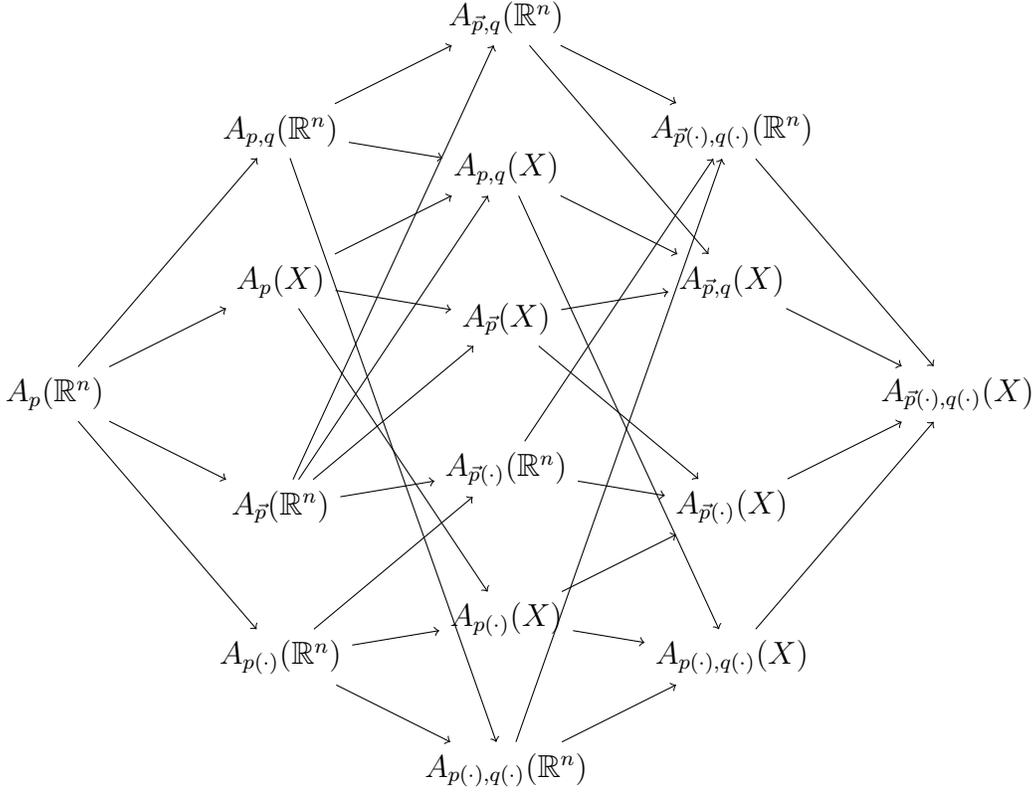
 
\begin{remark}
The interesting about the definition of $A_{\vec p(\cdot),q(\cdot)}(X)$ condition is that it is very broad. It can be translated into the well-known and studied classical weight conditions. Therefore, it is necessary to illustrate Fig.\ref{figure1}.

We first introduce the types of weights over the spaces of homogeneous type as follows.

    $(1)~$ If $\eta = 0$, then ${A_{\vec p( \cdot ),q( \cdot )}(X)} = A_{\vec p(\cdot)}(X)$.
    
    $(2)~$ If $\vec p( \cdot ) \equiv \vec p$ and $q( \cdot ) \equiv q$, then ${A_{\vec p( \cdot ),q( \cdot )}(X)} = {A_{\vec p,q}(X)}$.
    
    $(3)~$ If $\eta = 0$, $\vec p( \cdot ) \equiv \vec p$ and $q( \cdot ) \equiv q$, then ${A_{\vec p( \cdot ),q( \cdot )}(X)} = {A_{\vec p}(X)}$.
    
    $(4)~$ If $m=1$, then ${A_{\vec p( \cdot ),q( \cdot )}(X)} = A_{p(\cdot), q(\cdot)}(X)$ introduced in  \cite{Cen2024}.
    
    $(5)~$ If $m=1$, $\eta = 0$, then ${A_{\vec p( \cdot ),q( \cdot )}(X)} = {A_{p( \cdot )}(X)}$ introduced in \cite{Cruz2022}.
    
    $(6)~$ If $m=1$, $p(\cdot)\equiv p$, then $A_{p(\cdot), q(\cdot)}(X)=A_{p,q}(X)$.
    
    $(7)~$ If $m=1$, $p(\cdot)\equiv p$ and $\eta=0$, then $A_{p(\cdot), q(\cdot)}(X)=A_{p}(X)$.

\vspace{0.5cm}
When $X=\rn$, we can back to the following classical weights.

$(1')~$ If $X=\rn$, then ${A_{\vec p( \cdot ),q( \cdot )}(X)} = {A_{\vec p( \cdot ),q( \cdot )}(\rn)}$.

$(2')~$ If $X=\rn$, under the condition of $(1)$, then ${A_{\vec p( \cdot ),q( \cdot )}(X)} = A_{\vec p(\cdot)}(\rn)$ introduced in \cite{Cruz2020}.
    
    $(3')~$ If $X=\rn$, under the condition of $(2)$, then ${A_{\vec p( \cdot ),q( \cdot )}(X)} = {A_{\vec p,q}(\rn)}$ introduced in \cite{Moen2009}.
    
    $(4')~$ If $X=\rn$, under the condition of $(3)$, then ${A_{\vec p( \cdot ),q( \cdot )}(X)} = {A_{\vec p}(\rn)}$ introduced in \cite{Perez2009}.
    
    $(5')~$ If $X=\rn$, under the condition of $(4)$, then ${A_{\vec p( \cdot ),q( \cdot )}(X)} = A_{p(\cdot), q(\cdot)}(\rn)$ introduced in \cite{Ber2014}.
    
    $(6')~$ If $X=\rn$, under the condition of $(5)$, then ${A_{\vec p( \cdot ),q( \cdot )}(X)} = {A_{p( \cdot )}(\rn)}$ introduced in \cite{Cruz2011}.
    
    $(7')~$ If $X=\rn$, under the condition of $(6)$, then $A_{\vec{p}(\cdot), q(\cdot)}(X)=A_{p,q}(\rn)$ introduce in \cite{Muc1974}. 
    
    $(8')~$ If $X=\rn$, under the condition of $(7)$, then $A_{\vec{p}(\cdot), q(\cdot)}(X)=A_{p}(\rn)$ introduce in \cite{Muc1972}.
    
\end{remark}

We give two Lemmas to characterize ${A_{\vec p( \cdot ),q( \cdot )}}(X)$ as follows.

\begin{lemma}[\cite{Cen2024}, Lemma 2.14]\label{Apq_Ainfty}
    Let $p(\cdot) \in LH \cap \mathscr{P}_1$, $\frac{1}{p(\cdot)}-\frac{1}{q(\cdot)}=\eta \in [0,1)$, and $\omega \in A_{p(\cdot), q(\cdot)}(X)$. Then $u(\cdot):=$ $\omega(\cdot)^{q(\cdot)} \in A_{\infty}(X)$.
        \end{lemma}

\begin{lemma}\label{vweight4}
	Let  $p_i(\cdot) \in \P$, $i=1,\cdots,m$, with $\frac{1}{{p( \cdot )}} = \sum\limits_{i = 1}^m {\frac{1}{{{p_i}( \cdot )}}}$, and $\eta{\rm{ = }}\frac{1}{{{p}( \cdot )}}{\rm{ - }}\frac{1}{{{q}( \cdot )}} \in [0,m)$. If $\vec{\omega} \in A_{\vec{p}(\cdot),q(\cdot)}(X)$ with $\omega :=\prod\limits_{i = 1}^m {{\omega _i}}$, we have
	\begin{equation}
	\left\{
	\begin{aligned}
	&\omega _i^{-\frac{1}{m}} \in {A_{m{p_i}^\prime ( \cdot )}}(X), i=1,\dots,m,\\
	&\omega^{\frac{1}{m}}\in {A_{mq( \cdot )}}(X).
	\end{aligned}\right.
	\end{equation}
\end{lemma}
\begin{proof}[Proof:]
	Note that $$\frac{1}{{mp( \cdot )}} + \sum\limits_{i \ne j} {\frac{1}{{m{p_j}^\prime ( \cdot )}}}  = \frac{1}{{(m{p_j}^\prime )'( \cdot )}}.$$
	By H\"older's inequality, we obtain
	\begin{align*}
	{\left\| {\omega _j^{ - \frac{1}{m}}{\chi _B}} \right\|_{{L^{m{p_j}^\prime ( \cdot )}}}}{\left\| {\omega _j^{\frac{1}{m}}{\chi _B}} \right\|_{{L^{(m{p_j}^\prime )'( \cdot )}}}}
	=& {\left\| {\omega _j^{ - \frac{1}{m}}{\chi _B}} \right\|_{{L^{m{p_j}^\prime ( \cdot )}}}}{\left\| {{\omega ^{\frac{1}{m}}}\prod\limits_{i \ne j} {\omega _j^{ - \frac{1}{m}}} {\chi _B}} \right\|_{{L^{(m{p_j}^\prime )'( \cdot )}}}}\\
	\lesssim& {\left\| {{\omega ^{\frac{1}{m}}}{\chi _B}} \right\|_{{L^{mp( \cdot )}}}}\prod\limits_{i = 1}^m {{{\left\| {\omega _i^{ - \frac{1}{m}}{\chi _B}} \right\|}_{{L^{m{p_i}^\prime ( \cdot )}}}}}\\
	\lesssim& \mu(B){\left( {{{\mu(B)}^{\eta - m}}\left\| {\omega {\chi _B}} \right\|_{{L^{q( \cdot )}}}^{}\prod\limits_{i = 1}^m {\left\| {\omega _i^{ - 1}{\chi _B}} \right\|_{{L^{{p_i}^\prime ( \cdot )}}}^{}} } \right)^{\frac{1}{m}}}.
	\end{align*}
	Thus, we have
\begin{align}
\left[ {\omega _j^{ - \frac{1}{m}}} \right]_{{A_{m{p_j}^\prime ( \cdot )}}}^m \lesssim {\left[ {\vec \omega } \right]_{{A_{\vec p( \cdot ),q( \cdot )}}}}.
\end{align}
	
	We also have the following estimate.
	\begin{align*}
	{\frac{1}{\mu(B)}}{\left\| {{\omega ^{\frac{1}{m}}}{\chi _B}} \right\|_{{L^{mq( \cdot )}}}}{\left\| {{\omega ^{ - \frac{1}{m}}}{\chi _B}} \right\|_{{L^{(mq)'( \cdot )}}}} \lesssim& {\frac{1}{\mu(B)}}{\left\| {{\omega ^{\frac{1}{m}}}{\chi _B}} \right\|_{{L^{mq( \cdot )}}}}\prod\limits_{i = 1}^m {{{\left\| {{\omega_i ^{ - \frac{1}{m}}}{\chi _B}} \right\|}_{{L^{m{q_i}^\prime ( \cdot )}}}}}\\
	\le& {\frac{1}{\mu(B)}}{\left( {{{\left\| {\omega {\chi _B}} \right\|}_{{L^{q( \cdot )}}}}\prod\limits_{i = 1}^m {{{\left\| {{\omega_i ^{ - 1}}{\chi _B}} \right\|}_{{L^{{q_i}^\prime ( \cdot )}}}}} } \right)^{\frac{1}{m}}}\\
	\lesssim& {\left( {{{\mu(B)}^{\eta - m}}{{\left\| {\omega {\chi _B}} \right\|}_{{L^{q( \cdot )}}}}\prod\limits_{i = 1}^m {{{\left\| {{\omega_i ^{ - 1}}{\chi _B}} \right\|}_{{L^{{p_i}^\prime ( \cdot )}}}}} } \right)^{\frac{1}{m}}},\\
	\end{align*}
	where $\frac{1}{{{q_i}( \cdot )}}: = \frac{1}{{{p_i}( \cdot )}} - \frac{\eta}{{m}}$, $i=1,\cdots,m$.
	
	Therefore, we also get
	$$\left[ {{\omega ^{\frac{1}{m}}}} \right]_{{A_{mq( \cdot )}}}^m \lesssim {\left[ {\vec \omega } \right]_{{A_{\vec p( \cdot ),q( \cdot )}}}}.\qedhere$$
\end{proof}

The following lemma follows immediately from Lemma \ref{vweight4} and Lemma \ref{Apq_Ainfty}.

\begin{lemma}\label{Ap_Ainfty_}
    Let \( p_i(\cdot) \in \P \) for \( i = 1, \ldots, m \). Assume that $\frac{1}{{p( \cdot )}} = \sum\limits_{i = 1}^m {\frac{1}{{{p_i}( \cdot )}}}$ and $\eta = \frac{1}{p(\cdot)} - \frac{1}{q(\cdot)} \in [0,m)$. If $\vec{\omega} \in A_{\vec{p}(\cdot), q(\cdot)}(X)$, then both $u(\cdot):= \omega(\cdot)^{q(\cdot)}$ and $\sigma_j(\cdot) :=\omega_j(\cdot)^{-p_j'(\cdot)}$, for any $j=1,\ldots,m$, belong to $A_{\infty}(X)$.
\end{lemma}

\subsection{Motivations and Main Results}\label{sec2}

\subsubsection{\bf New variable multiple weight conditions}\label{se}
~~

We next present some historical sources of our motivations.

Since 1972 and 1974, Muckenhoupt et al. \cite{Muc1972,Muc1974} studied the characterizations of $A_{p}(\rn)$ and $A_{p,q}(\rn)$ by maximal operators $M$ and fractional maximal operators $M_{\alpha}$ respectively. Many people began to focus on the relationship between the characterization of weights and maximal operators.

In 2009, Lerner, Ombrosi, P\'{e}rez, Torres, and Trujillo-Gonz\'{a}lez firstly established the following characterizations for multilinear maximal operators ${\mathscr M}$.

\hspace{-20pt}{\bf Theorem $A_1$} (\cite{Perez2009}).\label{FRLB1} {\it\
Let $1 < p_1, \ldots, p_m < \infty$,
$\frac{1}{p} = \frac{1}{p_1} + \cdots + \frac{1}{p_m}$, and $\vec \omega$ is a multiple weight with $\omega :=\prod\limits_{i = 1}^m {{\omega _i}}$. Then ${{{\mathscr M}}}$ is bounded from ${L^{{p_1}}}({\omega _1}^{p_1}) \times  \cdots  \times {L^{{p_m}}}({\omega _m}^{p_m})$ to ${L^{p}}({\omega}^{p})$ if and only if $\vec \omega  \in {A_{\vec p}}(\rn)$.
}

\hspace{-20pt}{\bf Theorem $A_2$} (\cite{Perez2009}).\label{FRLB1} {\it\
Let $1 \le p_1, \ldots, p_m < \infty$,
$\frac{1}{p} = \frac{1}{p_1} + \cdots + \frac{1}{p_m}$, and $\vec \omega$ is a multiple weight with $\omega :=\prod\limits_{i = 1}^m {{\omega _i}}$. Then ${{{\mathscr M}}}$ is bounded from ${L^{{p_1}}}({\omega _1}^{p_1}) \times  \cdots  \times {L^{{p_m}}}({\omega _m}^{p_m})$ to ${WL^{p}}({\omega}^{p})$ if and only if $\vec \omega  \in {A_{\vec p}}(\rn)$.
}

Moen in the same year given the following result for multilinear fractional maximal operators ${\mathscr M}_\alpha$.

\hspace{-20pt}{\bf Theorem $A_3$} (\cite{Moen2009}).\label{FRLB1} {\it\
Let $1 < p_1, \ldots, p_m < \infty$,
$\frac{1}{p} = \frac{1}{p_1} + \cdots + \frac{1}{p_m}$, $ \frac{1}{p} -\frac{1}{q} = \frac{\alpha}{n} \in (0,m)$, and $\vec \omega$ is a multiple weight with $\omega :=\prod\limits_{i = 1}^m {{\omega _i}}$. Then ${{{\mathscr M}_{\alpha }}}$ is bounded from ${L^{{p_1}}}({\omega _1}^{p_1}) \times  \cdots  \times {L^{{p_m}}}({\omega _m}^{p_m})$ to ${L^{q}}({\omega}^{q})$ if and only if $\vec \omega  \in {A_{\vec p,q}}(\rn)$.
}

In 2012, Cruz-Uribe, Fiorenza, and Neugebauer \cite{Cruz2012} firstly studied the characterization of $A_{p(\cdot)}(\rn)$ by maximal operators $M$.

\, \hspace{-20pt}{\bf Theorem $A_4$}(\cite{Cruz2012}). {\it\
	Let $p(\cdot) \in  L H \cap \P$ and $\omega$ is a weight. Then $M$ is bounded on $L^{p(\cdot)}(\omega)$ if and only if $\omega \in A_{p(\cdot)}(\rn)$.
}

\, \hspace{-20pt}{\bf Theorem $A_5$}(\cite{Cruz2012}). {\it\
	Let $p(\cdot) \in  L H \cap \mathscr{P}_1$ and $\omega$ is a weight. Then
	$M$ is bounded from $L^{p(\cdot)}(\omega)$ to $WL^{p(\cdot)}(\omega)$
	if and only if $\omega \in A_{p(\cdot)}(\rn)$.
}

In 2014, Bernardis, Dalmasso, and Pradolini \cite{Ber2014} proved the characterizations for $A_{p(\cdot),q(\cdot)}(\rn)$ by fractional maximal operators $M_{\alpha}$ as follows.

\, \hspace{-20pt}{\bf Theorem $A_6$}(\cite{Ber2014}).{\it\
Let $p(\cdot), q(\cdot) \in  L H \cap \mathscr{P}$, $ \frac{1}{p(\cdot)}-\frac{1}{q(\cdot)}=\frac{\alpha }{n} \in[0,1)$, and $\omega$ is a weight. Then $M_{\alpha}$ is bounded from $L^{p(\cdot)}(\omega)$ to $L^{q(\cdot)}(\omega)$ if and only if $\omega \in A_{p(\cdot), q(\cdot)}(\rn)$.
}

In 2018, Cruz-Uribe and Shukla \cite{Cruz2018} obtained the following results, which solve the problem of boundedness of fractional maximal operators on variable Lebesgue spaces over the spaces of homogeneous type.

\, \hspace{-20pt}{\bf Theorem $A_7$}(\cite{Cruz2018}).{\it\
Let $p(\cdot), q(\cdot) \in   L H \cap \mathscr{P}$ and $ \frac{1}{p(\cdot)}-\frac{1}{q(\cdot)}=\eta \in[0,1)$ Then $M_\eta$ is bounded from $L^{p(\cdot)}(X)$ to $L^{q(\cdot)}(X)$.

Additionally, if \(\mu(X) < +\infty\), the requirement \( p(\cdot) \in L H\) can be substituted with \(p(\cdot) \in L H_0\).
}

\, \hspace{-20pt}{\bf Theorem $A_8$}(\cite{Cruz2018}). {\it\
Let $p(\cdot), q(\cdot) \in  L H \cap \mathscr{P}_1$ and $ \frac{1}{p(\cdot)}-\frac{1}{q(\cdot)}=\eta \in[0,1)$ Then $M_\eta$ is bounded from $L^{p(\cdot)}(X)$ to $WL^{q(\cdot)}(X)$.

Moreover, if \(\mu(X) < +\infty\), the condition \( p(\cdot) \in L H\) may be substituted with \(p(\cdot) \in L H_0\).
}

In 2020, Cruz-Uribe and Guzmán \cite{Cruz2020} established the characterization for $\M$ by ${A_{\vec p( \cdot )}}(\rn)$.

\, \hspace{-20pt}{\bf Theorem $A_9$} (\cite{Cruz2020}).\label{m-M} 
{\it\
Let $p_i( \cdot ) \in LH\cap \P$, $i = 1, \cdots ,m$, with $\frac{1}{{p( \cdot )}} = \sum\limits_{i = 1}^m {\frac{1}{{{p_i}( \cdot )}}}$, and $\vec{\omega}=(\omega_1,\ldots,\omega_m)$ is a multiple weight with $\omega:= \prod\limits_{i = 1}^m {{\omega _i}}$. Then $\M$ is bounded from ${L^{{p_1}( \cdot )}}({\omega _1}) \times  \cdots  \times {L^{{p_m}( \cdot )}}({\omega _m})$ to ${L^{p( \cdot )}}(\omega)$ if and only if $\vec \omega  \in {A_{\vec p( \cdot )}}(\rn)$.
}

In 2022, Cruz-Uribe and Cummings \cite{Cruz2022} demonstrated the following characterizations for $A_{p(\cdot)}(X)$ by maximal operators $M$.

\, \hspace{-20pt}{\bf Theorem $A_{10}$}(\cite{Cruz2022}).\label{VApp1} {\it\
Let $p(\cdot) \in L H \cap \P$ and $\omega$ is a weight. Then $M$ is bounded on $L^{p(\cdot)}(X, \omega)$ if and only if $\omega \in A_{p(\cdot)}(X)$.
}

\, \hspace{-20pt}{\bf Theorem $A_{11}$}(\cite{Cruz2022}).\label{VApp2} {\it\ Let $p(\cdot) \in   L H \cap \mathscr{P}$ and $\omega$ is a weight. Then
$M$ is bounded from $L^{p(\cdot)}(X, \omega)$ to $WL^{p(\cdot)}(X, \omega)$
if and only if $\omega \in A_{p(\cdot)}(X)$.
}

In April 2024, we \cite{Cen2024} proved the following results which actually promote our this study.

\, \hspace{-20pt}{\bf Theorem $A_{12}$}(\cite{Cen2024}).\label{cen1} {\it\ 
Let $p(\cdot) \in L H \cap \P$, $ \frac{1}{p(\cdot)}-\frac{1}{q(\cdot)}=\eta \in[0,1)$, and $\omega$ is a weight. Then $M_\eta$ is bounded from $L^{p(\cdot)}(X, \omega)$ to $L^{q(\cdot)}(X, \omega)$ if and only if $\omega \in A_{p(\cdot), q(\cdot)}(X)$.}

\, \hspace{-20pt}{\bf Theorem $A_{13}$}(\cite{Cen2024}).\label{cen2} {\it\ 
Let $p(\cdot) \in L H \cap \mathscr{P}$, $ \frac{1}{p(\cdot)}-\frac{1}{q(\cdot)}=\eta \in[0,1)$, and $\omega$ is a weight. Then
$M_\eta$ is bounded from $L^{p(\cdot)}(X, \omega)$ to $WL^{q(\cdot)}(X, \omega)$
if and only if $\omega \in A_{p(\cdot), q(\cdot)}(X)$.}

Naturally, we ask such a question.
\begin{que}
Whether the new condition $A_{\vec{p}(\cdot),q(\cdot)}(X)$ can be characterized by multilinear fractional maximal operator $\M_\eta$?
\end{que}

The answer to this question is affirmative. Let us first consider quantitative variable weighted estimates for multilinear fractional averaging operators ${{\A_{\eta,B}}}$ under the new condition $A_{\vec{p}(\cdot),q(\cdot)}(X)$.

\begin{theorem}\label{Cha.Apq}
Let $p_i(\cdot)\in \P, i=1,\ldots,m$ with $\frac{1}{{p( \cdot )}} = \sum\limits_{i = 1}^m {\frac{1}{{{p_i}( \cdot )}}}$, $\frac{1}{{{p}( \cdot )}}-\frac{1}{{{q}( \cdot )}}= \eta\in [0,m)$, and $\vec{\omega}$ is a multiple weight, then  ${{\A_{\eta,B}}}$ is bounded from ${L^{{p_1}( \cdot )}}(X,{\omega _1}) \times  \cdots  \times {L^{{p_m}( \cdot )}}({X,\omega _m})$ to ${L^{q( \cdot )}}(X,\omega)$ uniformly for all balls $B$ if and only if $\vec\omega  \in {A_{\vec p( \cdot ),q( \cdot )}(X)}$.
	Furthermore, we have that
	\begin{align*}
	{\left[ {\vec \omega } \right]_{{A_{\vec p( \cdot ),q( \cdot )}}(X)}}{ \lesssim _{\vec p(\cdot)}} \mathop {\left\| {{\A_{\eta,B}}} \right\|_{\prod\limits_{i = 1}^m {{L^{p_i( \cdot )}}({X,\omega _i})}  \to {L^{q( \cdot )}}(X,\omega )}} \le {\left[ {\vec \omega } \right]_{{A_{\vec p( \cdot ),q( \cdot )}}(X)}},
	\end{align*}
where the implicit constant only depends on the $\vec{p}(\cdot)$.
\end{theorem}

\begin{theorem}\label{W.Cha.Apq}
Let $p_i(\cdot)\in \P, i=1,\ldots,m$ with $\frac{1}{{p( \cdot )}} = \sum\limits_{i = 1}^m {\frac{1}{{{p_i}( \cdot )}}}$, $\frac{1}{{{p}( \cdot )}}-\frac{1}{{{q}( \cdot )}}= \eta\in [0,m)$, and $\vec{\omega}$ is a multiple weight, then  ${{\A_{\eta,B}}}$ is bounded from ${L^{{p_1}( \cdot )}}(X,{\omega _1}) \times  \cdots  \times {L^{{p_m}( \cdot )}}({X,\omega _m})$ to ${WL^{q( \cdot )}}(X,\omega)$ uniformly for all balls $B$ if and only if $\vec\omega  \in {A_{\vec p( \cdot ),q( \cdot )}(X)}$.
	Furthermore, we have that
	\begin{align*}
	{\left[ {\vec \omega } \right]_{{A_{\vec p( \cdot ),q( \cdot )}}(X)}}{ \lesssim} \mathop {\left\| {{\A_{\eta,B}}} \right\|_{\prod\limits_{i = 1}^m {{L^{p_i( \cdot )}}({X,\omega _i})}  \to {WL^{q( \cdot )}}(X,\omega )}} \le {\left[ {\vec \omega } \right]_{{A_{\vec p( \cdot ),q( \cdot )}}(X)}},
	\end{align*}
where the implicit constant is independent on the $\vec{\omega}$.
\end{theorem} 

\begin{remark}
In fact, the sufficiency of Theorems \ref{Cha.Apq} and \ref{W.Cha.Apq} may be derived as follows.
    \begin{align*}
        {\left\| {{\A_{\eta ,B}}(\vec f)} \right\|_{{L^{q( \cdot )}}(X,\omega )}} \le& {\mu(B)^{ \eta- m }}{\left\| {\omega {\chi _B}} \right\|_{{{q( \cdot )}}}}\prod\limits_{i = 1}^m {{{\left\| {\omega _i^{ - 1}{\chi _B}} \right\|}_{{{{p_i}^\prime ( \cdot )}}}}{{\left\| {{f_i}} \right\|}_{{L^{{p_i}( \cdot )}}({X,\omega _i})}}} \\ \le& {\left[ {\vec \omega } \right]_{{A_{\vec p( \cdot ),q( \cdot )}}}}\prod\limits_{i = 1}^m {{{\left\| {{f_i}} \right\|}_{{L^{{p_i}( \cdot )}}(X,{\omega _i})}}}.
    \end{align*}
\end{remark}

Next, we are ready to give characterizations for $A_{\vec{p}(\cdot),q(\cdot)}(X)$ by multilinear fractional maximal operator $\M_\eta$.

\begin{theorem}\label{mainthm_0}
    Let  $p_i( \cdot ) \in LH\cap \P$, $i = 1, \ldots ,m$, with $\frac{1}{{p( \cdot )}} = \sum\limits_{i = 1}^m {\frac{1}{{{p_i}( \cdot )}}}$, $ \frac{1}{p(\cdot)}-\frac{1}{q(\cdot)}=\eta \in[0,m)$, and $\vec \omega$ is a multiple weight with $\omega :=\prod\limits_{i = 1}^m {{\omega _i}}$. Then ${{{\mathscr M}_{\eta }}}$ is bounded from ${L^{{p_1}( \cdot )}}(X,{\omega _1}) \times  \cdots  \times {L^{{p_m}( \cdot )}}({X,\omega _m})$ to ${L^{q( \cdot )}}(X,\omega)$ if and only if $\vec \omega  \in {A_{\vec p( \cdot ),q( \cdot )}}(X)$.
    
Moreover, we have that 
\[
{\left[ {\vec \omega } \right]_{{A_{\vec p( \cdot ),q( \cdot )}}(X)}} \lesssim {\left\| \mathscr{M}_\eta \right\|_{\prod\limits_{i = 1}^m {{L^{p_i( \cdot )}}({X,\omega _i})}  \to {L^{q( \cdot )}}(X,\omega )}} \le  {C_{\vec \omega ,\eta ,m,\mu ,X,\vec p( \cdot )}}.
\]
where the implicit constant only depends on the $\vec{p}(\cdot)$.
    \end{theorem}
    
\begin{theorem}\label{mainthm_1}
    Let  $p_i( \cdot ) \in LH\cap \P$, $i = 1, \ldots ,m$, with $\frac{1}{{p( \cdot )}} = \sum\limits_{i = 1}^m {\frac{1}{{{p_i}( \cdot )}}}$, $ \frac{1}{p(\cdot)}-\frac{1}{q(\cdot)}=\eta \in[0,m)$, and $\vec \omega$ is a multiple weight with $\omega :=\prod\limits_{i = 1}^m {{\omega _i}}$. Then ${{{\mathscr M}_{\eta }}}$ is bounded from ${L^{{p_1}( \cdot )}}(X,{\omega _1}) \times  \cdots  \times {L^{{p_m}( \cdot )}}(X,{\omega _m})$ to ${WL^{q( \cdot )}}(X,\omega)$ if and only if $\vec \omega  \in {A_{\vec p( \cdot ),q( \cdot )}}(X)$.

Moreover, we have that 
\[
{\left[ {\vec \omega } \right]_{{A_{\vec p( \cdot ),q( \cdot )}}(X)}} \lesssim {\left\| \mathscr{M}_\eta \right\|_{\prod\limits_{i = 1}^m {{L^{p_i( \cdot )}}({X,\omega _i})}  \to {WL^{q( \cdot )}}(X,\omega )}} \le  {C_{\vec \omega ,\eta ,m,\mu ,X,\vec p( \cdot )}}.
\]
where the implicit constant is independent on the $\vec{\omega}$.
    \end{theorem}

\begin{remark}
For Theorems \ref{mainthm_0} and \ref{mainthm_1}, if \(\mu(X) < +\infty\), due to Lemma \ref{finite-bounded}, we can relace \( L H\) with \(L H_0\).
\end{remark}

\begin{remark}
It is not hard to find that Theorems \ref{mainthm_0} and \ref{mainthm_1} generalize Theorems {\bf $A_1$--$A_{13}$}.
\end{remark}

\subsubsection{\bf New variable text conditions}\label{sec2_2}
~~

In fact, there are other ways to study the weighted boundedness of the maximal operators, which are called the text conditions.
In 1982, Sawyer \cite{S} gave a characterization of a two-weighted norm characterization for $M_\alpha$, who given the follow theorems.

\, \hspace{-20pt}{\bf Theorem $B_1$}(\cite{S}). {\it\ The fractional maximal operator $M_\alpha$ is bounded
from $L^p(\rn,u)$ to $L^q(\rn,v)$ if and only if $(u,v)$ satisfies that
\[
  [u,v]_{S_{p,q}}:=\sup_{Q \subseteq \rn}\frac{\left(\int_{Q}M_\alpha(\sigma \chi_Q)^q vdx\right)^{1/q}}{\sigma(Q)^{1/{p}}}<\infty,
\]
where $\sigma=u^{1-p'}$, $0\le\alpha<n$, $1<p<n/\alpha$ and $1/q=1/p-\alpha/n$.
}

The above inequality is known as Sawyer's test condition.

In 2009, Moen \cite{M1} improved Sawyer's result by showing that $ \|M_\alpha\|_{L^p(u)\rightarrow L^q(v)}\approx [u,v]_{S_{p,q}}.$ 

\, \hspace{-20pt}{\bf Theorem $B_2$}(\cite{M1}). {\it\ Suppose that $0 \leq \alpha<n, 1<p \leq q<\infty$, and $(u, v)$ is a pair of weights with $\sigma=v^{1-p^{\prime}}$. Then
$M_\alpha$ is bounded from ${L^p(v)}$ to ${L^q(u)}$ if and only if $(u, v)$ satisfies
$$
[u, v]_{S_{p, q}}=\sup _Q \frac{\left(\int_Q M_\alpha\left(\chi_Q \sigma\right)^q u d x\right)^{1 / q}}{\sigma(Q)^{1 / p}}<\infty,
$$
In this case, 
$$
{\left\| {{M_\alpha }} \right\|_{{L^p}(v) \to {L^q}(u)}} \lesssim [u, v]_{S_{p, q}} .
$$
}

In 2013, Chen and Damian \cite{CD} continued to develop and prove the following condition for the boundness of multilinear fractional maximal operator $\M_{\alpha}$ in the case of two weights.

We say that $\vec{\omega}$ satisfies the $R H_{\vec{p}}$ condition if there exists a positive constant $C$ such that
\begin{align}\label{2.3.3}
     \prod_{i=1}^m\left(\int_Q \sigma_i d x\right)^{\frac{p}{p_i}} \leq C \int_Q \prod_{i=1}^m \sigma_i^{\frac{p}{p_i}} d x,
\end{align}
where $\sigma_i=\omega_i^{1-p_i^{\prime}}$ for $i=1, \ldots, m$. We denote by $[\vec{\omega}]_{R H_{\vec{p}}}$ the smallest constant $C$ in \eqref{2.3.3}.

\, \hspace{-20pt}{\bf Theorem $B_3$}(\cite{CD}). {\it\ Let $1<p_i<\infty, i=1, \ldots, m$ and $\frac{1}{p}=\frac{1}{p_1}+\ldots+\frac{1}{p_m}$. Let $v$ and $\omega_i$ be weights. If we suppose that $\vec{\omega} \in R H_{\vec p}$ then there exists a positive constant $C$ such that
\begin{align}\label{3.3.3}
    \|\M({\vec f \sigma})\|_{L^p(v)} \leq C \prod_{i=1}^m\left\|f_i\right\|_{L^{p_i}\left(\sigma_i\right)}, \quad f_i \in L^{p_i}\left(\sigma_i\right),
\end{align}
where $\sigma_i=\omega_i^{1-p_i^{\prime}}$, if and only if $(\vec{\omega},v) \in S_{\vec p}$. Moreover,
$$
{[\vec \omega ,v]_{{S_{\vec p}}}} \lesssim {\left\| M \right\|_{\prod\limits_{i = 1}^m {{L^{{p_i}}}({\sigma _i}) \to {L^p}(v)} }} \lesssim {[\vec \omega, v]_{{S_{\vec p}}}}[\vec \omega ]_{R{H_{\vec p}}}^{\frac{1}{p}}.
$$
where $[\vec{\omega},v]_{S_{\vec p}}: =\sup _Q\left(\int_Q \M \left({\vec \sigma \chi_Q}\right)^p v d x\right)^{\frac{1}{p}}\left(\prod_{i=1}^m \sigma_i(Q)^{\frac{1}{p_i}}\right)^{-1}<\infty$.
}

Li and Sun \cite{Lik} in same year built the following multilinear Sawyer's text condition for $\M_\alpha$.

\, \hspace{-20pt}{\bf Theorem $B_4$}(\cite{Lik}). {\it\ Suppose that $0\le \alpha<mn$, that $1<p_1,\cdots,p_m<\infty$, that $1/{p}=1/{p_1}+\cdots+1/{p_m}$, that $1/q=1/p-\alpha/n$ and that $q\ge \max_i\{p_i\}$.
Let $\omega_1$, $\cdots$, $\omega_m$, $v$ be weights and set $\sigma_i=\omega_i^{1-p_i'}$, $i=1,\cdots,m$. Define
\[
  [\vec{\omega},v]_{S_{\vec{p},q}}:=\sup_{Q \subseteq \rn}\frac{\left(\int_{Q}\M_\alpha(\sigma_1 1_Q,\cdots,\sigma_m 1_Q)^q vdx\right)^{1/q}}{\prod_{i=1}^m \sigma_i(Q)^{1/{p_i}}}.
\]
Then $\M_\alpha$ is bounded from $L^{p_1}(\omega_1)\times\cdots\times  L^{p_m}(\omega_m)$ to $L^q(v)$
if and only if $   (\vec{\omega},v) \in {S_{\vec{p},q}}$. Moreover,
\[
  \|\M_\alpha\|_{L^{p_1}(\omega_1)\times\cdots\times  L^{p_m}(\omega_m)\rightarrow L^q(v)}\approx[\vec{\omega},v]_{S_{\vec{p},q}}.
\]
}


\begin{que}
Whether we may establish corresponding variable exponent test conditions to enable the fractional maximal operator $M_\eta$ to obtain new two-weighted estimates?
\end{que}
The following new variable test condition may assist us to answer this question. 
\begin{definition}
Let $\omega$ and $v$ are both weights. We define $u(\cdot):=v(\cdot)^{q(\cdot)},\sigma(\cdot):=\omega(\cdot)^{-p'(\cdot)}$ and denote $p_1=p_{-}, p_2=p_{+}.$ We say $(\omega, v) \in C_{p(\cdot), q(\cdot)}(X)$, if

(1) ${[\omega ]_{C_{p( \cdot ),q( \cdot )}^1(X)}}: =\mathop {\sup }\limits_{Q \in \d} \sigma {(Q)^{{q_ - } - q(x)}}{\chi _Q}(x) < \infty,$

(2) $[\omega, v]_{C_{p(\cdot), q(\cdot)}^2(X)}:=\sup\limits _{Q \in \d}\left(\int_Q M_\eta^{\d}\left(\sigma x_Q\right)^{q(x)} u(x) d x\right)^{\frac{1}{q_-}} \cdot \max\limits _{t=1,2}\sigma(Q)^{-\frac{1}{p_t}}< \infty,$

where $\mathcal{D}$ represents the dyadic grid of X whose details may be found in Lemma \ref{cubes}.
We usually abbreviate ${C_{\vec p( \cdot ),q( \cdot )}}(X)$ to ${C_{\vec p( \cdot ),q( \cdot )}}$.
\end{definition}

We next give quantitative variable two-weighted estimates for $M_\eta$ to the answer this question.

\begin{theorem}\label{Text.condi.}
Let $p(\cdot), q(\cdot) \in \mathscr{P}$ with ${p_ + } \le {q_ - }$. Set $\omega,v$ are both weights. If  $(\omega, v) \in C_{p(\cdot), q(\cdot)}(X)$, then $M_\eta$ is bounded from $L^{p(\cdot)}(X, \omega)$ to $L^{q(\cdot)}(X, v)$.
In addition, we have that 
\begin{align*}
    \|M_{\eta}\|_{L^{p(\cdot)}(X,\omega)\rightarrow L^{q(\cdot)}(X,v)} \lesssim \sum\limits_{\theta  = \frac{1}{{{p_{\rm{ - }}}}},\frac{1}{{{p_{\rm{ + }}}}}} {{{\left( {{{[\omega ,v]}_{C_{p( \cdot ),q( \cdot )}^2(X)}} + {{[\omega ]}_{C_{p( \cdot ),q( \cdot )}^1(X)}}{{[\omega ,v]}_{C_{p( \cdot ),q( \cdot )}^2(X)}}} \right)}^\theta }}.
\end{align*}
where implicit constant is independent on the $(\omega,v)$. 
\end{theorem}

\begin{remark}
When $p(\cdot) \equiv p$, $q(\cdot) \equiv q$, we may remove the condition ${C_{p(\cdot),q(\cdot)}^1}(X)$, meanwhile, the condition ${C_{p(\cdot), q(\cdot)}^2(X)}$ is back to $S_{p,q}(X)$, which is the classical Sawyer's condition. It follows that 
\begin{align*}
    \|M_{\eta}\|_{L^{p}(X,\omega^p)\rightarrow L^{q}(X,v^q)} \lesssim  [\omega, v]_{S_{p,q}(X)}^{\frac{1}{p}},
\end{align*}
It is not difficult to see that ${C_{p(\cdot), q(\cdot)}}(X)$ is indeed a variable exponent extension of $S_{p,q}(X)$.

\end{remark}

\subsection{Structure and Notations}
~~

The structure of rest of the paper is as follows. In Sect. \ref{sec3}, we build and give some fundamental lemmas, which play a important roles in our proof. In Sect. \ref{proof}, we prove Theorems \ref{mainthm_0} and \ref{mainthm_1}. Finally, we prove Theorem \ref{Text.condi.} using the method of Sawyer in Sect. \ref{sec.5.}.

We need to introduce some notations for use in this paper.

For some positive constant \(C\) independent of appropriate parameters, we denote \(A \lesssim B\) to mean that \(A \leq CB\), and \(A \approx B\) to indicate that \(A \lesssim B\) and \(B \lesssim A\). Additionally, \(A \lesssim_{\alpha, \beta} B\) signifies that \(A \leq C_{\alpha,\beta} B\), where \(C_{\alpha,\beta}\) depends on \(\alpha\) and \(\beta\).
Consider an open set \(E \subseteq X\) and a measurable function \(p(\cdot): E \to [1, \infty)\). The conjugate exponent \(p'(\cdot)\) is defined by \(p'(\cdot) = \frac{p(\cdot)}{p(\cdot) - 1}\).

A weight \(\omega: X \rightarrow [0, \infty]\) is a locally integrable function that satisfies \(0 < \omega(x) < \infty\) for almost every \(x \in X\). Given a weight \(\omega\), its associated measure is defined as \(d\omega(x) = \omega(x) d\mu(x)\), and we also denote $\omega(E)=\int_E \omega(x)\,d\mu$. 
Given a set $E \subseteq X$, a function $f$, and $\sigma$ is a weight, we denote weighted averages by
\[\avf_{\sigma,E}=\frac{1}{\sigma(E)}\int_E \sigma \,d\mu.\]
If $\sigma =1$, we denote $\avf_{\sigma,E}$ by $\avf_{E}$.

\section{\bf Preliminaries}\label{sec3}

\subsection{Variable Lebesgue Spaces}\label{3.1}
~

This subsection presents foundational lemmas for variable Lebesgue spaces over the spaces of homogeneous type, which are crucial for supporting our main results. Some lemma proofs are the same as in the Euclidean case, with no essential technical differences. Readers can refer to \cite{red,Cruz2022,Di2011} for more details.

 \begin{lemma}[\cite{red}, Theorem 2.61]\label{lemma:fatou}
Given $\pp\in \P_0$, if $f\in L^{p(\cdot)}$ is such that $\{f_k\}$ convergesto $f$ pointwise a.e., then
	\[ 
	\|f\|_{\pp} \leq \liminf_{k\rightarrow \infty} \|f_k\|_{\pp}. \]
\end{lemma}

\begin{lemma}[\cite{red}, Theorem 2.59]\label{Levi}
	Let $p(\cdot) \in \P_1$. For a sequence of non-negative measureable functions, denoted as $\left\{f_k\right\}_{k=1}^{\infty}$ and increasing pointwise almost everywhere to a function $f \in L^{p(\cdot)}$, we can deduce that $\left\|f_k\right\|_{p(\cdot)} \rightarrow\|f\|_{p(\cdot)}$.
\end{lemma}

\begin{lemma}[\cite{red}, Theorem 2.34]\label{lemma:dual}
Let $\Omega \subseteq X$, $\pp \in \P$, then for every $f\in L^{p(\cdot)}$,
\[ \|f\|_{\pp} \approx \sup_{\|g\|_{\cpp} \leq 1}\int_{\Omega} |fg|\,dx, \]
where the implicit constants depend only on $\pp$.
\end{lemma}

\begin{lemma}[\cite{red}, Theorem 2.26]\label{Holder}
Let $p(\cdot) \in \P_1$, then
$$
\int_X|f(x) g(x)| d \mu \leq 4\|f\|_{p(\cdot)}\|g\|_{p^{\prime}(\cdot)}.
$$
\end{lemma}

\begin{lemma}[\cite{red}, Corollary 2.28.]\label{Gu6}
	Let $p( \cdot ), p_i( \cdot ) $ belong to $ {\P}_0$, $ i = 1, \ldots ,m$, and $\frac{1}{{p(\cdot)}} = \sum\limits_{i = 1}^m {\frac{1}{{{p_i}(\cdot)}}}$. Then
	\begin{equation*}
		{\left\| {{f_1} \cdots {f_m}} \right\|_{{L^{p( \cdot )}(X)}}} \lesssim \prod\limits_{i = 1}^m {{{\left\| {{f_i}} \right\|}_{{L^{{p_i}( \cdot )}(X)}}}}.
	\end{equation*}
\end{lemma}

\begin{lemma}[\cite{Cruz2022}, Lemma 2.1]\label{LMB.}
    For all $0 < r < R$ and any $y \in B(x, R)$, there exists a positive constant $C=C_{X}$, such that
    $$
    \frac{\mu(B(y, r))}{\mu(B(x, R))} \geq C\left(\frac{r}{R}\right)^{\log _2 C_\mu}.
    $$
    \end{lemma}
\begin{lemma}[\cite{Bra1996}, Lemma 1.9]\label{finite-bounded}
    $\mu(X) < \infty$ if and only if $X$ is bounded, which means there exist a ball $B \subseteq X$, such that $X=B$.
\end{lemma}
    
\begin{lemma}[\cite{Cruz2022}, Lemma 2.11]\label{lemma:diening}
    Let $p(\cdot) \in L H$, then 
    $
    \mathop {\sup }\limits_{B \subseteq X} \mu {(B)^{{p_ - }(B) - {p_ + }(B)}} \lesssim 1.
    $
The implict constant depends only on $p(\cdot)$ and $n$. The same inequality holds if we replace one of $p_+(Q)$ or $p_-(Q)$ by $p(x)$ for any $x \in Q$.
\end{lemma}
  
    
\begin{lemma}[\cite{red}, Proposition 2.21]\label{unity}
Let $p(\cdot) \in \P_1$, then
$$
\int_X\left(\frac{|f(x)|}{\|f\|_{p(\cdot)}}\right)^{p(x)} d \mu=1 .
$$
\end{lemma}
\begin{lemma}[\cite{red}, Corollary 2.23] \label{p.omega}
Let $\Omega \subseteq X$ and $p(\cdot) \in \P_1 (\Omega)$. 
     
 If $\|f\|_{L^{p(\cdot)}(\Omega)} \leq 1$, then
    $$
    \|f\|_{p(\cdot)}^{p_{+}(\Omega)} \leq \int_\Omega|f(x)|^{p(x)} d \mu \leq\|f\|_{p(\cdot)}^{p_{-}(\Omega)} .
    $$
     
    If $\|f\|_{L^{p(\cdot)}(\Omega)} \geq 1$, then
    $$
    \|f\|_{p(\cdot)}^{p_{-}(\Omega)} \leq \int_\Omega|f(x)|^{p(x)} d \mu \leq\|f\|_{p(\cdot)}^{p_{+}(\Omega)} .
    $$
    Moverover, we have $\|f\|_{p(\cdot)} \leq $ if and only if
    $\int_\Omega|f(x)|^{p(x)} d \mu \leq C_2$.
    When either $=1$ or $C_2=1$, the other constant is also to be 1.
\end{lemma}
    

    
\begin{lemma}[\cite{Cruz2012}] \label{lemma:infty-bound}
Let $\pp \in LH \cap \P_1$, if $\omega\in A_\pp(X)$, then 
	there exists a constant $t>1$, depending only on $\omega$ and $\pp$,
	such that
	\[ \int_X \frac{\omega(x)^{p(x)}}{\left(e+d\left(x_0, x\right)\right)^{t p_{-}}} d \mu \leq 1 \]
\end{lemma}

\begin{lemma}[\cite{Cruz2022}, Lemma 2.10]\label{lemma:p-infty-px}
    For any point $y\in G$, $G$ is a subset of $X$, and two exponents $p_1(\cdot)$ and $p_2(\cdot)$, there exists a constant $C_0>0$ such that
    $$
    |p_1(y)-p_2(y)| \leq \frac{C_0}{\log \left(e+d\left(x_0, y\right)\right)}.
    $$
    Then there exists a constant $C=C_{t,C_0}$ such that 
    \begin{equation}\label{2Log_1}
    \int_G|f(y)|^{p_1(y)} u(y) d \mu \leq C \int_G|f(y)|^{p_2(y)} u(y) d \mu+\int_G \frac{1}{\left(e+d\left(x_0, y\right)\right)^{t s_{-}(G)}} u(y) d \mu,
    \end{equation}
    for all functions $f$ with $|f(y)| \leq 1$ and every $t\geq 1$.
    \end{lemma}

\subsection{Properties of Weights}\label{3.2}
~ 

This subsection is aimed to exploring the properties of the variable multiple weights condition within spaces of homogeneous type. 
The following lemma reflects the properties of $A_{\infty}(X)$, defined by $\bigcup_{p \geq 1} A_p(X)$, whose proof are similar to that of \cite[Theorem 7.3.3]{249}.

\begin{lemma}\label{Ainfty}
Let $\omega$ is a weight, then the following conditions are equivalent:
    \begin{enumerate}
        \item $\omega \in A_{\infty}(X)$.
        \item There exist constants $\epsilon>0$ and $C_2>1$ such that 
        $$
        \frac{\mu(E)}{\mu(B)} \leq C_2\left(\frac{\omega(E)}{\omega(B)}\right)^\epsilon,
        $$
        for any ball $B$ and 
        its measurable subset $E$.
        \item The measure $d\omega(x)=\omega(x) d \mu(x)$ satisfies doubling condition and there exist constants $\delta>0$ and $C_1>1$ such that 
        $$
        \frac{\omega(E)}{\omega(B)} \leq C_1\left(\frac{\mu(E)}{\mu(B)}\right)^\delta,
        $$
        for any ball $B$ and its  measurable subset $E$.
    \end{enumerate}
    \end{lemma}

\begin{lemma}[\cite{Cruz2022}, Lemma 3.3]\label{fracexp}
    Let $p(\cdot) \in LH\cap \mathscr{P}_1$ and $\omega \in A_{p(\cdot)}(X)$. Then
    \begin{align*}
            \mathop {\sup }\limits_{B \subseteq X} \left\| {\omega {\chi _B}} \right\|_{p( \cdot )}^{{p_ - }(B) - {p_ + }(B)} & \lesssim (1+[\omega]_{A_{p(\cdot)}})^{p_+ - p_-}.
    \end{align*}
The same inequality holds if we replace one of $p_+(Q)$ or $p_-(Q)$ by $p(x)$ for any $x \in Q$.
\end{lemma}

\begin{lemma}[\cite{Cen2024}, Lemma 2.14.]\label{lemma:p-infty-cond}
Given $\pp \in LH \cap \P_1$.  Let $\omega\in A_\pp(X)$ and let
    	$u(x)=\omega(x)^{p(x)}$. For any cube $Q$ with
	$\|\omega\chi_Q\|_\pp \geq 1$, then
	$\|\omega\chi_Q\|_\pp \approx u(Q)^{\frac{1}{p_\infty}}$.  Moreover,
	given any $E \subseteq Q$,
    \begin{equation*}
        \left(\frac{\mu(E)}{\mu(B)}\right) \lesssim  [\omega]_{A_{p(\cdot)}} (1+[\omega]_{A_{p(\cdot)}})^{\frac{p_+ - p_-}{p_-} }\left(\frac{u(E)}{u(B)}\right)^{1 / p_{+}},
    \end{equation*}
where the implicit constant is independent on $\omega$.
\end{lemma}

\subsection{Dyadic Analysis}\label{3.3}
~

The proof of our main theorem relies heavily on dyadic cubes of spaces of homogeneous type, for its classical case in Euclidean space, as follows.
$$
Q = {2^k}([0,1)^n + m), \quad k\in \mathbb{Z}, m\in \Z^n.
$$
The following discussion adopts the framework of dyadic cubes as formulated by Hytönen and Kairema \cite{Hy2012}, as explicated in \cite{And2015}.
\begin{lemma}[\cite{And2015}, Theorem 2.1]\label{cubes}
    There exist a family $\mathcal{D}=\bigcup_{k \in \mathbb{Z}} \mathcal{D}_k$ (called the dyadic grid of X), composed of subsets of $X$, such that:
    \begin{enumerate}
        \item For cubes $Q_1, Q_2 \in \mathcal{D}$, either $Q_1 \cap Q_2=\varnothing$, $Q_1 \subseteq Q_2$, or $Q_2 \subseteq Q_1$.
        \item The cubes $Q \in \mathcal{D}_k$ are pairwise disjoint. And for any $k\in \mathbb{Z}$, $X=\bigcup_{Q \in \mathcal{D}_k} Q$. We call $\mathcal{D}_k$ as the $k$th generation.
        
        \item For any $Q_1\in \mathcal{D}_k$, there always exists at least one $Q_2$ in $\mathcal{D}_{k+1}$ (called a child of $Q_1$), such that $Q_2 \subseteq Q_1$, and there always exists exactly one $Q_3$ in $\mathcal{D}_{k-1}$ (called the parent of $Q_1$), such that $Q_1 \subseteq Q_3$.
        \item If $Q_2$ is a child of $Q_1$, then for a constant $0<\epsilon<1$, depended on the set $X$, $\mu\left(Q_2\right) \geq \epsilon \mu\left(Q_1\right)$. 
        \item For every $k\in \mathbb{Z}$ and $Q \in \mathcal{D}_k$, there exists constants $C_d$ and $d_0>1$, such that
        $$
        B\left(x_c(Q), d_0^k\right) \subseteq Q \subseteq B\left(x_c(Q), C_d d_0^k\right),
        $$
        where $x_c(Q)$ denotes the centre of cube $Q\in \mathcal{D}$.
    \end{enumerate}
    
    We call the family $\mathcal{D}$ as dyadic grid and the cubes $Q\in \mathcal{D}$ as dyadic cubes.
\end{lemma}
We introduce a dyadic cubes version of $A_{\vec{p}(\cdot),q(\cdot)}(X)$ below.

\begin{lemma}
 \label{Apqcube}
 Let $p_i( \cdot ) \in LH\cap \P$, $i = 1, \ldots ,m$, with $\frac{1}{{p( \cdot )}} = \sum\limits_{i = 1}^m {\frac{1}{{{p_i}( \cdot )}}}$, $\eta=\frac{1}{{{p}( \cdot )}}-\frac{1}{{{q}( \cdot )}} \in [0,m)$, and $\mathcal{D}$ is a dyadic grid. If $\vec{\omega} \in A_{\vec{p}(\cdot), q(\cdot)}$, then $\vec{\omega} \in {A_{\vec{p}( \cdot ),q( \cdot )}^{\mathcal D}}$. More precisely, 
$$
{\left[ \vec{\omega}  \right]_{A_{\vec{p}( \cdot ),q( \cdot )}^{\mathcal D}(X)}}: = \mathop {\sup }\limits_{Q \in {\mathcal D}} \mu {(Q)^{\eta  - m}}{\left\| {\omega {\chi _Q}} \right\|_{q( \cdot )}}\prod_{i=1}^{m}{\left\| {{\omega_i ^{ - 1}}{\chi _Q}} \right\|_{p_i'( \cdot )}} \lesssim {\left[ \vec{\omega}  \right]_{A_{\vec{p}( \cdot ),q( \cdot )}(X)}}.
$$
\end{lemma}
\begin{proof}
Using Theorem \ref{cubes} with fixing $Q \in \mathcal{D}_k$ and Lemma \ref{LMB.},
\begin{align*}
& \left\|\omega \chi_Q\right\|_{q(\cdot)}\prod_{i=1}^{m}\left\|\omega_i^{-1} \chi_Q\right\|_{p_i^{\prime}(\cdot)} \leq\left\|\omega \chi_{B\left(x_c(Q), C_d d_0^k\right)}\right\|_{q(\cdot)} \prod_{i=1}^{m} \left\|\omega_i^{-1} \chi_{B\left(x_c(Q), C_d r d_0^k\right)}\right\|_{p_i^{\prime}(\cdot)} \\
& \le {\left[ \vec{\omega}  \right]_{A_{\vec{p}( \cdot ),q( \cdot )}(X)}} \mu\left(B\left(x_c(Q), C d_0^k\right)\right)^{m-\eta} \lesssim {\left[ \vec{\omega}  \right]_{A_{\vec{p}( \cdot ),q( \cdot )}}}\mu\left(B\left(x_c(Q), d_0^k\right)\right)^{m-\eta} \\
&\le {\left[ \vec{\omega}  \right]_{A_{\vec{p}( \cdot ),q( \cdot )}(X)}}\mu(Q)^{m-\eta} .\qedhere
\end{align*}
\end{proof}

We introduce some notation as follows, which will be used multiple times in the proof of main theorems.

Let \(q(\cdot)\) and \(p_i(\cdot)\) for \(i = 1, \ldots, m\) belong to \(\P_0\). For a set \(E \subseteq X\), we define \(\beta(E)\) by
\begin{equation*}
\frac{1}{\beta(E)}=\sum\limits_{i = 1}^m {\frac{1}{{{{({p_i})}_ - }(E)}}},
\end{equation*}
and  define $\delta$ by
\begin{align}\label{delta}
\sum\limits_{i = 1}^m {\frac{1}{{{{({p_i})}_ - }}}}  - \eta  \ge \frac{1}{{\delta (E)}}: = \frac{1}{{\beta (E)}} - \eta  \ge \frac{1}{q_ -(E)}\ge \frac{1}{q(x)},
\end{align}
where $x \in E$ and $\eta=\sum\limits_{i = 1}^m {\frac{1}{{{p_i}( \cdot )}}}-\frac{1}{{{q}( \cdot )}}>0$.

We now give the following Lemma over the balls.
\begin{lemma}\label{B-q-relation}
Let $h(\cdot) \in \P$, $p_i(\cdot) \in LH \cap \P_1$, $i=1,\cdots,m$, with $\frac{1}{{p( \cdot )}} = \sum\limits_{i = 1}^m {\frac{1}{{{p_i}( \cdot )}}}$, $\eta=\frac{1}{{{p}( \cdot )}}-\frac{1}{{{q}( \cdot )}}>0$, and $v\in A_{h(\cdot)}(X)$.
Suppose that $\delta$ is defined as \eqref{delta}. Then we have
\begin{equation}
{V_0}: = \mathop {\sup }\limits_{C > 1} \mathop {\sup }\limits_{x \in CB} \mathop {\sup }\limits_{B \subseteq X} \left\| {{v^{ - 1}}{\chi _{B}}} \right\|_{h'( \cdot )}^{\delta (CB) - q(x)} < \infty.
\end{equation}               
\end{lemma}

\begin{proof}
Without loss of generality, we assume that $\|v^{-1}\chi_B\|_{h'(\cdot)}\leq 1$ and $m=2$. 

Set $B_{0}$ be the cube centered at $x_0$ with $\mu(B_0)=1$.  Then
either $r(B)\leq r(B_0)$ or $r(B)>r(B_0)$.  We will
prove~\eqref{EQ-keyboundednorm} in the first case; the proof of the
second case is the same, exchanging the roles of $B$ and $B_0$. Suppose  that  $\dist(B,B_{0})\le r(B_{0})$. Note that
$B \subseteq 4B_{0}$. 
%
%
Then for any $x \in CB$,
\begin{equation}\label{EQ-logcontinuityof-q}
0 \le q(x)-\delta(CB)\lesssim_{{n,\alpha ,{p_1}( \cdot ),{p_1}( \cdot )}}(p_{1})_{+}(CB)-(p_{1})_{-}(CB)+(p_{2})_{+}(CB)-(p_{2})_{-}(CB).
\end{equation}
Therefore, by H\"older's inequality and $v\in A_{h(\cdot)}$,
\begin{align} \label{eqn:small}
\mu(B)\lesssim\| v^{-1}\chi_B\|_{h'(\cdot)}(\mu(4B_{0}))^{-1}\|v\chi_{4B_{0}}\|_{h(\cdot)}\le {\left[ v \right]_{{A_{h( \cdot )}}}} \|v^{-1}\chi_B\|_{h'(\cdot)}\|v^{-1}\chi_{4B_{0}}\|_{h'(\cdot)}^{-1}\mu(4B_{0}). 
\end{align}
Hence, by \eqref{EQ-logcontinuityof-q}, \eqref{eqn:small}, and Lemma~\ref{lemma:diening},
\begin{align*}
\|v^{-1}\chi_B\|_{h'(\cdot)}^{\delta(CB)-q(x)} 
&\leq [v]_{A_{h(\cdot)}}^{q(x) - \delta (CB)} \left(1+\|v^{-1}\chi_{4B_{0}}\|_{h'(\cdot)}^{-1}\right)^{q_{+}-\inf\limits_{B} \delta (CB)} \mu(B)^{\delta(CB)-q(x)} \\
&\lesssim 1.
\end{align*}

Now assume that $\dist(B,B_{0})\geqslant r(B_{0})$ and note that the fact that $r(B)\leq r(B_0)$.  Then there exists a ball
${\tilde B}={\tilde B}(x_0,4dist(B,B_0))$ such that $B,\, B_0 \subseteq {\tilde B}$ and
$ r({\tilde B})\approx \dist(B,B_{0})\approx
\dist(B,x_0)=:d_{B}$. 
It follows the doubling property of $\mu$ that $\mu (\tilde B) \lesssim \mu (B({x_0},e + {d_B}))$.

Therefore, arguing as we did in
inequality~\eqref{eqn:small}, replacing $4B_0$ with ${\tilde B}$, it follows that
   \begin{equation*}
\mu(B)\lesssim {\left[ v \right]_{{A_{h( \cdot )}}}} \mu (\tilde B)\|v^{-1}\chi_B\|_{h'(\cdot)}\|v^{-1}\chi_{\tilde B}\|_{h'(\cdot)}^{-1}.
   \end{equation*}
By the above, Lemma \ref{lemma:diening}, and the fact that
$\|v^{-1}\chi_{B_0}\|_{h'(\cdot)} \le \|v^{-1}\chi_{\tilde B}\|_{h'(\cdot)}$, we get
\begin{align}\label{es.v}
\|v^{-1}\chi_B \|_{h'(\cdot)}^{\delta(CB)-q(x)}&\lesssim
[v]_{{A_{h( \cdot )}}}^{q(x) - \delta (CB)}\mu {(B)^{\delta (CB) - q(x)}}\left\| {{v^{ - 1}}{\chi _{{B_0}}}} \right\|_{h'( \cdot )}^{\delta (CB) - q(x)}\mu {(\tilde B)^{q(x) - \delta (CB)}} \notag\\ 
&\lesssim \mu {(\tilde B)^{q(x) - \delta (CB)}} 
\lesssim \mu {(B({x_0},e + {d_B}))^{q(x) - \delta (CB)}}.
\end{align}

To estimate this final term, note that since $p_j(\cdot)\in LH$, there
exist $x_{1}, x_{2} \in \overline{CB}$ such that $(p_{1})_{-}(CB)=p_{1}(x_{1})$
and $(p_{2})_{-}(CB)=p_{2}(x_{2})$.

Moreover,
$d(x_0,x_1),d(x_0,x_2)\approx d_{B}$.  
Therefore, again by
log-H\"older continuity,
\[ 
\left| {{\frac{1}{\delta(CB)}-\frac{1}{q_{\infty}}}} \right|
\leq
\left|\frac{1}{p_1(x_1)}-\frac{1}{(p_1)_\infty}\right|
+
\left|\frac{1}{p_2(x_2)}-\frac{1}{(p_2)_\infty}\right|
\lesssim 
\frac{1}{\log(e+d_{B})}. \] 
Therefore, for $x\in CB$, since $d(x_0,x)\approx d_{B}$,
\begin{align}\label{LLH}
\left|{\frac{1}{\delta(CB)}-\frac{1}{q(x)}}\right|
\leq \left|{\frac{1}{\delta(CB)}-\frac{1}{q_{\infty}}}\right|+\left|{\frac{1}{p_{\infty}}-\frac{1}{p(x)}}\right|
\lesssim\frac{1}{\log(e+d_{B})},
\end{align}
It follows from \eqref{es.v}, \eqref{LLH}, and the doubling property of $\mu$ that
\begin{equation*}
\| v^{-1}\chi_B\|_{h(\cdot)}^{\delta(CB)-q(x)}\lesssim \mu {(B({x_0},e + {d_B}))^{q(x) - \delta (CB)}} 
\lesssim 1.\qedhere
\end{equation*}
  \end{proof}

The next lemma is for dyadic cubes that is similar to the last above and plays a important role in our proof, which is dedicated to the proof of \eqref{EQ-I12estimate}.

\begin{lemma}\label{q-relation}
Let $h(\cdot) \in \P$, $p_i(\cdot) \in LH \cap \P_1$, $i=1,\cdots,m$, with $\frac{1}{{p( \cdot )}} = \sum\limits_{i = 1}^m {\frac{1}{{{p_i}( \cdot )}}}$, $\eta=\frac{1}{{{p}( \cdot )}}-\frac{1}{{{q}( \cdot )}}>0$, and $v\in A_{h(\cdot)}(X)$.
Suppose that $\mathcal{D}$ is a dyadic grid on $X$ and $\delta$ is defined as \eqref{delta}. Then for any $x \in Q$,
\begin{equation}\label{EQ-keyboundednorm}
\mathop {\sup }\limits_{Q \in \mathcal{D}} \left\| {{v^{ - 1}}{\chi _{Q}}} \right\|_{h'( \cdot )}^{\delta (Q) - q(x)} \le {V_0}.
\end{equation}               
\end{lemma}

\begin{proof}
Without loss of generality, we still assume that $\|v^{-1}\chi_Q\|_{h'(\cdot)}\leq 1$ and $m=2$. 
From Lemma \ref{cubes}, there exist $C>1$, such that $B \subseteq Q \subseteq CB$. 
It follows from the fact that $\delta (CB) \le \delta (Q) \le \delta (B)$ that
$$\left\| {{v^{ - 1}}{\chi _Q}} \right\|_{h'( \cdot )}^{\delta (Q) - q(x)} \le \left\| {{v^{ - 1}}{\chi _B}} \right\|_{h'( \cdot )}^{\delta (Q) - q(x)} \le \left\| {{v^{ - 1}}{\chi _B}} \right\|_{h'( \cdot )}^{\delta (CB) - q(x)} \le {V_0},$$
where the last estimate holds due to Lemma \ref{B-q-relation}.
\end{proof}


\begin{definition}
Let $\eta \in[0,m)$, $\v$ is a weight, and $\mathcal{D}$ is a dyadic grid. Define the multilinear weighted dyadic fractional maximal operator $\M_{\eta, \v}^{\mathcal{D}}$ by
$$
\M_{\eta, \v}^{\mathcal{D}} (\vec{f})(x)=\sup _{\substack{x \in Q \in \mathcal{D}}}{\v(Q)}^{\eta-m} \prod_{i=1}^{m}\int_Q|f_i| \v d\mu.
$$
When $\eta=0, \M_{0, \v}^{\mathcal{D}}=\M_\v^{\mathcal{D}}$, which is a multilinear weighted dyadic maximal operator. Moreover, when $m = 1$, we denote $\M_\v^{\mathcal{D}}$ by $M_{\v}^{\mathcal{D}}$.
When $\v=1, \M_{\eta, \v}^{\mathcal{D}}=:\M_\eta^{\mathcal{D}}$, which is a multilinear dyadic fractional maximal operator.
\end{definition}

The following lemma can guarantee that we always transform a proof involving $\M_{\eta}$ into that for $\M_{\eta}^{\mathcal D_i}$. In case $m=1$,  it was shown in \cite[Proposition 7.9]{Hy2012}. The proof for $m>1$ are similar, which we omit here.
\begin{lemma}\label{Suff_1}
Let $\eta \in[0,m)$, there exists a finite family $\left\{\mathcal{D}_i\right\}_{i=1}^N$ of dyadic grids such that
	$$
	\M_\eta (\vec{f})(x) \approx \sum_{i=1}^N \M_\eta^{\mathcal{D}_i} (\vec{f})(x),
	$$
	where the implicit constants depend only $X$, $\mu$, and $\eta$.
\end{lemma}

\begin{lemma}[\cite{Cruz2022}, Lemma 4.4]\label{M.eta.bound}
    Let $\mathcal{D}$ is a dyadic grid, $\sigma$ is a weight, and $1<p<\infty$. Then the dyadic maximal operator $M_{\sigma}^{\mathcal{D}}$ is bounded on $L^p(X,\sigma)$, which is also bounded from $L^{1}(X,\sigma)$ to $WL^{1}(X,\sigma)$.
    \end{lemma}

We now present the multilinear fractional Calderón-Zygmund decomposition over the spaces of homogeneous type as follows.
\begin{lemma}\label{CZD}
Let $\eta \in[0,m)$, $\mathcal{D}$ is a dyadic grid on $X$, and $\v \in A_{\infty}$.
Set $\mu(X) = \infty$. If $f_i \in L_{\text{loc}}^1(\v)$ satisfying $\mathop {\lim }\limits_{j \to \infty } \v {\left( {Q_j} \right)^{\eta  - m}}\prod_{i=1}^{m}\int_{{Q_j}} {\left| f_i \right|\v d\mu =0}$ for any nested sequence $\left\{Q_j \in \mathcal{D}\right\}_{j=1}^{\infty}$, where $Q_{j}$ is a child of $Q_{j+1}$, then for any $\lambda > 0$, there exists a (possibly empty) collection of mutually disjoint dyadic cubes $\left\{ {Q_j^{}} \right\}$, called Calderón-Zygmund cubes for $\vec{f}$ at the height $\lambda$, and a constant $C_{CZ} > 1$, which is independent of $\lambda$ and dependent of $\mathcal{D}, X, \v$, such that
$$
X_{\eta, \lambda}^{\mathcal{D}}:=\left\{x \in X: \M_{\eta,\v}^{\mathcal{D}} (\vec{f})(x)>\lambda\right\} =\bigcup_j Q_j .
$$
Moreover, for each $j$,
\begin{align}\label{CZ_1}
\lambda<\v {\left( {{Q_j}} \right)^{\eta  - m}}\prod_{i=1}^{m}\int_{{Q_j}} {\left| f_i \right|\v d\mu } \le C_{C Z} \lambda .
\end{align}
Now, suppose that $\left\{ {Q_j^k} \right\}$ is the Calderón-Zygmund cubes at height $a^k$ for each $k \in \mathbb{Z}$ and $a > C_{CZ}$. These sets, $E_j^k:=Q_j^k \setminus X_{\eta,a^{k+1}}^{\mathcal{D}}$, are mutually disjoint for all indices $j$ and $k$, such that 
\begin{align}\label{sigema}
\left( {1 - {{\left( {\frac{{{C_{cz}}}}{a}} \right)}^{\frac{1}{{m - \eta }}}}} \right)\v (Q_j^k) \le \v (E_j^k) \le \v (Q_j^k).
\end{align}

If set $\mu(X) < \infty$, then Calderón-Zygmund cubes can be established for every function $f_i \in L_{loc}^1(\v)$ at any height $\lambda >\lambda_0:=\prod_{i=1}^{m}\int_X {\left| f_i \right|\v d\mu }$, meanwhile, \eqref{CZ_1} also holds. Under these conditions, the sets $E_j^k$ are pairwise disjoint with \eqref{sigema} holds, for $k > \log_a \lambda_0$.
\end{lemma}
\begin{proof}

The first case is that $\mu (X) = \infty $. We only need to consider  $X_{\eta,\lambda}^\mathcal{D} \ne \emptyset.$ Otherwise, we can take $\left\{ {{Q_j}} \right\}$ to be the empty sets.
	
As the property of the dyadic cube in Theorem \ref{cubes}, for every $x \in X_{\eta,\lambda}^\mathcal{D}$, there exists a dyadic cube ${{Q_k^x}}$ of each generation $k > 0$, such that $x \in {{Q_k^x}}$ and $\M_{\eta,\v}^{\mathcal{D}} (\vec{f})(x) > \lambda$. So there exist $k$, such that 
\begin{align}\label{CZ_2}
\v(Q_k^x)^{\eta-m}\prod_{i=1}^{m}\int_{Q_k^x}|f_i| d \v>\lambda.
\end{align}
Since
$
\lim _{k \rightarrow \infty} \v(Q_k^x)^{\eta-m}\prod_{i=1}^{m}\int_{Q_k^x}|f_i| d \v = 0,
$
then there are only finite $k$ such that \eqref{CZ_2} holds.
Select $k$ to be the smallest integer such that \eqref{CZ_2} holds, in this case, we denote the cube with generation $k$ by $Q_x$. What's  more, the set $\left\{ {{Q_x}:x \in X_{\eta, \lambda}^{\mathcal{D}}} \right\}$ can be enumerated as $\left\{ {{Q_j}} \right\}$ due to there are countable dyadic cubes. If ${Q_i} \cap {Q_j} \ne \emptyset $, without loss of generality, we define ${Q_i} \subseteq {Q_j}$. Moreover, by the maximality, ${Q_i} = {Q_j}$.
Thus, the set $\left\{ {{Q_x}:x \in X_{\eta, \lambda}^{\mathcal{D}}} \right\}:=\left\{ {{Q_j}} \right\}$ is countably non-overlapping maximal dyadic cubes.
Hence, $X_{\eta,\lambda}^{\mathcal{D}}\subseteq \bigcup_j Q_j$.

On the other hand, if $z \in Q_x$, for some $x \in X_{\eta, \lambda}^{\mathcal{D}}$, then 
$$
\lambda < \v(Q_x)^{\eta-m}\prod_{i=1}^{m}\int_{Q_x}|f_i| d \v\le \M_{\eta,\v}^{\mathcal{D}} f(z).
$$
Thus, $X_{\eta,\lambda}^{\mathcal{D}}= \bigcup_j Q_j $.

Next, we will prove \eqref{CZ_1}. The left inequality of \eqref{CZ_1} holds since the choice of $Q_j$. For the second inequality, by the maximality of each $Q_j$, we can deduce that its parent ${\tilde Q}_j$ satisfies
$$
\v({\tilde Q}_j)^{\eta-m}\prod_{i=1}^{m}\int_{Q_j}|f_i| d \v \leq \lambda
$$
It follows from Lemmas \ref{cubes} and \ref{LMB.} that
$$
\v(Q_j)^{\eta-m}\prod_{i=1}^{m}\int_{Q_j}|f_i| d \v \leq {\left( {\frac{{\v ( {{{\tilde Q}_j}})}}{{\v \left( {{Q_j}} \right)}}} \right)^{m - \eta }} \lambda 
\leq {\left( {\frac{{\v ( {B( {{x_c}( {{{\tilde Q}_j}}),Cd_0^{k + 1}} )})}}{{\v \left( {B\left( {{x_c}\left( {{Q_j}} \right),d_0^k} \right)} \right)}}} \right)^{m - \eta }} \lambda \leq {( {Cd_0^{{{\log }_2}{C_\mu }}})^{m - \eta }} \lambda.
$$
Consequencely, $\eqref{CZ_1}$ holds.

Setting $a > C_{CZ}$, we define the Calderón-Zygmund cubes $\{Q_j^k\}$ at heights $a^k$ for $k \in \mathbb{Z}$. For simplicity, we define $ X_{\eta,a^k}^{\mathcal{D}} \equiv X_k $. Given $Q_i^{k+1}$ and for any $x \in Q_i^{k+1}$, we have $Q_i^{k+1} \in \{Q_k^x\}$ (defined as above). It follows that there must be an index $j$ for which $Q_i^{k+1} \subseteq Q_j^k$.

Next, we want to show that the $E_j^k$ are pairwise disjoint for all $j,k$. Setting $k_1 \le k_2$, it suffices to prove that $E_{{j_1}}^{{k_1}} \cap E_{{j_2}}^{{k_2}} = \emptyset$ for $E_{{j_1}}^{{k_1}} \ne E_{{j_2}}^{{k_2}}$. If $k_1=k_2$ and $j_1\ne j_2$, then $Q_{{j_1}}^{{k_1}} \cap Q_{{j_2}}^{{k_2}} = \emptyset$ can deduce the desired results. 
If $k_1 < k_2$, then $E_{{j_1}}^{{k_1}} \subseteq {\left( {{X_{{k_1} + 1}}} \right)^c} \subseteq {\left( {{X_{{k_2}}}} \right)^c}$ and $E_{{j_2}}^{{k_2}} \subseteq {X_{{k_2}}}$ can deduce the desired results.

Finally, we will prove that $\v (Q_j^k) \approx \v (E_j^k)$.
It follows obviously that
\begin{align*}
\v {(Q_j^k \cap {X_{k + 1}})^{m - \eta }} &= {\left( {\sum\limits_{i:Q_i^{k + 1} \subseteq Q_j^k} {\v (Q_i^{k + 1})} } \right)^{m - \eta }} \le \sum\limits_{i:Q_i^{k + 1} \subseteq Q_j^k} {{{\left( {\v (Q_i^{k + 1})} \right)}^{m - \eta }}} \\
 &\le \frac{1}{{{a^{k + 1}}}}\sum\limits_{i:Q_i^{k + 1} \subseteq Q_j^k} {\prod_{i=1}^{m}\int_{Q_i^{k + 1}} {\left| f_i \right|d\v } } \le \frac{1}{{{a^{k + 1}}}}\prod_{i=1}^{m}\int_{Q_j^k} {\left| f_i \right|d\v }  \le \frac{{{C_{cz}}}}{a}\v {(Q_j^k)^{m - \eta }}.
\end{align*}
Note that $\v (Q_j^k) = \v (Q_j^k \cap {X_{k + 1}}) + \v (E_j^k)$, then we have
$$( {1 - {{({C_{cz}/a})}^{\frac{1}{{m - \eta }}}}})\v (Q_j^k) \le \v (E_j^k) \le \v (Q_j^k).\qedhere$$
\end{proof}

\section{\bf The proof of Theorems \ref{Cha.Apq}--\ref{mainthm_1}}\label{proof}
\subsection{Necessity}\label{proof1}
~

In this subsection, we aim to prove the necessity of Theorems \ref{Cha.Apq}--\ref{mainthm_1}. 

\begin{proof}
Since ${\A_{\eta ,B}}( {\vec f})(x) \le {\M_\eta }( {\vec f})(x),$ it is sufficent to prove the case of Theorem \ref{W.Cha.Apq}, which means that we want to prove
\begin{align}\label{prove1}
	{\left[ {\vec \omega } \right]_{{A_{\vec p( \cdot ),q( \cdot )}}(X)}}{ \lesssim} \mathop {\left\| {{\A_{\eta,B}}} \right\|_{{L^{{p_1}( \cdot )}}(X,{\omega _1}) \times  \cdots  \times {L^{{p_m}( \cdot )}}(X,{\omega _m}) \to {WL^{q( \cdot )}}(X,\omega )}}.
\end{align}

Now, we supppose that 
$\left\| {{\A_{\eta,B}}} \right\|: = {\left\| {{\A_{\eta,B}}} \right\|_{{L^{{p_1}( \cdot )}}(X,{\omega _1}) \times  \cdots  \times {L^{{p_m}( \cdot )}}(X,{\omega _m}) \to W{L^{q( \cdot )}}(X,\omega )}} < \infty$, which means that 
\begin{align}\label{WB-M}
\mathop {\sup }\limits_{t > 0} {\left\| {t\omega {\chi _{\left\{ {x \in X:{\M_\eta }(\vec{f})(x) > t} \right\}}}} \right\|_{q( \cdot )}} \le {\left\| {\A_{\eta,B}} \right\|}\prod_{i=1}^{m}{\left\| {\omega_i f_i} \right\|_{p_i( \cdot )}}.
\end{align}
Thus, it is sufficient to prove \eqref{prove1} as follows.

Firstly, we claim that for every $B \subseteq X$, 
\begin{align}\label{Nes.ineq.1}
{\left\| {\omega {\chi _B}} \right\|_{q( \cdot )}} < \infty.
\end{align} 
If ${\left\| {\omega {\chi _B}} \right\|_{q( \cdot )}} = \infty$. For any $x \in B$, there exist $E\subseteq B$, such that $x \in E$.
For any $t<\mu {(B)^{\eta  - m}}(\mu (E))^m$, then ${\A_{\eta,B}}(\vec{\chi_E})(x) = \mu {(B)^{\eta  - m}}(\mu (E))^m{\chi _B}(x)>t$. Moreover, it follows from \eqref{WB-M} that
$$
\infty = t{\left\| {\omega {\chi _B}} \right\|_{q( \cdot )}} \le {\left\| {t\omega {\chi _{\left\{ {x \in X:{\A_{\eta,B}}(\vec{\chi_E})(x) > t} \right\}}}} \right\|_{q( \cdot )}} \le {\left\| {\A_{\eta,B}} \right\|}\prod_{i=1}^{m}{\left\| {\omega_i {\chi _E}} \right\|_{p_i( \cdot )}}.
$$
If $\prod_{i=1}^{m}{\left\| {\omega_i {\chi _E}} \right\|_{p_i( \cdot )}} = \infty$, then there exists $i$ such that ${\left\| {{\omega _i}{\chi _E}} \right\|_{{p_i}( \cdot )}} = \infty$, which can produce that $\int_E {{\omega _i}} {( \cdot )^{{p_i}( \cdot )}}d\mu  = \infty$ due to Lemma \ref{p.omega}. Moverover, it follows from Lebesgue differential theorem that  ${\omega _i}(x) = \infty$. Because of the arbitrariness of $x$, then we have ${\left. {{\omega _i}} \right|_B} = \infty$. But this contradicts the definition of the weight.

We next show that $\vec{\omega} \in A_{\vec{p}( \cdot ),q( \cdot )}(X)$. 

{\bf{Case 1: ${\left\| {{\omega_i ^{ - 1}}{\chi _B}} \right\|_{p_i'( \cdot )}}< \infty$.}} 

In this case, due to homogeneity, we assume that \(\left\| \omega_i^{-1} \chi_B \right\|_{p_i'(\cdot)} = 1\). It suffices to prove that
\begin{align}\label{Nes.ineq.2}
\mathop {\sup }\limits_{B \subseteq X} \mu {(B)^{\eta  - m}}{\left\| {\omega {\chi _B}} \right\|_{q( \cdot )}} \lesssim {\left\| {\A_{\eta,B}} \right\|}
\end{align}
We need define two sets
\[ B_0 \equiv \{x \in B\,:\, p_i'(x) < \infty\}, \qquad B_\infty \equiv \{x \in B\,:\,p_i'(x) = \infty\}. \]
According to the definition of the norm, for any \(\lambda \in \left( \frac{1}{2}, 1 \right)\),
$$1 \le {\rho _{p_i'( \cdot )}}\left( {\frac{{{\omega_i ^{ - 1}}{\chi _B}}}{\lambda }} \right) = \int_{{B_0}} {{{\left( {\frac{{\omega_i {{(x)}^{ - 1}}}}{\lambda }} \right)}^{p_i'(x)}}} {\mkern 1mu} d\mu  + {\lambda ^{ - 1}}{\left\| {{\omega_i ^{ - 1}}{\chi _{{B_\infty }}}} \right\|_\infty }.$$
Therefore, since the two right-hand terms in the above equation are at least $1/2$, we claim that one of the following two situations must be true: either ${\left\| {{\omega_i ^{ - 1}}{\chi _{{B_\infty }}}} \right\|_\infty } \geq \frac{1}{2}$, or given $\lambda_0 \in (\frac{1}{2},1)$, then $\int_{B_0}\left(\frac{\omega_i(x)^{-1}}{\lambda}\right)^{p_i'(x)}\,d\mu \geq \frac{1}{2}$ for any $\lambda \in [\lambda_0,1)$.

Consider first that when the first situation holds.
Set 
$s > {\left\| {{\omega_i ^{ - 1}}{\chi _{{B_\infty }}}} \right\|_\infty } = \operatorname{essinf}\limits_{x \in {B_\infty }} \omega_i (x)$, there exists a subset $E \subseteq B_\infty$ with $\mu(E)>0$, such that $\mu {(E)^{ - 1}}\omega_i (E) \le s$. Note that ${p_i( \cdot )}$ is equal to $1$ on $B_\infty$,
then ${\left\| {\omega_i {\chi _E}} \right\|_{p_i( \cdot )}} = \omega_i (E)$.
Then, for all $t < \mu {(B)^{\eta  - m}}(\mu (E))^m$, we have
${\A_{\eta,B}}{(\vec{\chi _E})}(x) \ge \mu {(B)^{\eta  - m}}(\mu (E))^m{\chi _B}(x)>t{\chi _B}(x)$.
Thus, it follows from \eqref{WB-M} that  
$$t{\left\| {\omega {\chi _B}} \right\|_{q( \cdot )}} \le {\left\| {t\omega {\chi _{\left\{ {x \in X:{\A_{\eta,B}}{(\vec{\chi _E})}(x) > t} \right\}}}} \right\|_{q( \cdot )}} \lesssim {\left\| {\A_{\eta,B}} \right\|}\prod_{i=1}^{m}{\left\| {\omega_i {\chi _E}} \right\|_{p_i( \cdot )}} = {\left\| {\A_{\eta,B}} \right\|}\prod_{i=1}^{m}\omega_i (E).$$
Letting $t\to \mu {(B)^{\eta  - m}}(\mu (E))^m$, we get that
$\mu(B)^{\eta-m}\mu(E)^m\left\|\omega \chi_B\right\|_{q(\cdot)} \lesssim {\left\| {\A_{\eta,B}} \right\|}\prod_{i=1}^m \omega_i(E)$.
Then,
\[\mu {(B)^{\eta  - m}}\| \omega {\chi _B}\| _{{q( \cdot )}} \lesssim {\left\| {\A_{\eta,B}} \right\|}\mu {(E)^{- m}}\prod_{i=1}^{m}\omega_i (E) \le{\left\| {\A_{\eta,B}} \right\|} s^m\]
Letting $s \to \parallel {\omega_i ^{ - 1}}{\chi _{{B_\infty }}}\parallel _\infty ^{ - 1}$, we have
\[\mu {(B)^{\eta  - m}}{\left\| {\omega_i {\chi _B}} \right\|_{q( \cdot )}} \lesssim {\left\| {\A_{\eta,B}} \right\|}\left\| {{\omega_i ^{ - 1}}{\chi _{{B_\infty }}}} \right\|_\infty ^{ - m} \le{\left\| {\A_{\eta,B}} \right\|} 2^m,\]
then \eqref{Nes.ineq.2} is valid.

When the second situation holds, we define $B_R = \{x \in B_0\,:\,p_i'(x) < R\}$, for any $R>1$. By Lemma~\ref{Levi}, there exists $R$ that close to $\infty$ sufficiently, such that $\int_{B_R} \left(\frac{\omega_i(x)^{-1}}{\lambda_0}\right)^{p_i'(x)}\,d\mu > \frac{1}{3}.$
It follows from ${\left\| {{\omega_i ^{ - 1}}{\chi _B}} \right\|_{p_i'( \cdot )}} = 1$ and Lemma~\ref{p.omega} that
\begin{align*}
	\int_{B_R}\left(\frac{\omega_i(x)^{-1}}{\lambda_0}\right)^{p_i'(x)}\,d\mu
	\le \int_{B_R}\left(\frac{2}{\lambda_0}\right)^{p_i'(x)}\left(\frac{\omega_i(x)^{-1}}{2}\right)^{p_i'(x)}\,d\mu  \le \left(\frac{2}{\lambda_0}\right)^R < \infty.
\end{align*}
We need to use the following auxiliary function
\[ G(\lambda) = \int_{B_R} \left(\frac{\omega_i(x)^{-1}}{\lambda}\right)^{p_i'(x)}\,d\mu, \]
where $\frac{1}{3} < G(\lambda_0) < \infty$. The Lebesgue dominated convergence theorem can deduce that $G$ is continuous on $[\lambda_0,1]$.

For any $\lambda \in [\lambda_0,1)$, if $G(1) \geq \frac{1}{3}$, by Lemma~\ref{p.omega}, 
\[ \frac{1}{3\lambda} \leq \frac{1}{\lambda}\int_{B_R}\omega_i(x)^{-p_i'(x)}\,d\mu \leq G(\lambda) \leq \lambda^{-R} < \infty. \]
Let $\lambda$ sufficiently close to 1, then $\lambda^{-R} \leq 2$ and
 \begin{align} 
	\frac{1}{3} \leq \int_{B_R}\left(\frac{\omega_i(x)^{-1}}{\lambda}\right)^{p_i'(x)}\,d\mu \leq 2. \label{finally} 
\end{align}
If $G(1) < \frac{1}{3}$, by continuity of $G$, there exists $\lambda \in (\lambda_0,1)$ such that $G(\lambda)=\frac{1}{3}$. Then \eqref{finally} holds for this $\lambda$ as well.

Fixed $\lambda$ and let 
$$f_i(x) = \omega_i {(x)^{ - p_i'(x)}}{\lambda ^{1 - p_i'(x)}}{\chi _{{B_R}}}, \quad \vec{f}=(f_1,\cdots,f_2).$$
Then
\[ \rho_{\pp}(\omega_i f_i) = \int_{B_R}\left(\frac{\omega_i(x)^{-1}}{\lambda}\right)^{p_i'(x)}\,d\mu  \leq 2. \]
By Lemma~\ref{p.omega}, ${\left\| {\omega_i f_i} \right\|_{p_i( \cdot )}} \le 2^{\frac{1}{{{{(p_i')}_ - }}}}$. For any $x \in B$,

\[ {\A_{\eta,B}}(\vec{f})(x) = \mu {(B)^{\eta - m }} \prod_{i=1}^{m}\int_B {f_id\mu}  =  
\lambda^m \mu {(B)^{\eta - m }}\prod_{i=1}^{m}\int_{{B_R}} {{{(\frac{{\omega_i {{(x)}^{ - 1}}}}{\lambda })}^{p_i'(x)}}d\mu} 
\geq (\frac{\lambda }{3})^m\mu {(B)^{\eta  - m}}. \]
For any $t < (\frac{\lambda }{3})^m\mu {(B)^{\eta  - m}}$, it follows from \eqref{WB-M}  that
\[t{\left\| {\omega {\chi _B}} \right\|_{q( \cdot )}} \le {\left\| {t\omega {\chi _{\left\{ {x \in X:{\A_{\eta,B}}(\vec{f})(x) > t} \right\}}}} \right\|_{q( \cdot )}} \le {\left\| {\A_{\eta,B}} \right\|}\prod_{i=1}^{m}{\left\| {\omega_i f_i} \right\|_{p_i( \cdot )}} \le {\left\| {\A_{\eta,B}} \right\|}\prod_{i=1}^{m}{2^{\frac{1}{{{{(p_i')}_ - }}}}}.
\]
Letting $t \to (\frac{\lambda }{3})^m\mu {(B)^{\eta  - m}}$, \eqref{Nes.ineq.2} is valid.

{\bf{Case 2: ${\left\| {{\omega_i ^{ - 1}}{\chi _B}} \right\|_{p_i'( \cdot )}}= \infty$.}}

In this case, we will use the perturbation method to prove.

Given $\epsilon \in (0,1)$, we define
\begin{align*}
    \omega_{i,\epsilon}(x) :=
    \begin{cases}
\omega_i(x),\quad &\omega_i(x)>\epsilon, \\
\epsilon, \quad & \omega_i(x)\leq \epsilon,
    \end{cases}
\end{align*}
with $\omega_{\epsilon}:=\prod_{i=1}^{m}\omega_{i,\epsilon}$. 

Note that ${\omega_{i,\epsilon}}^{-1} \le {\omega_{i}}^{-1} \le \epsilon^{-1}\, a.e.$, then  ${\left\| {{\omega_{i,\epsilon} ^{ - 1}}{\chi _B}} \right\|_{p_i'( \cdot )}} < \infty$. Combining \eqref{WB-M}, it follows that 
\begin{align*}
  t{\left\| {{\omega}{\chi _{\left\{ {x \in X:{\A_{\eta,B}}(\vec{f})(x) > t} \right\}}}} \right\|_{q( \cdot )}} \le {\left\| {\A_{\eta,B}} \right\|} \prod_{i=1}^{m}\left\| {\omega_i f_i}\right\|_{p_i({\cdot})} \le {\left\| {\A_{\eta,B}} \right\|}  \prod_{i=1}^{m}\left\| {\omega_{i,\epsilon} f_i}\right\|_{p_i({\cdot})}.
\end{align*}
This shows that $(\omega_{1,\epsilon},\cdots,\omega_{m,\epsilon},\omega)$ satisfies \eqref{WB-M} (Here we need to explain that in fact, in the proof process before {\bf{Case 2}}, we do not use the condition that $\omega  = \prod\limits_{i = 1}^m {{\omega _i}}$, so we can also get the results produced in {\bf{Case 1}} for this multiple weight $(\omega_{1,\epsilon},\cdots,\omega_{m,\epsilon},\omega)$).

Since ${\left\| {{\omega_{i,\epsilon} ^{ - 1}}{\chi _B}} \right\|_{p_i'( \cdot )}} < \infty$, it follows from \eqref{Nes.ineq.2} that $$ \mu {(B)^{\eta  - m}}{\left\| {\omega{\chi _B}} \right\|_{q( \cdot )}}\prod_{i=1}^{m}{\left\| {{\omega_{i,\epsilon} ^{ - 1}}{\chi _B}} \right\|_{p_i'( \cdot )}} \lesssim {\left\| {\A_{\eta,B}} \right\|}.$$

The desired result that ${\left[ \vec{\omega}  \right]_{{A_{\vec{p}( \cdot ),q( \cdot )}}}(X)} \lesssim {\left\| {\A_{\eta,B}} \right\|}$ follows instantly from Lemma \ref{lemma:fatou}, since $\mathop {\lim }\limits_{\varepsilon  \to 0} \omega _{i,\varepsilon } = \omega _i$ a.e..
This finishes the necessity.
\end{proof}

\medskip
{\bf{A concise proof of the necessity for Theorems \ref{Cha.Apq} and \ref{mainthm_1}}}

\begin{proof}
Similar to the previous argument, it is enough to prove that
the necessity of Theorem \ref{Cha.Apq}.
	
By Lemma \ref{lemma:dual}, there exists ${{h_j}}$, such that ${\left\| {{\omega _j}{h_j}} \right\|_{{{{p_j}( \cdot )}}}} \le 1,j=1,\ldots,m$. Then we have 
	\begin{align*}
	&{\mu(B)^{ \eta- m}}{\left\| {\omega {\chi _B}} \right\|_{{{q( \cdot )}}}}\prod\limits_{i = 1}^m {{{\left\| {\omega _i^{ - 1}{\chi _B}} \right\|}_{{{{p_i}^\prime ( \cdot )}}}}}\\
	\lesssim&_{\vec{p}(\cdot)} {\left\| {\omega {\chi _B}} \right\|_{{L^{q( \cdot )}}}}{\mu(B)^{ \eta- m }}\prod\limits_{i = 1}^m {\int_B {{h_i}} d\mu}\\
	\le& {\left\| {{\A_{\eta ,B}}} \right\|_{{L^{{p_1}( \cdot )}}(X,{\omega _1}) \times  \cdots  \times {L^{{p_m}( \cdot )}}(X,{\omega _m}) \to {L^{q( \cdot )}}(X,\omega )}}\prod\limits_{i = 1}^m {{{\left\| {{h_i}} \right\|}_{{L^{{p_i}( \cdot )}}(X,{\omega _i})}}}.
	\end{align*}
	Thus,
	$${\left[ {\vec \omega } \right]_{{A_{\vec p( \cdot ),q( \cdot )}}}}{ \lesssim _{\vec p(\cdot)}}{\left\| {{\A_{\eta ,B}}} \right\|_{{L^{{p_1}( \cdot )}}(X,{\omega _1}) \times  \cdots  \times {L^{{p_m}( \cdot )}}(X,{\omega _m}) \to {L^{q( \cdot )}}(X,\omega )}}.\qedhere$$
\end{proof}

\subsection{Sufficiency}\label{proof2}
~

The aim of this subsection is to prove the sufficiency of Theorem \ref{mainthm_0}, which implies the sufficiency of the other theorems.
We will address the case for $\mu(X) < \infty$ at the end of this subsection. We firstly consider the case for $\mu(X) = \infty$.

{\bf{Case 1: $\mu(X) = \infty$.}} 

We first simplify some details with four steps.

{\bf{Step 1.}} 
Lemma \ref{Suff_1} implies that to establish the boundedness of $\M_\eta$, it is sufficient to demonstrate the boundedness of $\M_\eta^{\mathcal{D}}$.  


{\bf{Step 2.}} 
For \(\pp \in \P_0\) and a given weight \(v\), we define \(\Lp_v(X)\) as the quasi-Banach function space with the norm
$${\left\| f \right\|_{L_v^{p( \cdot )}(X)}}: = {\left\| {{v^{\frac{1}{{p( \cdot )}}}}f} \right\|_{L_{}^{p( \cdot )}(X)}}.$$ The norm has many of the same basic properties as the $\Lp(X)$ norm.

	Let $u(\cdot)=\omega(\cdot)^{q(\cdot)}$ and
	$\sigma_l(\cdot)=\omega_l(\cdot)^{-p_l'(\cdot)}$, $l=1,\dots,m$. 
    According to Lemma \ref{Ap_Ainfty_} and Lemma \ref{Ainfty},  $u(\cdot)$ and $\sigma_l(\cdot)$ are both in $A_{\infty}(X)$ and satisfy the doubling property.

	\[   (\omega_{l}(x)\sigma_{l}(x))^{p_{l}(x)}
	=(\omega_{l}(x)^{p_{l}^{\prime}(x)-1})^{-p_{l}(x)}
	= \omega_{l}(x)^{-p_{l}^{\prime}(x)}
	= \sigma_{l}(x).
	\]
	Therefore,
	$${\left\| {{\M_\eta^{\mathcal{D}}}({f_1}{\sigma _1}, \ldots ,{f_m}{\sigma _m})} \right\|_{L_u^{q( \cdot )}(X)}} = {\left\| {{\M_\eta^{\mathcal{D}}}({f_1}{\sigma _1}, \ldots ,{f_m}{\sigma _m})\omega } \right\|_{L_{}^{q( \cdot )}(X)}},$$
	and for $l=1,\dots,m$,
	\begin{align*}
	\|f_{l}\|_{L_{\sigma_{l}}^{p_{l}(\cdot)}(X)}
	= \|f_{l}\sigma_{l}\omega_{l}\|_{L^{p_{l}(\cdot)}(X)}.  
	\end{align*}     
	Thereby, it suffices to prove that 
	\begin{align*}\label{EQ-boundenessofM2}
	\|{\M_\eta^{\mathcal{D}}}(f_{1}\sigma_{1},\ldots,f_{m}\sigma_{m})\|_{\Lq_{u}(X)}
	\lesssim \prod\limits_{i = 1}^m {{{\left\| {{f_i}} \right\|}_{L_{{\sigma _i}}^{{p_i}( \cdot )}(X)}}},
	\end{align*}
	since we can replace $f_{l}$ by $f_{l}/\sigma_{l}$, $l=1,\ldots,m$.

{\bf{Step 3.}} By the homogeneity, it suffices to consider that $f_l$ is a nonnegative function with \(\|f_l\|_{L^{p_l(\cdot)}_{\sigma_{l}}(X)}=1\), for each \(l=1, \ldots, m\), it follows from Lemma \ref{p.omega} that
	$$\int_{X} | {f_l}(x){|^{{p_l}(x)}}{\sigma _l}(x){\mkern 1mu} d \mu \le 1.$$
	Therefore, it is sufficient to
demonstrate that
\begin{align*}
\|{\M_\eta^\mathcal{D}}(f_{1}\sigma_{1},\dots,f_{m}\sigma_{m})\|_{\Lq_{u}}\lesssim 1.
\end{align*}
From Lemma \ref{p.omega}, it is enough to prove
\begin{equation}\label{modular_vertion_Maxboundeness}
\int_{X}{\M_\eta^\mathcal{D}}(f_{1}\sigma_{1},\ldots,f_{m}\sigma_{m})^{q(x)}u(x)\,d \mu \lesssim 1.
\end{equation}

{\bf{Step 4.}} We also need to show that for any nested sequence $\left\{Q_k\in \mathcal{D}\right\}_{k=1}^{\infty}$ with $Q_{k}$ is a child of $Q_{k+1}$,
\begin{equation}\label{Suff.lim}
\mathop {\lim }\limits_{k \to \infty } \mu {({Q_k})^{\eta  - m}}\prod_{i=1}^{m}\int_{{Q_k}} f_i \sigma_i d\mu  = 0,
\end{equation}
which can guarantee us to use Lemma \ref{CZD}.

Indeed, since $u$ is doubling, if we fix a sequence with $k=1$, then 
\begin{equation*}
    u\left(Q_1\right) \leq u\left(B\left(x_c\left(Q_1\right), C_d d_0\right)\right) \leq C_u^{\log _2 C_d} u\left(B\left(x_c\left(Q_1\right), d_0\right)\right).
\end{equation*}
By Lemma \ref{cubes}, for any $k$, with the similar argument, we have
$$
\frac{1}{u\left(Q_k\right)} \lesssim \frac{1}{u\left(B\left(x_c\left(Q_k\right), C_d d_0^k\right)\right)}.
$$
Using lemma \ref{Ainfty} and combining above two estimates, we get
$$
\frac{u\left(Q_1\right)}{u\left(Q_k\right)} \lesssim \frac{u\left(B\left(x_c\left(Q_1\right), d_0\right)\right)}{u\left(B\left(x_c\left(Q_k\right), C_d d_0^k\right)\right)} \lesssim \left(\frac{\mu\left(B\left(x_c(Q_1), d_0\right)\right.}{\mu\left(B\left(x_c\left(Q_k\right), C_d d_0^k\right)\right)}\right)^\delta.
$$
If we rearrange and apply Lemma \ref{LMB.} (the lower mass bound), then
$$\mu\left(B\left(x_c\left(Q_1\right), C d_0^k\right)\right)^\delta \lesssim \mu\left(B\left(x_c\left(Q_k\right), C_d d_0^k\right)\right)^\delta \lesssim u\left(Q_k\right).$$
By continuity of $\mu$ and the fact that $X =\mathop {\lim }\limits_{k \to \infty } B\left( {{x_c}\left( {{Q_1}} \right),Cd_0^k} \right)$, we have $\mathop {\lim }\limits_{k \to \infty } u \left( {{Q_k}} \right)=\infty$. 

By the definition of $A_{\vec{p}(\cdot),q(\cdot)}$, Lemma \ref{Gu6}, Lemma \ref{Holder}, and Lemma \ref{p.omega},
\begin{align*}
    \mu {\left( {{Q_k}} \right)^{\eta  - m}}\prod_{i=1}^{m}\int_{{Q_k}} {f_i \sigma_i d\mu }  
    &\lesssim \mu {\left( {{Q_k}} \right)^{\eta  - m}} \prod_{i=1}^{m} \|f_i\sigma_i^{\frac{1}{p_i(\cdot)}}\|_{p_i(\cdot)} \| {\sigma_i^{\frac{1}{p'_i(\cdot)}} {\chi _{{Q_k}}}} \|_{p'_i( \cdot )}\\
    &\lesssim \mu {\left( {{Q_k}} \right)^{\eta  - m}} \prod_{i=1}^{m}\|\omega_i^{-1}\chi_Q\|_{p'_i(\cdot)}\\
    & \lesssim {\left[ \vec{\omega}  \right]_{{A_{\vec{p}( \cdot ),q( \cdot )}}}}\left\| {\omega {\chi _{{Q_k}}}} \right\|_{q( \cdot )}^{ - 1}\\
    &\lesssim \left\| {\omega {\chi _{{Q_k}}}} \right\|_{q( \cdot )}^{ - 1} . 
\end{align*}
Since Lemma \ref{p.omega} implies $\mathop {\lim }\limits_{k \to \infty } u\left( {{Q_k}} \right) = \mathop {\lim }\limits_{k \to \infty } \left\| {\omega {\chi _{{Q_k}}}} \right\|_{q( \cdot )}^{} = \infty$, $\eqref{Suff.lim}$ is valid.

Now, we begin to estimate.

Without loss of generality, we merely need to prove the situation when $m=2$.
Define the functions
\begin{align*}
h_{1}=f_{1}\chi_{\lbrace f_{1}>1\rbrace}, \;
 h_{2}=f_{1}\chi_{\lbrace f_{1}\leq 1 \rbrace}, 
h_{3}=f_{2}\chi_{\lbrace f_{2}>1\rbrace},\;
 h_{4}=f_{2}\chi_{\lbrace f_{2} \leq 1\rbrace},
\end{align*}
and for brevity define
\[     \rho(1)= 1, \quad \rho(2)= 1, \quad
\rho(3)= 2, \quad \rho(4)= 2.  \]
Then,
\begin{align}\notag
&\int_X \M_\eta^\mathcal{D}\left( f_{1}\sigma_{1},
f_{2}\sigma_{2}\right)(x)^{q(x)}u(x)\,d \mu\\ \notag
\le&
\int_X \M_\eta^\mathcal{D}\left( h_{1}\sigma_{1}, h_{3}\sigma_{2}\right)(x)
^{q(x)}u(x)\,d \mu
+\int_X \M_\eta^\mathcal{D}\left( h_{1}\sigma_{1}, h_{4}\sigma_{2}\right)(x)
^{q(x)}u(x)\,d \mu\\ \notag
&+ 
\int_X \M_\eta^\mathcal{D}\left( h_{2}\sigma_{1},
h_{3}\sigma_{2}\right)(x)^{q(x)}u(x)\,d \mu 
+ \int_X \M_\eta^\mathcal{D}\left( h_{2}\sigma_{1}, h_{4}\sigma_{2}\right)(x)
^{q(x)}u(x)\,d \mu\\ \label{similar_2}
=:& I_{1}+I_{2}+I_{3}+I_{4},
\end{align}
where we note that $I_2$ is similar to $I_3$.

Let $k\in \mathbb{Z}$ and $a>C_{C Z}>1$, and we define
$$
X_k=\left\{ {x \in X:{\M_{\eta}^\d}({f_1}{\sigma _1}, \cdots ,{f_m}{\sigma _m})(x) > {a^k}} \right\} .
$$
Since $f_i \in L_{ {loc }}^1$ and $\mathop {\lim }\limits_{k \to \infty } \mu {({Q_k})^{\eta  - m}}\prod_{i=1}^{m}\int_{{Q_k}} f_i \sigma_id\mu  = 0$, then by Lemma \ref{CZD}, \({X_k} = \bigcup_{j \in \mathbb{Z}} {Q_j^k}\), where \(\left\{Q_j^k\right\}_j\) is a family of non-overlapping maximal dyadic cubes with the property that
\begin{align}\label{similar_}
    a^k < \mu(Q_j)^{\eta - m} \prod_{i=1}^{m}  \int_{Q_j} f_i \sigma_i \, d\mu  \leq C_{CZ} a^k < a^{k+1}.
\end{align}
Furthermore, the sets \(E_j^k = \Q \setminus X_{k+1}\) are each disjoint for any $k,j$. 
For any $x \in X$, it is follow obviously that 
\begin{align}\label{zdkz_}
    \M_{\eta}^{\d} (f_1\sigma_1,\cdots,f_m\sigma_m)(x) \approx \sum_{k,j}\mu(\Q)^{\eta} \prod_{i_1}^{m} \langle f_i\sigma_i\rangle_{\Q} \chi_{E^k_j}(x).
\end{align}
Since $\mu\left(Q_j^k\right) \approx \mu\left(E_j^k\right)$ and $u,\sigma_l \in A_{\infty}(X)$, by Lemma \ref{CZD} and Lemma \ref{Ainfty}, we obtain $u\left(Q_j^k\right) \approx u\left(E_j^k\right)$ and
$\sigma_l\left(Q_j^k\right) \approx \sigma_l\left(E_j^k\right)$, for $l=1,2$. 


\textbf{Estimate for $I_1$:}
\begin{align}
I_1 
 \label{case2_1}
&\lesssim \sum_{k,j}\int_{E_{j}^{k}}\prod_{l=1,3}
\bigg(\int_{\Q}{h_{l}\sigma_{\rho(l)}d \mu}\bigg)^{q(x)}
\mu(Q_j^k)^{(\eta-2)q(x)}u(x)\,d\mu.
\end{align}   
Obviously,
\begin{equation}\label{Bound_function_geq1}
\int_{\Q}h_1(y)\sigma_{1}(y)\,d \mu
\leq \int_\subRn f_1(y)^{p_1(y)}\sigma_1(y)\,d \mu \leq 1. 
\end{equation}
The above estimate is also valid for $h_3$. 

First, suppose that
\begin{equation}\label{EQ-I1estimte}
\int_{E_{j}^{k}}\prod_{l=1,3}
\mu(Q_j^k)^{(\eta-2)q(x)}\sigma_{\rho(l)}(\Q)^{\delta(\Q)}u(x)\,d \mu
\lesssim \left(\sigma_{1}(\Q)^{\frac{1}{(p_{1})_{-}(\Q)}}
\sigma_{2}(\Q)^{\frac{1}{(p_{2})_{-}(\Q)}}\right)^{\delta(\Q)}.
\end{equation}

Next, by \eqref{delta}, for almost all $x\in \Q$, we have that
$\delta(\Q)\leq q_{-}(\Q)\leq q(x)$. Thus,
\begin{align*}
I_{1} 
&\leq  \sum_{k,j}\int_{E_{j}^{k}}
\prod_{l=1,3}\bigg(\int_{\Q}h_{l}(y)\sigma_{\rho(l)}(y)\,d \mu \bigg)
^{\delta(\Q)}\mu(Q_j^k)^{(\eta-2)q(x)}u(x)\,d \mu\\
& \leq \sum_{k,j}\int_{E_{j}^{k}}
\prod_{l=1,3}\bigg(\frac{1}{\sigma_{\rho(l)}(\Q)}
\int_{\Q}h_{l}(y)^{\frac{p_{l}(y)}{(p_{l})_{-}(\Q)}}\sigma_{\rho(l)}(y)\,d \mu 
\bigg)^{\delta(\Q)} \\
& \qquad \qquad \times 
\sigma_{\rho(l)}(\Q)^{\delta(\Q)}\mu(Q_j^k)^{(\eta-2)q(x)}u(x)\,d \mu.
\end{align*}
Using Hölder's inequality with the measure \(\sigma_{\rho(l)} \, d \mu\) for \(l=1, 3\),
\begin{align}
&\bigg(({\sigma_{\rho(l)}(\Q)})^{-1}
\int_{\Q}h_{l}(y)^{\frac{p_{l}(y)}{(p_{l})_{-}(\Q)}}
\sigma_{\rho(l)}(y)\,d \mu\bigg)^{\delta(\Q)} \notag\\
\le& \bigg(({\sigma_{\rho(l)}(\Q)})^{-1}
\int_{\Q}h_{l}(y)^{\frac{p_{l}(y)}{(p_{l})_{-}}}\sigma_{\rho(l)}(y)\,d \mu
\bigg)^{(p_{l})_{-}\frac{\delta(\Q)}{(p_{l})_{-}(\Q)}}
= \langle h_{l}^{\frac{p_{l}(\cdot)}{(p_{l})_{-}}}
\rangle_{\sigma_{\rho(l)},\Q}^{(p_{l})_{-}\frac{\delta(\Q)}{(p_{l})_{-}(\Q)}}.\label{EQ-estimateHölder}
\end{align}
%

%
%
It follows from \eqref{EQ-estimateHölder}, \eqref{EQ-I1estimte}, and Young's inequality that
\begin{align}
I_1 
\nonumber &\lesssim  \sum_{k,j}\left(\prod_{l=1,3}
\langle h_{l}^{\frac{p_{l}(\cdot)}{(p_{l})_{-}}}
\rangle_{\sigma_{\rho(l)},\Q}^{(p_{l})_{-}\frac{\beta(\Q)}{(p_{l})_{-}(\Q)}}
\sigma_{\rho(l)}(\Q)^{\frac{\beta(\Q)}{(p_{l})_{-}(\Q)}}\right)^{\frac{{\delta (\Q)}}{{\beta (\Q)}}}\\
\label{eqn:final-I1-est}     &\lesssim \sum_{k,j}\left(\sum_{l=1,3}
\langle h_{l}^{\frac{p_{l}(\cdot)}{(p_{l})_{-}}}
\rangle_{\sigma_{\rho(l)},\Q}^{(p_{l})_{-}}\sigma_{\rho(l)}( Q_{j}^{k})\right)^{\frac{{\delta (\Q)}}{{\beta (\Q)}}}\\
\nonumber  &\lesssim  \sum_{\theta=1,c}\left(\sum_{k,j}\sum_{l=1,3}
\langle h_{l}^{\frac{p_{l}(\cdot)}{(p_{l})_{-}}}
\rangle_{\sigma_{\rho(l)},\Q}^{(p_{l})_{-}}\sigma_{\rho(l)}(
E_{j}^{k})\right)^\theta.\\
\intertext{ By 
Lemma~\ref{M.eta.bound}, since $(p_l)_->1$, the sum}
\nonumber     &\leq \sum_{\theta=1,c}\left(\sum_{l=1,3} \int_{\subRn}
M_{\sigma_{\rho(l)}}^{\mathcal{D}}(h_{l}^{\frac{p_{l}(\cdot)}{(p_{l})_{-}}})(x)
^{(p_{l})_{-}}\sigma_{\rho(l)}(x)\,d \mu\right)^\theta\\
\nonumber     & \lesssim \sum_{\theta=1,c}\left(\sum_{l=1,3} \int_{\subRn}
h_{l}(x)^{p_{l}(x)}\sigma_{\rho(l)}(x)\,d \mu\right)^\theta\\
\nonumber     & \lesssim 1,
\end{align}
where $c={c_{\eta,p_1(\cdot),p_2(\cdot)}}\ge1$, and it is independent  on $\Q$.

Last, we need to verify \eqref{EQ-I1estimte}. By rearrangement,
\begin{align*}
& \int_{E_{j}^{k}}\prod_{l=1,3}\mu(Q_j^k)^{({\eta - 2})q(x)}\sigma_{\rho(l)}(\Q)^{\delta(\Q)}u(x)\,d \mu \\
&\leq  \prod_{l=1,3}
\bigg(\frac{\sigma_{\rho(l)}(\Q)}
{\|\omega_{\rho(l)}^{-1}\chi_\Q\|_{p_{\rho(l)}^{\prime}(\cdot)}}
\bigg)^{\delta(\Q)}\\
&\times
\int_{\Q}\bigg(\prod_{l=1,3}
\|\omega_{\rho(l)}^{-1}\chi_\Q\|_{p_{\rho(l)}^{\prime}(\cdot)}^{\delta(\Q)-q(x)}\bigg)
\bigg(\prod_{l=1,3}
\mu(Q_j^k)^{({\eta - 2})q(x)}\|\omega_{\rho(l)}^{-1}\chi_{\Q}\|_{p_{\rho(l)}^{\prime}(\cdot)}^{q(x)}
u(x)\bigg)\,d \mu.
\end{align*}      

It suffices to prove that for \(l=1,\,2\),
\begin{align}
\int_{\Q}\prod_{l=1 }^{2}
\|\omega_{l}^{-1}\chi_{\Q}\|_{p_{l}^{\prime}(\cdot)}^{q(x)}
\mu(Q_j^k)^{(\eta - 2)q(x)}u(x)\,d \mu \lesssim& 1.\label{eqn:modular-Ap}\\ 
\bigg(\frac{\sigma_{l}(\Q)}
{\|\omega_{l}^{-1}\chi_{\Q}\|_{p_{l}^{\prime}(\cdot)}}\bigg)^{\delta(\Q)}
\lesssim&  \sigma_{l}(\Q)^{\frac{\delta(\Q)}{(p_{l})_{-}(\Q)}}\label{EQ-I11estimate}\\
\|\omega_{l}^{-1}\chi_{\Q}\|_{p_{l}^{\prime}(\cdot)}^{\delta(\Q)-q(x)}
\lesssim& 1.\label{EQ-I12estimate}
\end{align}

Firstly, \eqref{eqn:modular-Ap} follows from the $A_{\vec{p}(\cdot),q(\cdot)}(X)$ condition and Lemma \ref{p.omega}.

Next, we proceed to prove \(\eqref{EQ-I12estimate}\). Assume that
$
\|\omega_{l}^{-1}\chi_{\Q}\|_{p_{l}^{\prime}(\cdot)}
\le 1,
$
otherwise, there is nothing to prove. 
By Lemma \ref{q-relation} and Lemma \ref{vweight4}, 
$$\left\| {{\omega _l}^{ - 1}{\chi _{Q_j^k}}} \right\|_{{{p_l}^\prime ( \cdot )}}^{\delta (Q) - q(x)} = \left\| {{\omega _l}^{ - \frac{1}{m}}{\chi _{Q_j^k}}} \right\|_{{m{p_l}^\prime ( \cdot )}}^{m(\delta (Q) - q(x))} \lesssim 1.$$
Then \eqref{EQ-I12estimate} is obvious.

Last, we prove~\eqref{EQ-I11estimate} as follows. Suppose that
$\|\omega_{l}^{-1}\chi_{\Q}\|_{p_{l}^{\prime}(\cdot)} >1$. Then, by invoking Lemma~\ref{p.omega} and considering that \((p_l^{\prime})_{\pm}(\Q) = (p_{l})_{\mp}(\Q)^{\prime}\), it follows that
\begin{equation*}
\bigg(\frac{\sigma_{l}(\Q)}
{\| \omega_{l}^{-1}\chi_\Q\|_{p_{l}^{\prime}(\cdot)}}\bigg)^{\delta(\Q)}
\leq \bigg( \sigma_{l}(\Q)
^{\frac{(p_{l})_{-}(\Q)^{\prime}-1}{(p_{l})_{-}(\Q)^{\prime}}} 
\bigg)^{\delta(\Q)}
=\sigma_{l}(\Q)^{\frac{\delta(\Q)}{(p_{l})_{-}(\Q)}}.
\end{equation*}
If $\|\omega_{l}^{-1}\chi_{\Q}\|_{p_{l}^{\prime}(\cdot)}
\leq 1$, 
\begin{align*}
\frac{\sigma_l(\Q)}{\|\omega_l^{-1}\chi_\Q\|_{p_l'(\cdot)}}
\leq  \sigma_l(\Q)^{\frac{(p_l)_+(\Q)^{'}-1}{(p_l)_+(\Q)^{'}}}  
= &\sigma_l(\Q)^{\frac{1}{(p_l)_+(\Q)}}  \\
= &\sigma_l(\Q)^{\frac{1}{(p_l)_-(\Q)}}
\sigma_l(\Q)^{\frac{1}{(p_l)_+(\Q)}-\frac{1}{(p_l)_-(\Q)}}. 
\end{align*}

Using Lemma~\ref{p.omega} and Lemma~\ref{fracexp},
\begin{align*}
\sigma_l(\Q)^{\frac{1}{(p_l)_+(\Q)}-\frac{1}{(p_l)_-(\Q)}}
&\leq  \|\omega_l^{-\frac{1}{2}}\chi_\Q\|_{2p_l'(\cdot)}
^{[2(p_l')_{-}]\big(\frac{1}{(p_l)_+(\Q)}-\frac{1}{(p_l)_-(\Q)}\big)}\\
&=\|\omega_l^{-\frac{1}{2}}\chi_\Q\|_{2p_l'(\cdot)}
^{[2(p_l')_{-}]\big(1-\frac{1}{(p_l)_+(\Q)^{'}}-1+\frac{1}{(p_l)_-(\Q)^{'}}\big)} \\
&= \|\omega_l^{-\frac{1}{2}}\chi_\Q\|_{2p_l'(\cdot)}
^{[2(p_l')_{-}]\big(\frac{1}{(p_l^{\prime})_+(\Q)}-\frac{1}{(p_l^{\prime})_-(\Q)}\big)} \\
& \le \|\omega_l^{-\frac{1}{2}}\chi_\Q\|_{2p_l'(\cdot)}
^{c[(2p_l^{\prime})_-(\Q)-(2p_l^{\prime})_+(\Q)]}\\
& \lesssim 1.
\end{align*}
Hence, \eqref{EQ-I11estimate} is valid.
This completes the estimate for $I_1$. 
\vspace{0.7cm}

{\bf{Estimate for $I_2$:}}
Initially, we notice that $1, \sigma_l$, and $u$ are in $A_{\infty}$. Considering $\{Q_j^k\}$ as the Calderón-Zygmund dyadic cubes for $f_2$ relative to $\mu$, and selecting a nested tower of cubes $\{Q_{k, 0}\}$, it is observed that the measures $\mu(Q_{k, 0}), \sigma_l(Q_{k, 0})$, and $u(Q_{k, 0})$ all tend towards infinity. We will often use the doubling property for $A_\infty$ in following.

Finding a cube $Q_{k_0, 0} =: Q_0 \in \mathcal{D}_{k_0}$ s.t. $\mu\left(Q_0\right), u\left(Q_0\right) $ and $\sigma\left(Q_0\right)\geq 1$ and fixing a $LH_{\infty}$ base point $x_0=x_c\left(Q_0\right)$, by Lemma \ref{LHxy}. Define $N_0=2 A_0 C_d$ and the sets
\begin{align*}
\mathscr{F} & =\left\{(k, j)\in \mathbb{Z}\times \mathbb{Z}: Q_j^k \subseteq Q_0\right\}; \\
\mathscr{G} & =\left\{(k, j)\in \mathbb{Z}\times \mathbb{Z}: Q_j^k \nsubseteq Q_0 \text { and } d\left(x_0, x_c\left(Q_j^k\right)\right)<N_0 d_0^k\right\}; \\
\mathscr{H} & =\left\{(k, j)\in \mathbb{Z}\times \mathbb{Z}: Q_j^k \nsubseteq Q_0 \text { and } d\left(x_0, x_c\left(Q_j^k\right)\right) \geq N_0 d_0^k\right\}.
\end{align*}

Similar to the estimate for $I_1$, we have
\begin{align}\label{similar_I2}
I_2
\lesssim\sum_{k,j}\int_{E_{j}^{k}}\prod_{l=1,4}
\langle h_{l}\sigma_{\rho(l)}\rangle_{\Q}^{q(x)}{\mu({Q_j^k}) ^{\eta\cdot q(x)}}u(x)\,d \mu \notag   
=\sum_{(k,j)\in\mathscr{F}}
+\sum_{(k,j)\in\mathscr{G}}
+\sum_{(k,j)\in\mathscr{H}}
=: P_1+P_2+P_3.
\end{align}

We estimate each $P_i$ in turn as follows. 

\textbf{Estimate for $P_1$:}
It follows from the fact that $h_{4}\leq 1$ that
\begin{align*}
P_1
& = \sum_{(k,j)\in \mathscr{F}}\int_{E_{j}^{k}}
\prod_{l=1,4}\langle h_{l}\sigma_{\rho(l)}\rangle_{\Q}^{q(x)}{\mu({Q_j^k})^{\eta\cdot q(x)}}u(x)\,d \mu
\\
& \leq \sum_{(k,j)\in\mathscr{F}}\int_{E_{j}^{k}}
\langle h_{1}\sigma_{1}\rangle_{\Q}^{q(x)}\langle
\sigma_{2}\rangle_{\Q}^{q(x)}{\mu({Q_j^k})^{\eta\cdot q(x)}}u(x)\,d \mu \\
& = \sum_{(k,j)\in\mathscr{F}}\int_{E_{j}^{k}}
\bigg(\int_{\Q} h_1\sigma_1d \mu\bigg)^{q(x)}
\sigma_{2}(\Q)^{q(x)-\delta(\Q)}
\sigma_{2}(\Q)^{\delta(\Q)}\mu(\Q)^{(\eta - 2)q(x)}u(x)\,d \mu. \\
\intertext{By \eqref{Bound_function_geq1} and \eqref{EQ-estimateHölder}, the above}
& \leq \sum_{(k,j)\in\mathscr{F}}\int_{E_{j}^{k}}
\bigg(\int_{\Q} h_1\sigma_1d \mu\bigg)^{\delta(\Q)}
\sigma_{2}(\Q)^{q(x)-\delta(\Q)}
\sigma_{2}(\Q)^{\delta(\Q)}\mu(\Q)^{(\eta - 2)q(x)}u(x)\,d \mu \\
& = \sum_{(k,j)\in\mathscr{F}}\int_{E_{j}^{k}}
\langle h_{1}\rangle_{\sigma_{1},\Q}^{\delta(\Q)}
\sigma_{2}(\Q)^{q(x)-\delta(\Q)}\sigma_{1}(\Q)^{\delta(\Q)}\sigma_{2}(\Q)^{\delta(\Q)}
\mu(\Q)^{(\eta - 2)q(x)}u(x)\,d \mu \\
&\leq
\sum_{(k,j)\in\mathscr{F}}\big(\sigma_{2}(\Q)+1\big)^{q_{+}(\Q)-\delta(\Q)}
\langle h_{1}^{\frac{p_{1}(\cdot)}{(p_{1})_{-}}}\rangle_{\sigma_{1},\Q}
^{(p_{1})_{-}\frac{\delta(\Q)}{(p_{1})_{-}(\Q)}}\\
&   \qquad \qquad\times
\int_{E_{j}^{k}}\sigma_{1}(\Q)^{\delta(\Q)}\sigma_{2}(\Q)^{\delta(\Q)}\mu(\Q)^{(\eta - 2)q(x)}u(x)\,d \mu.
\\
\intertext{Define $\delta_- := \mathop {\inf }\limits_Q \delta(Q)$, and it follows from \eqref{EQ-I1estimte} that the above}
&\lesssim 
\big(\sigma_{2}(Q_0)+1\big)^{q_{+}-\delta_-}\sum_{(k,j)\in\mathscr{F}}
\langle h_{1}^{\frac{p_{1}(\cdot)}{(p_{1})_{-}}}
\rangle_{\sigma_{1},\Q}^{(p_{1})_{-}\frac{\delta(\Q)}{(p_{1})_{-}(\Q)}} 
\sigma_{1}(\Q)^{\frac{\delta(\Q)}{(p_{1})_{-}(\Q)}}\sigma_{2}(\Q)^{\frac{\delta(\Q)}{(p_{2})_{-}(\Q)}}.\\
\intertext{Using Young's inequality and Lemma \ref{M.eta.bound}, the above}
&\leq \big(\sigma_{2}(Q_0)+1\big)^{q_{+}-\delta_{-}}\sum_{(k,j)\in
	\mathscr{F}}\left(
\langle h_{1}^{\frac{p_{1}(\cdot)}{(p_{1})_{-}}}\rangle_{\sigma_{1},\Q}^{(p_{1})_{-}}
\sigma_{1}(\Q)+\sigma_{2}(\Q)\right)^{\frac{{\delta (\Q)}}{{\beta (\Q)}}}  \\
&\lesssim  \big(\sigma_{2}(Q_0)+1\big)^{q_{+}-\delta_{-}}\sum_{\theta=1,c}\left(\sum_{(k,j)\in\mathscr{F}}
\left(\langle
h_{1}^{\frac{p_{1}(\cdot)}{(p_{1})_{-}}}\rangle_{\sigma_{1},\Q}^{(p_{1})_{-}}
\sigma_{1}(E_{j}^{k})+\sigma_{2}(E_{j}^{k})\right)\right)^\theta\\
& \lesssim \sum_{\theta=1,c}\left(\sum_{(k,j)\in \mathscr{F}}\int_{E_{j}^{k}}
M_{\sigma_{1}}^{\mathcal{D}}\big(f_{1}^{\frac{p_{1}(\cdot)}{(p_{1})_{-}}}\big)(x)
^{(p_{1})_{-}}\sigma_{1}(x)\,d \mu
+ \sum_{(k,j)\in \mathscr{F}} \sigma_{2}(E_{j}^{k})\right)^\theta\\
& \lesssim \sum_{\theta=1,c}\left(\int_{\subRn} f_{1}(x)^{p_{1}(x)}\sigma_{1}(x)\,d \mu
+ \sigma_{2}(Q_0)\right)^\theta\\
& \lesssim 1,
\end{align*}
where $c={c_{\eta,p_1(\cdot),p_2(\cdot)}}\ge1$, and it is independent  on $\Q$.

\textbf{Estimate for $P_2$:} 
Set $B_j^k= B(x_c(Q_j^k), A_0(C_d+1) N_0 d_0^k)$. For $(k,j) \in \mathscr{G}$, as $Q^k_j \not\subseteq Q_0$, if $x_c(Q^k_j) \in Q_0$, then by Lemma \ref{cubes},
$Q_0 \subseteq Q^k_j \subseteq B^k_j.$
If $x_c(Q^k_j) \not\in Q_0$, noting that $Q_0 \supseteq B(x_0,d_0^{k_0})$, we have
\[ d_0^{k_0} \leq d(x_0,x_c(Q^k_j)) \leq N_0d_0^k. \]
By Lemma \ref{cubes} again, since $x_0 \in B(x_c(Q^k_j),N_0d_0^k)$ and $Q_0 \subseteq B(x_0,C_dd_0^{k_0})$, then for every $x \in Q_0,$
\[ d(x,x_c(Q^k_j)) \leq A_0(d(x,x_0)+d(x_0,x_c(Q^k_j))) \leq A_0(C_dd_0^{k_0}+N_0d_0^k) \leq A_0(C_d+1)N_0d_0^k. \]
Hence, for any $(k,j) \in \mathscr{G}$, $Q_0 \subseteq B^k_j$. Furthermore, $u(B^k_j), \sigma_2(B^k_j) \geq 1$. Note also that by doubling property and Lemma \ref{cubes}, $\mu(Q^k_j) \approx \mu(B^k_j)$, $u(Q^k_j) \approx u(B^k_j)$, and $\sigma_l(\Q) \approx \sigma_l(B^k_j)$ for $l=1,2$. 

Firstly, we claim that the followng facts are both valid.
\begin{align}\label{cd1w}
    \frac{1}{\mu(\Q)^{2-\eta}}\int_{\Q} h_{4}(y)\sigma_{2}(y)\,d \mu \leq c_0
\end{align}
and
\begin{equation}\label{P_infty_Bounded}
    \sigma_{1}(\Q)^{q_{\infty}}\sigma_{2}(\Q)^{q_{\infty}}
    \mu(\Q)^{\left( {\eta - 2} \right)q_{\infty}}u(E_{j}^{k})
    \lesssim \sigma_{1}(\Q)^{\frac{q_{\infty}}{(p_{1})_{\infty}}}
    \sigma_{2}(\Q)^{\frac{q_{\infty}}{(p_{2})_{\infty}}}.
    \end{equation}

We now can  estimate $P_2$.  Using Lemmas~\ref{lemma:p-infty-px} and~\ref{lemma:infty-bound} with \eqref{cd1w} and \eqref{Bound_function_geq1}, there
exists $t>1$ such that
\begin{align*}
P_2 
& =
\sum_{(k,j)\in \mathscr{G}}\int_{E_{j}^{k}}
c_0^{q(x)} \bigg(\int_{\Q}{h_{1}\sigma_{1}d \mu}\bigg)^{q(x)}
\bigg(\frac{c_0^{-1}}{\mu(\Q)^{2 - \eta}}
\int_{\Q}h_{4}\sigma_{2}d \mu\bigg)^{q(x)}u(x)\,d \mu\\
& \lesssim \sum_{(k,j)\in\mathscr{G}} c_0^{q_+}\int_{E_{j}^{k}}
{\bigg(\int_{\Q}h_{1}\sigma_{1}d \mu\bigg)^{q_{\infty}}
	\bigg(\frac{c_0^{-1}}{\mu(\Q)^{{2 - \eta}}}
	\int_{\Q}h_{4}\sigma_{2}d \mu}\bigg)^{q_{\infty}}u(x)\,d \mu\\
& \qquad \qquad \qquad 
+ \sum_{(k,j)\in \mathscr{G}}\int_{E_{j}^{k}}
\frac{u(x)}{(e+d(x,x_0))^{tq_{-}}}\,d \mu \\
&\lesssim \sum_{(k,j)\in \mathscr{G}}
\prod_{l=1,4}\langle h_{l}\rangle_{\sigma_{\rho(l),\Q}}^{q_{\infty}}
\sigma_{1}(\Q)^{q_{\infty}}
\sigma_{2}(\Q)^{q_{\infty}}\mu(\Q)^{\left( {\eta - 2} \right)q_{\infty}}u(E_{j}^{k})+1\\
& \lesssim \left(\sum_{(k,j)\in \mathscr{G}}
\prod_{l=1,4}\langle h_{l}\rangle_{\sigma_{\rho(l)},\Q}^{q_{\infty}}
\sigma_{\rho(l)}(\Q)^{\frac{q_{\infty}}{(p_{\rho(l)})_{\infty}}} + 1\right).
\end{align*}
The last inequality is obtained by \eqref{P_infty_Bounded}.

Secondly, by Lemma~\ref{Gu6} and again by
Lemma~\ref{lemma:p-infty-cond}, we have that
\begin{multline} \label{sigma1_Averange_Bound}
\frac{1}{\sigma_{1}(\Q)}\int_{\Q}{h_{1}(y)\sigma_{1}(y)\,d \mu}
\lesssim 
\sigma_{1}(\Q)^{-1}\|h_{1}\|_{L_{\sigma_{1}}^{p_{1}(\cdot)}}
\|\chi_{\Q}\|_{L_{\sigma_{1}}^{p_{1}^{\prime}(\cdot)}} \\
\leq
\sigma_{1}(\Q)^{-1}\|f\|_{L_{\sigma_{1}}^{p_{1}(\cdot)}} 
\|\omega_{1}^{-1}\chi_{\Q}\|_{\cpap}
\lesssim  \sigma_{1}(\Q)^{\frac{1}{(p_{1}^{\prime})_{\infty}}-1}
\leq \sigma_{1}(Q_0)^{-\frac{1}{(p_{1})_{\infty}}}
\lesssim 1 .
\end{multline}

Due to the fact that 
$\frac{1}{(p_{\infty})}=\frac{1}{(p_{1})_{\infty}}+\frac{1}{(p_{2})_{\infty}},$ it follows from the Young's inequality that
\begin{align}\label{eqn:J2-final-est}
 P_2
 & \lesssim \sum_{(k,j)\in \mathscr{G}}
\left(\langle h_{1}\rangle_{\sigma_{1},\Q}^{(p_{1})_{\infty}}
\sigma_{1}(\Q)\right)^{\frac{{{q_\infty }}}{{{p_\infty }}}}
+ \sum_{(k,j)\in \mathscr{G}}
\left(\langle h_{4}\rangle_{\sigma_{2},\Q}^{(p_{2})_{\infty}}\sigma_{2}(\Q)\right)^{\frac{{{q_\infty }}}{{{p_\infty }}}}
+1;\\
\label{eqn:J2-final-est-2} & \lesssim \left(\sum_{(k,j)\in \mathscr{G}}
\langle h_{1}\rangle_{\sigma_{1},\Q}^{(p_{1})_{\infty}}
\sigma_{1}(E_{j}^{k})
+ \sum_{(k,j)\in \mathscr{G}}
\langle h_{4}\rangle_{\sigma_{2},\Q}^{(p_{2})_{\infty}}
\sigma_{2}(E_{j}^{k}) +1\right)^{\frac{{{q_\infty }}}{{{p_\infty }}}}. \\
\intertext{Due to Lemmas \ref{lemma:p-infty-px}
	and \ref{lemma:infty-bound}, there exists $t>1$ such that the items in parentheses above are bounded by}
\nonumber  & \sum_{(k,j)\in \mathscr{G}}\int_{E_{j}^{k}}
\langle h_{1}\rangle_{\sigma_{1},\Q}^{p_{1}(x)}
\sigma_{1}(x)\,d \mu
+ \sum_{(k,j)\in \mathscr{G}}\int_{E_{j}^{k}}
\frac{\sigma_{1}(x)}{(e+d(x,x_0))^{t(p_{1})_{-}}}\,d \mu\\ 
\nonumber    & \qquad \qquad \qquad 
+\sum_{(k,j)\in \mathscr{G}}
\int_{E_{j}^{k}}{M}_{\sigma_{2}}^{\mathcal{D}}h_{4}(x)^{(p_{2})_{\infty}}
\sigma_{2}(x)\,d \mu +1\\
\nonumber  & \lesssim \sum_{(k,j)\in \mathscr{G}}\int_{E_{j}^{k}}
\langle
h_{1}^{\frac{p_{1}(\cdot)}{(p_{1})_{-}}}\rangle_{\sigma_{1},\Q}^{(p_{1})_{-}}
\sigma_{1}(x)\,d \mu
+ \int_{\subRn}{M}_{\sigma_{2}}^{\mathcal{D}}h_{4}(x)^{(p_{2})_{\infty}}
\sigma_{2}(x)\,d \mu + 1.\\
\intertext{We can use Lemma \ref{M.eta.bound} and then the above is bounded by} 
\nonumber  & \int_{\subRn}
{M}_{\sigma_{1}}^{\mathcal{D}}(h_{1}^{\frac{p_{1}(\cdot)}{(p_{1})_{-}}})(x)^{(p_{1})_{-}}
\sigma_{1}(x)\,d \mu
+ \int_{\subRn}h_{4}(x)^{(p_{2})_{\infty}}\sigma_{2}(x)\,d \mu + 1\\
\nonumber &\lesssim \int_{\subRn} h_{1}(x)^{p_{1}(x)}\sigma_{1}(x)\,d \mu
+ \int_{\subRn}h_{4}(x)^{(p_{2})_{\infty}}\sigma_{2}(x)\,d \mu + 1.\\    
\intertext{It follows instantly from Lemmas \ref{lemma:p-infty-px} and \ref{lemma:infty-bound} that the above is controlled by} 
\nonumber & \int_{\subRn}h_{4}(x)^{p_{2}(x)}\sigma_{2}(x)\,d \mu
+\int_{\subRn}\frac{\sigma_{2}(x)}{(e+d(x,x_0))^{t(p_{2})_{-}}}\,d \mu +1\nonumber \lesssim 1.
\end{align}
Last, the estimate for $P_2$ is finished.

We now prove the previous two facts as follows.

We prove \eqref{cd1w} as follows. By
Lemma \ref{lemma:p-infty-cond} for $\omega_2^{-\frac{1}{2}}\in
A_{2p_2'(\cdot)}$, it follows that

\begin{align*}
    \frac{1}{\mu(Q_j^k)} \approx \frac{1}{\mu(B_j^k)}
\leq \frac{\mu(Q_0)}{\mu(B_j^k)}
&\lesssim \bigg(\frac{\sigma_{2}(Q_0)}
{\sigma_{2}(B_j^k)}\bigg)^{\frac{1}{(1 - \eta)2(p_{2}^{\prime})_{\infty}}} \lesssim \left(\sigma_{2}(B^k_j)^{-\frac{1}{2(p_{2}^{\prime})_{\infty}}}\right)^\frac{1}{1-\eta}\\
&\approx \left(\|\omega_{2}^{-\frac{1}{2}}\chi_{B_j^k}\|_{2\cpbp}^{-1}\right)^{\frac{1}{1-\eta}}
= \left(\|\omega_{2}^{-1}\chi_{B_j^k}\|_{\cpbp}^{-\frac{1}{2}}\right)^{\frac{1}{1-\eta}}.
\end{align*}

Since $\sigma(B^k_j) \geq 1$, by  Lemma \ref{p.omega}, it is easy to see that $\|\omega_{2}^{-1}\chi_{B_j^k}\|_{\cpbp}^{-1} \leq 1$ . We have that
\begin{align}\label{cd2w}
    \frac{1}{\mu(Q_j^k)} \lesssim \left(\|\omega_{2}^{-1}\chi_{B_j^k}\|_{\cpbp}^{-\frac{1}{2}}\right)^{\frac{1}{1-\eta}}  \leq \|\omega_{2}^{-1}\chi_{B_j^k}\|_{\cpbp}^{-\frac{1}{2}} \leq 1
\end{align}
Hence, by Lemma~\ref{Gu6},
\begin{align*}
\frac{1}{\mu(\Q)^{2}}\int_{\Q} h_{4}(y)\sigma_{2}(y)\,d\mu
& \lesssim \|\omega_{2}^{-1}\chi_{\Q}\|_{\cpbp}^{-1}
\int_{\Q} h_{4}(y)\sigma_{2}(y)^{\frac{1}{p_2(y)}}\sigma_{2}(y)^{\frac{1}{p_2'(y)}}\,d\mu\\
& \lesssim \|\omega_{2}^{-1}\chi_{\Q}\|_{\cpbp}^{-1}
\|h_{4}\|_{L_{\sigma_{2}}^{p_{2}(\cdot)}}
\|\chi_{\Q}\|_{L_{\sigma_{2}}^{p_{2}^{\prime}(\cdot)}}\\
& \leq \|\omega_{2}^{-1}\chi_{\Q}\|_{\cpbp}^{-1}
\|f_{2}\|_{L_{\sigma_{2}}^{p_{2}(\cdot)}}
\|\omega_{2}^{-1}\chi_{\Q}\|_{\cpbp}\\
& \leq c_0.
\end{align*}
Therefore, by \eqref{cd2w} and the fact that $\eta \in [0,2)$, \eqref{cd1w} is valid.

Finally, we prove \eqref{P_infty_Bounded} as follows. Since $1 \lesssim \sigma_l(\Q),\, u(\Q) $, it is derived from Lemma \ref{lemma:p-infty-cond} and the definition of $A_{\vec{p}(\cdot),q(\cdot)}(X)$ that

\begin{align*}  
\bigg[\sigma_{1}(\Q)\sigma_{2}(\Q)\bigg]^{q_{\infty}}
& \lesssim 
\bigg( \|\omega_{1}^{-\frac{1}{2}}\chi_{\Q}\|_{2\cpap}^{2(p_{1}^{\prime})_{\infty}}
\|\omega_{2}^{-\frac{1}{2}}\chi_{\Q}\|_{2\cpbp}^{2(p_{2}^{\prime})_{\infty}}\bigg)^{q_{\infty}} \\
& = \bigg(\|\omega_{1}^{-1}\chi_{\Q}\|_{\cpap}^{(p_{1}^{\prime})_{\infty}-1}
\|\omega_{2}^{-1}\chi_{\Q}\|_{\cpbp}^{(p_{2}^{\prime})_{\infty}-1}\bigg)^{q_{\infty}}
\bigg(\prod_{l=1}^{2} \|\omega_{l}^{-1}\chi_{\Q}\|_{p_{l}^{\prime}(\cdot)}
\bigg)^{q_{\infty}}\\
&\lesssim \bigg( \|\omega_{1}^{-1}\chi_{\Q}\|_{\cpap}^{(p_{1}^{\prime})_{\infty}-1}
\|\omega_{2}^{-1}\chi_{\Q}\|_{\cpbp}^{(p_{2}^{\prime})_{\infty}-1}\bigg)^{q_{\infty}}
\bigg(\frac{\mu(\Q)^{2 - \eta}}{\|\omega\chi_{\Q}\|_{\qq}}\bigg)^{q_{\infty}}\\
& \lesssim \bigg(
\sigma_{1}(\Q)^{\frac{(p_{1}^{\prime})_{\infty}-1}{(p_{1}^{\prime})_{\infty}}}
\sigma_{2}(\Q)^{\frac{(p_{2}^{\prime})_{\infty}-1}{(p_{2}^{\prime})_{\infty}}}
\bigg)^{q_{\infty}}\frac{\mu(\Q)^{(2 - \eta)q_{\infty}}}{u(\Q)}\\
&\leq \sigma_{1}(\Q)^{\frac{q_{\infty}}{(p_{1})_{\infty}}}
\sigma_{2}(\Q)^{\frac{q_{\infty}}{(p_{2})_{\infty}}}
\frac{\mu(\Q)^{(2 - \eta)q_{\infty}}}{u(E_{j}^{k})}.
\end{align*}
Thus, \eqref{P_infty_Bounded} is valid by rearrangement.

{\bf{Estimate for $P_3$:}} Firstly, we claim that
\begin{equation}\label{Suff.ineq_16}
    \sup _{x \in Q_j^k} d\left(x_0, x\right) \approx \inf _{x \in Q_j^k} d\left(x_0, x\right),
\end{equation}
where the implicit constant is independent on $Q_j^k$. In our analysis, the validity of inequality \eqref{Suff.ineq_16} will be established through substitution of \( Q_j^k \) with the ball \( A_j^k = N_0^{-1}B_j^k \), which encompasses \( Q_j^k \). For this purpose, we fix a pair \( (k, j) \) within \( \mathscr{H} \) and choose an arbitrary \( x \) from \( A_j^k \). We get that
\begin{align*} d(x,x_0) \leq A_0[d(x,x_c(Q^k_j))+d(x_0,x_c(Q^k_j))] 
\leq& A_0[C_dd_0^k+d(x_0,x_c(Q^k_j))]\\
\leq& \left(A_0+\frac{1}{2}\right)d(x_0,x_c(Q^k_j)). 
\end{align*}
Meanwhile, we have
\begin{align*} d(x_0,x_c(Q^k_j)) \leq A_0[d(x_0,x)+d(x,x_c(Q^k_j))] =& \frac{1}{2}N_0d_0^k + A_0d(x_0,x)\\ 
\leq& \frac{1}{2}d(x_0,x_c(Q^k_j)) + A_0d(x_0,x).
\end{align*}
It follows that
$$
d\left(x_0, x_c\left(Q_j^k\right)\right) \leq 2 A_0 d\left(x_0, x\right) .
$$
Consequently, \eqref{Suff.ineq_16} holds. 

To proceed with the estimate for \( P_3 \), it becomes necessary to partition \( \mathscr{H} \) into two distinct subsets,
$$
\mathscr{H}_1=\left\{(k, j)\in \mathscr{H}: \sigma_2\left(Q_j^k\right) \leq 1\right\}, \quad \mathscr{H}_2=\left\{(k, j)\in \mathscr{H}: \sigma_2\left(Q_j^k\right)>1\right\} .
$$

Initially, we aggregate over \( \mathscr{H}_1 \). Consider \( x_{+} \) within \( \overline{A_j^k} \), chosen such that \( q_{+}(A_j^k) = q(x_{+}) \), a selection made possible by the continuity of \( q(\cdot) \) in \( LH_0 \). Subsequently, in accordance with the \( LH_{\infty} \) criterion and inequality \eqref{Suff.ineq_16}, it holds for almost every \( x \) in \( Q_j^k \) that
\begin{align}
0&\leq  q_{+}\left(Q_j^k\right)-q(x) \leq\left|q\left(x_{+}\right)-q_{\infty}\right|+\left|q(x)-q_{\infty}\right| \notag\\ 
& \leq \frac{C_{\infty}}{\log \left(e+d\left(x_0, x_{+}\right)\right)}+\frac{C_{\infty}}{\log \left(e+d\left(x_0, x\right)\right)} 
 \approx\frac{1}{\log \left(e+d\left(x_0, x\right)\right)}.\label{eqn:p-J3-est}
\end{align}
Last, in the same way, for $l=1,2$ we have that $p_l(\cdot)$ satisfies
\begin{equation} \label{eqn:p2-J3-est}
|(p_{l})_{-}(\Q)-p_{l}(x)|
\lesssim \frac{1}{\log(e+d(x,x_0))}.
\end{equation}
The corresponding sums over $\mathscr{H}_1$ and $\mathscr{H}_2$ are defined as $P_{31}$ and $P_{32}$ respectively.

It follows obviously from Lemmas \ref{lemma:p-infty-px}, \ref{lemma:infty-bound}, and \eqref{eqn:p-J3-est} that
\begin{align*}
&P_{31}:=\sum_{(k,j)\in \mathscr{H}_{1}} \int_{E_{j}^{k}}
\prod_{l=1,4}\langle h_{l}\sigma_{\rho(l)} \rangle_{\Q}^{q(x)}
{\mu({Q_j^k})^{\eta \cdot q(x)}}u(x)\,d \mu \\
& \qquad \lesssim \sum_{(k,j)\in \mathscr{H}_{1}} \int_{E_{j}^{k}}
\prod_{l=1,4}\langle h_{l}\sigma_{\rho(l)}
\rangle_{\Q}^{q_{+}(\Q)} {\mu({Q_j^k})^{\eta \cdot {q_ + }(Q_j^k)}}u(x)\,d \mu
+\sum_{(k,j)\in \mathscr{H}_{1}}\int_{E_{j}^{k}}\frac{u(x)}{(e+d(x,x_0))^{tq_{-}}}\,d \mu\\
& \qquad  \leq \sum_{(k,j)\in \mathscr{H}_{1}}
\int_{E_{j}^{k}}\prod_{l=1,4}\langle h_{l}\sigma_{\rho(l)}
\rangle_{\Q}^{q_{+}(\Q)} {\mu({Q_j^k})^{\eta \cdot {q_ + }(Q_j^k)}}u(x)\,d \mu + 1.
\intertext{Note that the properties of $h_1$ and $h_4$ (see the before), and $\sigma_2(\Q)\leq 1$. It follows immediately from Lemma \ref{lemma:diening}, \eqref{Bound_function_geq1}, and \eqref{EQ-I1estimte} that the above}
&\qquad = \sum_{(k,j)\in \mathscr{H}_{1}} \int_{E_{j}^{k}}
\bigg(\int_{\Q}{h_{1}\sigma_{1}d \mu}\bigg)^{q_{+}(\Q)}
\bigg(\frac{1}{\sigma_{2}(\Q)}\int_{\Q}{h_{4}\sigma_{2}d \mu}\bigg)^{q_{+}(\Q)}\\
& \qquad \qquad \qquad \qquad \times
\mu(\Q)^{(\eta - 2)q_{+}(\Q)}\sigma_{2}(\Q)^{q_{+}(\Q)}u(x)\,d\mu +1; \\
& \qquad \lesssim \sum_{(k,j)\in \mathscr{H}_{1}} 
\langle h_{1} \rangle_{\sigma_{1},\Q}^{\delta(\Q)}
\langle h_{4}\rangle_{\sigma_{2},\Q}^{\delta(\Q)}
\int_{E_{j}^{k}}\mu(\Q)^{(\eta - 2)q(x)}
\sigma_{1}(\Q)^{\delta(\Q)}\sigma_{2}(\Q)^{\delta(\Q)} u(x)\,d\mu + 1\\
& \qquad \lesssim  \left(\sum_{(k,j)\in \mathscr{H}_{1}}
\left(\prod_{l=1,4}\langle h_{l}\rangle_{\sigma_{\rho(l),\Q}}^{\delta(\Q)}\right)
\sigma_{1}(\Q)^{\frac{\delta(\Q)}{(p_{1})_{-}(\Q)}}\sigma_{2}(\Q)^{\frac{\delta(\Q)}{(p_{2})_{-}(\Q)}}
+ 1\right)\\
&\qquad \leq  \left(\sum_{(k,j)\in \mathscr{H}_{1}}
\langle   h_{1}^{\frac{p_{1}(\cdot)}{(p_{1})_{-}(\Q)}} \rangle_{\sigma_{1},\Q}^{\delta(\Q)}
\sigma_{1}(\Q)^{\frac{\delta(\Q)}{(p_{1})_{-}(\Q)}}
\langle h_{4}\rangle_{\sigma_{2},\Q}^{\delta(\Q)}
\sigma_{2}(\Q)^{\frac{\delta(\Q)}{(p_{2})_{-}(\Q)}} + 1\right).\\
\intertext{Reusing the H\"older's inequality and Young's inequality, the above}
&\qquad \leq \left(\sum_{(k,j)\in \mathscr{H}_{1}}
\langle h_{1}^{\frac{p_{1}(\cdot)}{(p_{1})_{-}}}
\rangle_{\sigma_{1},\Q}^{(p_{1})_{-}\frac{\delta(\Q)}{(p_{1})_{-}(\Q)}}
\sigma_{1}(\Q)^{\frac{\delta(\Q)}{(p_{1})_{-}(\Q)}}
\langle h_{4}\rangle_{\sigma_{2},\Q}^{\delta(\Q)}
\sigma_{2}(\Q)^{\frac{\delta(\Q)}{(p_{2})_{-}(\Q)}} +1\right)\\
& \qquad \lesssim \left(\sum_{\theta=1,c}\left(\sum_{(k,j)\in \mathscr{H}_{1}}
\langle
h_{1}^{\frac{p_{1}(\cdot)}{(p_{1})_{-}}}\rangle_{\sigma_{1},\Q}^{(p_{1})_{-}}
\sigma_{1}(\Q)
+ \sum_{(k,j)\in \mathscr{H}_{1}}
\langle h_{4}\rangle_{\sigma_{2},\Q}^{(p_{2})_{-}(\Q)}
\sigma_{2}(\Q) \right)^{\theta}+ 1\right)\\
& \qquad =: \sum_{\theta=1,c}\left(P_{311}+P_{312}\right)^{\theta}+1,
\end{align*}
where $c={c_{\eta,p_1(\cdot),p_2(\cdot)}}\ge1$, and it is independent  on $\Q$.

On the one hand, the estimate for $P_{311}$ is similar to the estimate for $I_1$, we can indeed do as \eqref{eqn:final-I1-est}.
On the other hand, it follows from Lemmas \ref{lemma:p-infty-px}, \ref{lemma:infty-bound}, and \ref{M.eta.bound} that 

\begin{align*}
P_{312}
& \lesssim  \sum_{(k,j)\in \mathscr{H}_{1}} \int_{E_{j}^{k}}
\langle
h_{4}\rangle_{\sigma_{2},\Q}^{(p_{2})_{-}(\Q)}\sigma_{2}(x)\,d \mu \\
&\lesssim \sum_{(k,j)\in \mathscr{H}_{1}}
\int_{E_{j}^{k}}\langle
h_{4}\rangle_{\sigma_{2},\Q}^{(p_{2})_{\infty}}
\sigma_{2}(x)\,d \mu
+ \sum_{(k,j)\in \mathscr{H}_{1}}
\int_{E_{j}^{k}}\frac{\sigma_{2}(x)}{(e+d(x,x_0))^{t(p_{2})_{-}}}\,d \mu \\
&\leq \int_{\subRn}
{M}_{\sigma_{2}}^{\mathcal{D}}h_{4}(x)^{(p_{2})_{\infty}}\sigma_{2}(x)\,d \mu
+
\int_{\subRn}\frac{\sigma_{2}(x)}{(e+d(x,x_0))^{t(p_{2})_{-}}}\,d \mu\\
&\lesssim \int_{\subRn}h_{4}(x)^{(p_{2})_{\infty}}\sigma_{2}(x)\,d \mu + 1\\
& \lesssim \int_{\subRn}h_{4}(x)^{p_{2}(x)}\sigma_{2}(x)\,d \mu
+\int_{\subRn}\frac{\sigma_{2}(x)}{(e+d(x,x_0))^{t(p_{2})_{-}}}\,d \mu + 1\\
& \lesssim 1.
\end{align*}
Thus, the estimate for $P_{31}$ is accomplished.

For estimating $P_{32}$, it is obtained from Lemma \ref{Gu6} that 
\begin{align}
\int_{\Q}{h_{1}\sigma_{1}\,d \mu}\lesssim \|h_{1}\|_{L_{\sigma_{1}}^{\pap}}
\|\chi_{\Q}\|_{L_{\sigma_{1}}^{\cpap}} \label{H2-first-estimate}
\lesssim \|f_{1}\|_{L_{\sigma_{1}}^{\pap}}
\|\omega_{1}^{-1}\chi_{\Q}\|_{\cpap}
\le \|\omega_{1}^{-1}\chi_{\Q}\|_{\cpap}
\end{align}
By the similar methods, we get
\begin{equation}\label{H2-second-estimate}
\int_{\Q}{h_{4}\sigma_{2}\,d \mu}
\leq c_0\|\omega_{2}^{-1}\chi_{\Q}\|_{\cpbp}.
\end{equation}

We next divide: $\mathscr{H}_2$=$\mathscr{H}_{21}  \cup\mathscr{H}_{22}$, where
\[ \mathscr{H}_{21} = \{ (k,j)\in \mathscr{H}_2 : \sigma_1(\Q) \geq 1 \},
\quad
\mathscr{H}_{22} = \{ (k,j)\in \mathscr{H}_2 : \sigma_1(\Q)< 1 \}.  \]
The corresponding sums over $\mathscr{H}_{21}$ and $\mathscr{H}_{22}$ are defined as $P_{321}$ and $P_{322}$ respectively.
It is obtained from \eqref{H2-first-estimate}, \eqref{H2-second-estimate}, and Lemma~\ref{lemma:p-infty-px} that
\begin{align*}
P_{321}:=&\sum_{(k,j)\in \mathscr{H}_{21}} \int_{E_{j}^{k}}\prod_{l=1,4}
\langle h_{l}\sigma_{\rho(l)}\rangle_{\Q}^{q(x)}{\mu ({Q_j^k})^{\eta \cdot q(x)}}u(x)\,d \mu \\
\lesssim&\sum_{(k,j)\in \mathscr{H}_{21}}
\int_{E_{j}^{k}} \prod_{l=1,4}
\bigg(c_0^{-1}\|\omega_{\rho(l)}^{-1}\chi_{\Q}\|_{p_{\rho(l)}^{\prime}(\cdot)}^{-1}
\int_{\Q}h_{l}\sigma_{\rho(l)}d \mu \bigg)^{q(x)}\prod_{l=1}^{2}
\bigg(\frac{\|\omega_{l}^{-1}\chi_{\Q}\|_{p_{l}^{\prime}(\cdot)}}{\mu(\Q)^{1 - \frac{\eta}{{2}}}}\bigg)^{q(x)}u(x)\,d \mu\\
 \lesssim& \sum_{(k,j)\in
	\mathscr{H}_{21}}\int_{E_{j}^{k}}
\prod_{l=1,4}\bigg(c_0^{-1}\|\omega_{\rho(l)}^{-1}\chi_{\Q}\|_{p_{\rho(l)}^{\prime}(\cdot)}^{-1}
\int_{\Q}h_{l}\sigma_{\rho(l)}d \mu\bigg)^{q_\infty} 
\prod_{l=1}^{2}\bigg(\frac{\|\omega_{l}^{-1}\chi_{\Q}\|_{p_{l}^{\prime}(\cdot)}}
{\mu(\Q)^{1 - \frac{\eta}{{2}}}}\bigg)^{q(x)}u(x)\,d \mu \\
&  + 
\sum_{(k,j)\in \mathscr{H}_{21}} \int_{E_{j}^{k}} \prod_{l=1}^{2}
\bigg(\frac{\|\omega_{l}^{-1}\chi_{\Q}\|_{p_{l}^{\prime}(\cdot)}}{\mu(\Q)^{1 - \frac{\eta}{{2}}}}\bigg)^{q(x)}
\frac{u(x)}{(e+d(x,x_0))^{tq_{-}}}\,d \mu \\
=:&P_{3211}+ P_{3212}.
\end{align*}

Due to the fact that $1 \le \sigma_{2}(\Q) \approx  \sigma_2(E_j^k)$, it is acquired from \eqref{Suff.ineq_16}, \eqref{eqn:modular-Ap}, and Lemma~\ref{lemma:infty-bound} that

\begin{align*}
P_{3212}
& \leq \sum_{(k,j)\in \mathscr{H}_{21}}\sup_{x\in
	\Q}(e+d(x,x_0))^{-tq_{-}}
\int_{\Q}{\prod_{l=1}^{2}\|\omega_{l}^{-1}\chi_{\Q}\|_{p_{l}^{\prime}(\cdot)}^{q(x)}
	\mu(\Q)^{(\eta - 2)q(x)}u(x)\,d \mu}\\
& \lesssim \sum_{(k,j)\in \mathscr{H}_{21}}\inf_{x\in \Q}(e+d(x,x_0))^{-tq_{-}}\sigma_{2}(E_j^k)\\
& \lesssim \int_{\subRn}\frac{\sigma_{2}(x)}{(e+d(x,x_0))^{tq_{-}}}\,d \mu \lesssim 1.
\end{align*}

Next, for estimating $P_{3211}$, by the fact that for $l=1,2$, $\sigma_l(\Q)\geq 1$, it follows from Lemma \ref{lemma:p-infty-cond} that

\begin{equation} \label{eqn:p-infty-Hest}
\frac{\sigma_{l}(\Q)}
{\|\omega_{l}^{-1}\chi_{\Q}\|_{p_l'(\cdot)}}
\lesssim\frac{\sigma_{l}(\Q)}
{\sigma_{l}(\Q)^{\frac{1}{(p_{l}^{\prime})_{\infty}}}}
=  \sigma_{l}(\Q)^{\frac{1}{(p_{l})_{\infty}}}. 
\end{equation}
By~\eqref{eqn:modular-Ap}, \eqref{eqn:p-infty-Hest}, and Young's inequality,     
\begin{align*}
P_{3211}&\lesssim \sum_{(k,j)\in
	\mathscr{H}_{21}}\int_{E_{j}^{k}}{\prod_{l=1,4}
	\langle h_{l} \rangle_{\sigma_{\rho(l)},\Q}^{q_{\infty}}
	\sigma_{1}(\Q)^{\frac{q_{\infty}}{(p_{1})_{\infty}}}
	\sigma_{2}(\Q)^{\frac{q_{\infty}}{(p_{2})_{\infty}}}}
\prod_{l=1}^{2}\bigg(\frac{\|\omega_{l}^{-1}\chi_{\Q}
	\|_{p_{l}^{\prime}(\cdot)}}{\mu(\Q)^{1-\frac{\eta}{2}}}\bigg)^{q(x)}u(x)\,d \mu \\
& \leq \sum_{(k,j)\in \mathscr{H}_{21}}
\prod_{l=1,4} \langle h_{l} \rangle_{\sigma_{\rho(l)},\Q}^{q_{\infty}}
\sigma_{1}(\Q)^{\frac{q_{\infty}}{(p_{1})_{\infty}}}
\sigma_{2}(\Q)^{\frac{q_{\infty}}{(p_{2})_{\infty}}}
\int_{\Q}{\prod_{l=1}^{2}\|\omega_{l}^{-1}\chi_{\Q} 
	\|_{p_{l}^{\prime}(\cdot)}^{q(x)}\mu(\Q)^{(\eta - 2)q(x)}u(x)\,d \mu}\\
&\lesssim \left(\sum_{(k,j)\in \mathscr{H}_{21}}
\langle h_{1}\rangle_{\sigma_{1},\Q}^{(p_{1})_{\infty}}\sigma_{1}(\Q)
+\sum_{(k,j)\in \mathscr{H}_{21}}\langle
h_{4}\rangle_{\sigma_{2},\Q}^{(p_{2})_{\infty}}
\sigma_{2}(\Q)\right)^{\frac{{{q_\infty }}}{{{p_\infty }}}}.\\
\end{align*}     
The estimate for the above is identical to the estimate for $P_2$. Here we have to explain that the fact that $\sigma_1(\Q)\geq 1$ can make us obtain \eqref{sigma1_Averange_Bound}, then we can refer to the inequality \eqref{eqn:J2-final-est} to get estimate for the above.
Eventually, we can get that $P_{321} \lesssim 1.$

Now, it is $P_{322}$'s turn. This estimate is similar to before. However, we may need to replace the exponent $q_\infty$ with $\delta(\Q)$.
It follows from \eqref{eqn:p2-J3-est} that for $x\in \Q$,
\[ \bigg|\frac{1}{q(x)}-\frac{1}{\delta(\Q)}\bigg|
\leq \bigg|\frac{1}{p_1(x)}-\frac{1}{(p_1)_-(\Q)}\bigg|
+ \bigg|\frac{1}{p_2(x)}-\frac{1}{(p_2)_-(\Q)}\bigg| 
\lesssim \frac{1}{\log(e+d(x,x_0))}.  \]

It follows instantly from the estimate for $P_{321}$ that 
\begin{align*}
&P_{322}:=\sum_{(k,j)\in \mathscr{H}_{22}} \int_{E_{j}^{k}}\prod_{l=1,4}
\langle h_{l}\sigma_{\rho(l)}\rangle_{\Q}^{q(x)}{\mu(\Q)^{\eta \cdot q(x)}}u(x)\,d \mu \\
  \lesssim& \sum_{(k,j)\in
	\mathscr{H}_{22}}\int_{E_{j}^{k}}
\prod_{l=1,4}\bigg(c_0^{-1}\|\omega_{\rho(l)}^{-1}\chi_{\Q}\|_{p_{\rho(l)}^{\prime}(\cdot)}^{-1}
\int_{\Q}h_{l}\sigma_{\rho(l)}d \mu\bigg)^{\delta(\Q)} 
\prod_{l=1}^{2}\bigg(\frac{\|\omega_{l}^{-1}\chi_{\Q}\|_{p_{l}^{\prime}(\cdot)}}
{\mu(\Q)^{1 - \frac{\eta}{{2}}}}\bigg)^{q(x)}u(x)\,d \mu \\
& + 
\sum_{(k,j)\in \mathscr{H}_{22}} \int_{E_{j}^{k}} \prod_{l=1}^{2}
\bigg(\frac{\|\omega_{l}^{-1}\chi_{\Q}\|_{p_{l}^{\prime}(\cdot)}}{\mu(\Q)^{1 - \frac{\eta}{{2}}}}\bigg)^{q(x)}
\frac{u(x)}{(e+d(x,x_0))^{tq_{-}}}\,d \mu \\
 =:&P_{3221}+ P_{3222}.
\end{align*}

On the one hand, the estimate for $P_{3222}$ is same as the estimate for $P_{3212}$.  On the other hand, for
estimating $P_{3221}$, by \eqref{EQ-I11estimate}, \eqref{eqn:modular-Ap}, \eqref{eqn:J2-final-est-2}, and Young's inequality, we can obtain that
\begin{align*}
P_{3221}&\lesssim \sum_{(k,j)\in
	\mathscr{H}_{22}}\int_{E_{j}^{k}}\left({\prod_{l=1,4}
	\langle h_{l} \rangle_{\sigma_{\rho(l)},\Q}^{\delta(\Q)}
	\sigma_{l}(\Q)^{\frac{\delta(\Q)}{(p_l)_-(\Q)}}}\right)
\prod_{l=1}^{2}\bigg(\frac{\|\omega_{l}^{-1}\chi_{\Q}
	\|_{p_{l}^{\prime}(\cdot)}}{\mu(\Q)^{1 - \frac{\eta}{{2}}}}\bigg)^{q(x)}u(x)\,d \mu \\
& \leq \sum_{(k,j)\in \mathscr{H}_{22}}
\left(\prod_{l=1,4}\langle h_{l} \rangle_{\sigma_{\rho(l)},\Q}^{\delta(\Q)}
\sigma_{l}(\Q)^{\frac{\delta(\Q)}{(p_l)_-(\Q)}}\right)
\int_{\Q}{\prod_{l=1}^{2}\|\omega_{l}^{-1}\chi_{\Q} 
	\|_{p_{l}^{\prime}(\cdot)}^{q(x)}\mu(\Q)^{(\eta - 2)q(x)}u(x)\,d \mu}\\
&\lesssim \left(\sum_{(k,j)\in \mathscr{H}_{22}}
\langle h_{1}\rangle_{\sigma_{1},\Q}^{(p_{1})_-(\Q)}\sigma_{1}(\Q)
+\sum_{(k,j)\in \mathscr{H}_{22}}\langle
h_{4}\rangle_{\sigma_{2},\Q}^{(p_{2})_-(\Q)}
\sigma_{2}(\Q)\right)^{\frac{{\delta (Q_j^k)}}{{\beta (Q_j^k)}}}\\
&\lesssim  \sum_{\theta=1,c}\left(\sum_{(k,j)\in \mathscr{H}_{22}}
\langle h_{1}\rangle_{\sigma_{1},\Q}^{(p_{1})_-(\Q)}\sigma_{1}(\Q)
+\sum_{(k,j)\in \mathscr{H}_{22}}\langle
h_{4}\rangle_{\sigma_{2},\Q}^{(p_{2})_-(\Q)}
\sigma_{2}(\Q)\right)^\theta.
\end{align*}     
where $c={c_{\eta,p_1(\cdot),p_2(\cdot)}}\ge1$, and it is independent  on $\Q$.

The estimate for the second term in the parenthesis is the same as the
final estimate for $P_{312}$. The estimate for the first term is the same as the estimate
for $P_{311}$ above, noting that since $h_1\geq 1$ and by H\"older's
inequality,
\[ \langle h_{1}\rangle_{\sigma_{1},\Q}^{(p_{1})_-(\Q)}
\leq \langle
h_{1}^{\frac{\pap}{(p_{1})_-(\Q)}}\rangle_{\sigma_{1},\Q}^{(p_{1})_-(\Q)}
\leq \langle
h_{1}^{\frac{\pap}{(p_{1})_-}}\rangle_{\sigma_{1},\Q}^{(p_{1})_-}. \]
Thus, it follows that $P_{32} \le P_{321}+P_{322} \lesssim 1.$ Further, $P_3 \le P_{31}+P_{32} \lesssim 1$ and $I_2 \le P_1+P_2+P_3 \lesssim 1$.
This completes the estimate for $I_2$.

\textbf{Estimate for $I_4$:}
The estimate for $I_4$ is similar to that for $I_2$. We will decompose $I_4$ into the same parts as above. Some parts have a very similar estimate to $I_2$, so we will only give key inequalities and omit some details. For other parts, we will need to modify the argument and present these in more detail.

By the multilinear fractional Calder\'on-Zygmund cubes related to $\M^\mathcal{D}(h_2\sigma_1,h_4\sigma_2)$.  We may decompose and define the sets $\mathscr{F}$, $\mathscr{G}$
and $\mathscr{H}$. In this situation, we define the sums over these sets by $M_1$, $M_2$ and $M_3$.

\textbf{Estimate for $M_1$:} 

Similar to the estimate for $P_1$, from the fact that
$h_2,\,h_4 \leq 1$, we have

\begin{align*}
M_1
&= \sum_{(k,j)\in \mathscr{F}}\int_{E_{j}^{k}}{\prod_{l=2,4}\langle h_{l}\sigma_{\rho(l)}\rangle_{\Q}^{q(x)}\mu(\Q)^{\eta q(x)}u(x)\,d \mu}\\
&\leq \sum_{(k,j)\in \mathscr{F}}\int_{E_{j}^{k}}{\prod_{l=1}^{2}\langle \sigma_{l}\rangle_{\Q}^{q(x)}\mu(\Q)^{\eta q(x)}u(x)\,d \mu}\\
&= \sum_{(k,j)\in \mathscr{F}}\int_{E_{j}^{k}}
\prod_{l=1}^{2}\sigma_{l}(\Q)^{q(x)-\delta(\Q)}
\sigma_{1}(\Q)^{\delta(\Q)}\sigma_{2}(\Q)^{\delta(\Q)}\mu(\Q)^{(\eta  - 2)q(x)}u(x)\,d \mu \\
&\leq \sum_{(k,j)\in \mathscr{F}}
\prod_{l=1}^{2}\bigg(1+\sigma_{l}(\Q) \bigg)^{q_{+}(\Q)-\delta(\Q)}
\int_{E_{j}^{k}}\sigma_{1}(\Q)^{\delta(\Q)}\sigma_{2}(\Q)^{\delta(\Q)}\mu(\Q)^{(\eta  - 2)q(x)}u(x)\,d \mu\\
&\lesssim \prod_{l=1}^{2}\bigg(1+\sigma_{l}(Q_0)\bigg)^{q_{+}-\delta_{-}}
\sum_{(k,j)\in \mathscr{F}}
\sigma_{1}(\Q)^{\frac{\delta(\Q)}{(p_{1})_{-}(\Q)}}\sigma_{2}(\Q)^{\frac{\delta(\Q)}{(p_{2})_{-}(\Q)}}\\
&\lesssim \sum_{\theta=1,c}\left(\sum_{(k,j)\in \mathscr{F}}\sigma_{1}(\Q)
+ \sum_{(k,j)\in \mathscr{F}}\sigma_{2}(\Q)\right)^\theta\\
&\lesssim \sum_{\theta=1,c}\left(\sum_{(k,j)\in \mathscr{F}}\sigma_{1}(E_{j}^{k})
+\sum_{(k,j)\in \mathscr{F}}\sigma_{2}(E_{j}^{k})\right)^\theta\\
&\leq \sum_{\theta=1,c}\left(\sigma_{1}(Q_0)+\sigma_{2}(Q_0) \right)^\theta\\
& \lesssim 1,
\end{align*}
where $c={c_{\eta,p_1(\cdot),p_2(\cdot)}}\ge1$, and it is independent  on $\Q$.

\textbf{Estimate for $M_2$:} 

Similar to the estimate for $P_2$, it follows from the
condition of $A_{\vec{p}(\cdot),q(\cdot)}$ and
Lemma \ref{Gu6} that

\begin{align*}
&{\mu(\Q)^{\eta  - 2}}\int_{\Q}{h_{2}\sigma_{1}\,d \mu}\int_{\Q}{h_{4}\sigma_{2}\,d \mu} \\
 \le& \left[ {\vec \omega } \right]_{{A_{\vec p( \cdot ),q( \cdot )}}(X)} \|\omega\chi_{\Q}\|_{\qq}^{-1}
\prod_{l=1}^{2}
\|\omega_{l}^{-1}\chi_{\Q}\|_{p_{l}^{\prime}(\cdot)}^{-1}
\|h_{2}\|_{L_{\sigma_{1}}^{\pap}}\|h_{4}\|_{L_{\sigma_{2}}^{\pbp}}
\|\chi_{\Q}\|_{L_{\sigma_{1}}^{\cpap}}\|\chi_{\Q}\|_{L_{\sigma_{2}}^{\cpbp}}\\
  =&\left[ {\vec \omega } \right]_{{A_{\vec p( \cdot ),q( \cdot )}}(X)}
\|\omega\chi_{\Q}\|_{\qq}^{-1}\prod_{l=1}^{2}\|\omega_{l}^{-1}\chi_{\Q}\|_{p_{l}^{\prime}(\cdot)}^{-1}\|h_{2}\|_{L_{\sigma_{1}}^{\pap}}\|h_{4}\|_{L_{\sigma_{2}}^{\pbp}}\|\omega_{1}^{-1}\chi_{\Q}\|_{\cpap}
\|\omega_{2}^{-1}\chi_{\Q}\|_{\cpbp}.
\end{align*}
Due to the fact that $1\le u(Q_0)\lesssim u(Q_j^k)$, applying Lemma \ref{p.omega}, the above can be bounded by
$\left[ {\vec \omega } \right]_{{A_{\vec p( \cdot ),q( \cdot )}}(X)} \|\omega\chi_{\Q}\|_{\qq}^{-1} 
 \le c_0\left[ {\vec \omega } \right]_{{A_{\vec p( \cdot ),q( \cdot )}}(X)}$.

Then, it follows from Lemma \ref{lemma:p-infty-px} that
\begin{align*}
M_2
&= \sum_{(k,j)\in \mathscr{G}}\int_{E_{j}^{k}}
\prod_{l=2,4}\langle h_{l}\sigma_{\rho(l)}\rangle_{\Q}^{q(x)}{\mu(\Q)^{\eta \cdot {q(x)}}}u(x)\,d \mu \\
& \lesssim \sum_{(k,j)\in \mathscr{G}}
\int_{E_{j}^{k}}\bigg((c_0\left[ {\vec \omega } \right]_{{A_{\vec p( \cdot ),q( \cdot )}}(X)})^{-1}{{\mu(\Q)^{\eta  - 2}}}
\int_{\Q}h_{2}\sigma_{1}d \mu \int_{\Q}h_{4}\sigma_{2}d \mu\bigg)^{q(x)}u(x)\,d \mu\\
&\lesssim \sum_{(k,j)\in \mathscr{G}}
\int_{E_{j}^{k}} \prod_{l=2,4}\langle
h_{l}\sigma_{\rho(l)}\rangle_{\Q}^{q_{\infty}}{\mu(\Q)^{\eta  \cdot {q_\infty }}}u(x)\,d \mu
+ \sum_{(k,j)\in \mathscr{G}}\int_{E_{j}^{k}}\frac{u(x)}{(e+d(x,x_0))^{tq_{-}}}\,d\mu.
\end{align*}

It is stemmed from Lemma~\ref{lemma:infty-bound} that the second sum of above is bounded by 1.
By \eqref{P_infty_Bounded}, we have  
 
\begin{align*}
&\sum_{(k,j)\in \mathscr{G}}\int_{E_{j}^{k}}\prod_{l=2,4}
\langle h_{l}\sigma_{\rho(l)}\rangle_{\Q}^{q_{\infty}}{\mu(\Q)^{\eta  \cdot {q_\infty }}}u(x)\,d \mu \\
& \quad  = \sum_{(k,j)\in
	\mathscr{G}}\int_{E_{j}^{k}}\prod_{l=2,4}
\langle h_{l}\rangle_{\sigma_{\rho(l)},\Q}^{q_{\infty}}
\sigma_{1}(\Q)^{q_{\infty}}\sigma_{2}(\Q)^{q_{\infty}}\mu(\Q)^{(\eta  - 2)q_{\infty}}u(x)\,d \mu
\\
& \quad \lesssim \sum_{(k,j)\in \mathscr{G}}
\langle h_{2}\rangle_{\sigma_{1},\Q}^{q_{\infty}}
\sigma_{1}(\Q)^{\frac{q_{\infty}}{(p_{1})_{\infty}}}
\langle h_{4}\rangle_{\sigma_{2},\Q}^{q_{\infty}}
\sigma_{2}(\Q)^{\frac{q_{\infty}}{(p_{2})_{\infty}}}.\\
\intertext{Using Young's inequality and by
	Lemmas~\ref{Suff_1},~\ref{lemma:p-infty-px},
	and~\ref{lemma:infty-bound}, the above sum}
& \quad \lesssim \left(\sum_{(k,j)\in \mathscr{G}}
\langle h_{2}\rangle_{\sigma_{1},\Q}^{(p_{1})_{\infty}}\sigma_{1}(\Q)
+  \sum_{(k,j)\in \mathscr{G}}
\langle h_{4}\rangle_{\sigma_{2},\Q}^{(p_{2})_{\infty}}\sigma_{2}(\Q)\right)^{\frac{{{q_\infty }}}{{{p_\infty }}}}\\
& \quad \lesssim \left(\sum_{(k,j)\in \mathscr{G}}
\langle
h_{2}\rangle_{\sigma_{1},\Q}^{(p_{1})_{\infty}}\sigma_{1}(E_{j}^{k})
+  \sum_{(k,j)\in \mathscr{G}}
\langle
h_{4}\rangle_{\sigma_{2},\Q}^{(p_{2})_{\infty}}\sigma_{2}(E_{j}^{k})\right)^{\frac{{{q_\infty }}}{{{p_\infty }}}}\\ 
& \quad \le 
\left(\int_{\subRn}{M}_{\sigma_{1}}^\mathcal{D} h_{2}(x)^{(p_{1})_{\infty}}\sigma_{1}(x)\,d \mu
+\int_{\subRn}{M}_{\sigma_{2}}^\mathcal{D} h_{4}(x)^{(p_{2})_{\infty}}\sigma_{2}(x)\,d \mu\right)^{\frac{{{q_\infty }}}{{{p_\infty }}}}\\
& \quad \lesssim 
\left(\int_{\subRn}h_{2}(x)^{(p_{1})_{\infty}}\sigma_{1}(x)\,d \mu
+ \int_{\subRn}h_{4}(x)^{(p_{2})_{\infty}}\sigma_{2}(x)\,d \mu\right)^{\frac{{{q_\infty }}}{{{p_\infty }}}}\\
& \quad \lesssim  \left(\int_{\subRn}h_{2}(x)^{p_{1}(x)}\sigma_{1}(x)\,d \mu
+\int_{\subRn}h_{4}(x)^{p_{2}(x)}\sigma_{2}(x)\,d \mu+ \sum\limits_{l = 1}^2 {\int_{\subRn}\frac{\sigma_{l}(x)}{(e+d(x,x_0))^{t(p_{l})_{-}}}\,d \mu}\right)^{\frac{{{q_\infty }}}{{{p_\infty }}}} \\
& \quad \lesssim 1.
\end{align*}

\textbf{Estimate for $M_3$:} 

Similar to estimating $P_3$, the set $\mathscr{H}$ may be divided into $\mathscr{H}_{1}$ and $\mathscr{H}_{2}$. Further, here we continue to divided $\mathscr{H}_{1}$ into the following two sets. We denote them by
\[ \mathscr{H}_{11} = \{ (k,j) \in \mathscr{H}_{1} : \sigma_1(\Q)\leq
1, \sigma_2(\Q) \leq 1 \} \]
and
\[ \mathscr{H}_{12} = \{ (k,j) \in \mathscr{H}_{1} : \sigma_1(\Q)>
1, \sigma_2(\Q) \leq 1\}. \]
The corresponding sums over $\mathscr{H}_{11}$ and $\mathscr{H}_{12}$ are defined as $M_{311}$ and $M_{312}$ respectively.

The estimate for $M_{311}$ is similar to that for $P_{31}$. 
Note that the fact that $h_2,\,h_4\leq 1$.  From Lemmas \ref{lemma:p-infty-px} and \ref{lemma:infty-bound}, we get
\begin{align*}
M_{311}=&\sum_{(k,j)\in \mathscr{H}_{11}} \int_{E_{j}^{k}}
\prod_{l=2,4}\langle h_{l}\sigma_{\rho(l)} \rangle_{\Q}^{q(x)}
{\mu(\Q)^{\eta  \cdot q(x)}}u(x)\,d \mu \\
  \lesssim& \sum_{(k,j)\in \mathscr{H}_{11}} \int_{E_{j}^{k}}
\prod_{l=2,4}\langle h_{l}\sigma_{\rho(l)}
\rangle_{\Q}^{q_{+}(\Q)} {\mu(\Q)^{\eta  \cdot {q_{+}(\Q)}}}u(x)\,d \mu
+\sum_{(k,j)\in \mathscr{H}_{11}}\int_{E_{j}^{k}}\frac{u(x)}{(e+d(x,x_0))^{tq_{-}}}\,d \mu\\
 \le& \sum_{(k,j)\in \mathscr{H}_{11}}
\prod_{l=2,4}\langle h_{l}
\rangle_{\sigma_{\rho(l)},\Q}^{q_{+}(\Q)}  
\int_{E_{j}^{k}}\mu(\Q)^{(\eta  - 2) q_{+}(\Q)} \prod_{l=2,4}\sigma_{\rho(l)}(\Q)^{q_{+}(\Q)} u(x)\,d \mu + 1.
\intertext{Firstly, by Lemma \ref{lemma:diening}, we may replace $\mu(Q_j^k)^{(\eta  - 2)q(x)}$ with $\mu(Q_j^k)^{(\eta  - 2)q_+(Q_j^k)}$ and apply this to \eqref{EQ-I1estimte}. Then due to the fact that for $l=1,\,2$, $\sigma_l(\Q)\leq 1$, the above }
 \leq& \sum_{(k,j)\in \mathscr{H}_{11}}
\prod_{l=2,4}\langle h_{l}
\rangle_{\sigma_{\rho(l)},\Q}^{\delta(\Q)}  
\int_{E_{j}^{k}}\mu(\Q)^{(\eta  - 2) q_{+}(\Q)}
\prod_{l=2,4}\sigma_{\rho(l)}(\Q)^{\delta(\Q)} u(x)\,d \mu + 1 \\
\lesssim&  \left(\sum_{(k,j)\in \mathscr{H}_{11}}
\prod_{l=2,4}\langle h_{l}\rangle_{\sigma_{\rho(l),\Q}}^{\delta(\Q)}
\sigma_{1}(\Q)^{\frac{\delta(\Q)}{(p_{1})_{-}(\Q)}}\sigma_{2}(\Q)^{\frac{\delta(\Q)}{(p_{2})_{-}(\Q)}}
+ 1\right).\\
\intertext{By Young's inequality, the above}
\lesssim&  \left(\sum_{\theta=1,c}\left(\sum_{(k,j)\in \mathscr{H}_{11}}
\langle
h_{2}\rangle_{\sigma_{1},\Q}^{(p_{1})_{-}(\Q)}
\sigma_{1}(\Q)
+ \sum_{(k,j)\in \mathscr{H}_{11}}
\langle h_{4}\rangle_{\sigma_{2},\Q}^{(p_{2})_{-}(\Q)}
\sigma_{2}(\Q)\right)^\theta+ 1\right)\\
\lesssim& 1.
\end{align*}
where $c={c_{\eta,p_1(\cdot),p_2(\cdot)}}\ge1$, and it is independent  on $\Q$. The last estimate is valid since we can do as that of $P_{312}$.

\medskip

Next, by \eqref{H2-first-estimate} (we can replace the $h_1$ with $h_2$) and \eqref{H2-second-estimate}, it follows from Lemma \ref{lemma:p-infty-px} that
\begin{align*}
M_{312}=&\sum_{(k,j)\in \mathscr{H}_{12}} \int_{E_{j}^{k}}\prod_{l=2,4}
\langle h_{l}\sigma_{\rho(l)}\rangle_{\Q}^{q(x)}{\mu(\Q)^{\eta  \cdot q(x)}}u(x)\,d \mu \\
  \lesssim& \sum_{(k,j)\in
	\mathscr{H}_{12}}\int_{E_{j}^{k}}
\prod_{l=2,4}\bigg(c_0^{-1}\|\omega_{\rho(l)}^{-1}\chi_{\Q}\|_{p_{\rho(l)}^{\prime}(\cdot)}^{-1}
\int_{\Q}h_{l}\sigma_{\rho(l)}d \mu\bigg)^{\delta(\Q)}
\prod_{l=1}^{2}\bigg(\frac{\|\omega_{l}^{-1}\chi_{\Q}\|_{p_{l}^{\prime}(\cdot)}}
{\mu(\Q)^{1 - \frac{\eta }{{2}}}}\bigg)^{q(x)}u(x)\,d \mu \\
  &+
\sum_{(k,j)\in \mathscr{H}_{12}} \int_{E_{j}^{k}} \prod_{l=1}^{2}
\bigg(\frac{\|\omega_{l}^{-1}\chi_{\Q}\|_{p_{l}^{\prime}(\cdot)}}{\mu(\Q)^{1 - \frac{\eta}{{2}}}}\bigg)^{q(x)}
\frac{u(x)}{(e+d(x,x_0))^{tq_{-}}}\,d \mu \\
=:&{M}_{3121}+ {M}_{3122}.
\end{align*}

On the one hand, estimate for ${M}_{3122}$ is same as the estimate for $P_{3212}$.
On the other hand, similar to the estimate for $P_{3221}$, we can get that
\begin{align*}
{M}_{3121}&\lesssim \sum_{(k,j)\in
	\mathscr{H}_{12}}\int_{E_{j}^{k}}\left({\prod_{l=2,4}
	\langle h_{l} \rangle_{\sigma_{\rho(l)},\Q}^{\delta(\Q)}
	\sigma_{l}(\Q)^{\frac{\delta(\Q)}{(p_l)_-(\Q)}}}\right)
\prod_{J=1}^{2}\bigg(\frac{\|\omega_{J}^{-1}\chi_{\Q}
	\|_{p_{J}^{\prime}(\cdot)}}{\mu(\Q)^{1 - \frac{\eta}{{2}}}}\bigg)^{q(x)}u(x)\,d \mu \\
& \leq \sum_{(k,j)\in \mathscr{H}_{12}}
\left(\prod_{l=2,4}\langle h_{l} \rangle_{\sigma_{\rho(l)},\Q}^{\delta(\Q)}
\sigma_{l}(\Q)^{\frac{\delta(\Q)}{(p_{l})_-(\Q)}}\right)
\int_{\Q}{\prod_{J=1}^{2}\|\omega_{J}^{-1}\chi_{\Q} 
	\|_{p_{J}^{\prime}(\cdot)}^{q(x)}\mu(\Q)^{(\eta  - 2)q(x)}u(x)\,d \mu}\\
&\lesssim \sum_{\theta=1,c}\left(\sum_{(k,j)\in \mathscr{H}_{12}}
\langle h_{2}\rangle_{\sigma_{1},\Q}^{(p_{1})_-(Q_j^k)}\sigma_{1}(\Q)
+\sum_{(k,j)\in \mathscr{H}_{12}}\langle
h_{4}\rangle_{\sigma_{2},\Q}^{(p_{2})_-(Q_j^k)}
\sigma_{2}(Q_j^k)\right)^\theta.
\end{align*}     
where $c={c_{\eta,p_1(\cdot),p_2(\cdot)}}\ge1$, and it is independent  on $\Q$.

The above is the same as the estimate for $P_{312}$. Then, $M_{312} \lesssim 1$.
\medskip

To estimate the sum over $\mathscr{H}_{2}$, $M_{32}$, similar to that of $P_3$, we consider the following two sets  $\mathscr{H}_{21}$ and $\mathscr{H}_{22}$.

On the one hand, the estimate over $\mathscr{H}_{21}$, ${M}_{321}$, is analogous to the estimate
over this set as before. we only need to replace $h_1$ with $h_2$.  This procedures will produce two 
terms just like $P_{3211}$ and $P_{3212}$ above and we denote them by $M_{3211}$ and $M_{3212}$ respectively. The estimate for $M_{3212}$ is the same as the estimate for $P_{3212}$. The estimate for $M_{3211}$ is the similar to the estimate for $P_{3211}$. Sure, we need to remind that in the final line, the $h_2$ term can be estimated like the $h_4$ term since $h_2,\,h_4 \leq 1$.  
In short, we can easily get that ${M}_{321} \lesssim 1$.

On the other hand, for estimating the sum over $\mathscr{H}_{22}$ that is ${M}_{322}$, we can do as before, getting the terms like $P_{3221}$ and $P_{3222}$ by which we denote $M_{3221}$ and $M_{3222}$ respectively, and we also need to replace $h_1$ by $h_2$.  The
estimate of $M_{3222}$ is again the same.  To estimate $M_{3222}$ we argue as before 
except we replace the exponent $\delta(\Q)$ by $q_\infty$.   But then the last line of the estimate can be 
$$ \sum_{\theta=1,c}\left(\sum_{(k,j)\in \mathscr{H}_{22}}
\langle h_{1}\rangle_{\sigma_{1},\Q}^{(p_{1})_\infty}\sigma_{1}(\Q)
+\sum_{(k,j)\in \mathscr{H}_{22}}\langle
h_{4}\rangle_{\sigma_{2},\Q}^{(p_{2})_{\infty}}
\sigma_{2}(\Q)\right)^\theta.$$
The subsequent process is completely similar to the estimate for $P_2$.  

In a nutshell, we can easily derive that ${M}_{322} \lesssim 1.$
It follows immediately that $I_4 \le M_1+M_2+M_3 \lesssim 1.$

All in all, $\sum\limits_{i = 1,2,3,4} {{I_i}} \lesssim 1$, which is the desired result.

{\bf{Case 2: $\mu(X) < \infty$.}} 

Last but not least, we consider the case for $\mu(X) < \infty$. We alway can regard $X$ as a dyadic cube $Q_0$ due to Lemmas $\ref{finite-bounded}$ and $\ref{cubes}$. This proof is similar to before and we just need to make some changes for Calderón-Zygmund Decomposition. 
We also consider $f_i$ is a nonnegative funcion with \(\|\omega_{i} f_i \|_{p_i(\cdot)} =1\)  and decompose $\vec{f}:=(f_1,f_2)=(h_1,h_3)+(h_1,h_4)+(h_2,h_3)+(h_2,h_4) $ as before.
The construction of Calderón-Zygmund cubes at any height \(\lambda>\lambda_0 := \mu {\left( {Q_j^k} \right)^{\eta  - 2}}\int_{{Q^k}} {h_s \sigma_{\rho(s)}d\mu }\int_{{Q^k}} {h_t \sigma_{\rho(t)} d\mu }\), where $s \in \left\{ {1,2} \right\}$ and $t \in \left\{ {3,4} \right\}$.  

By the condition of $A_{\vec{p}(\cdot),q(\cdot)}$, Lemma \ref{Holder}, and Lemma \ref{finite-bounded},
\begin{align*}
    \lambda_0 \leq \mu \left( {X} \right)^{\eta  - 2}\prod_{l=s,t} \int_{{X}} {h_l \sigma_{\rho(l)}d\mu } 
    &\le 16 \, \mu {\left( {{X}} \right)^{\eta  - 2}} \prod_{l=i,j} \|h_l\sigma_{\rho(l)}^{\frac{1}{p_l(\cdot)}}\|_{p_{\rho(l)}(\cdot)} \| {\sigma_{\rho(l)}^{\frac{1}{p'_l(\cdot)}}} \|_{p'_{\rho(l)}( \cdot )}\\
    &\le 16 \, \mu {\left( {{X}} \right)^{\eta  - 2}} \prod_{l=1,2}\|\omega_{l}^{-1}\|_{p'_{l}(\cdot)}\\
    & \le 16 {\left[ \vec{\omega}  \right]_{{A_{\vec{p}( \cdot ),q( \cdot )}}}}\| {\omega } \|_{q( \cdot )}^{ - 1} \lesssim 1, 
\end{align*}
where $s=1,2$ and $t=3,4$.

By Lemma \ref{CZD}, set $a=2 C_{C Z}$ and $\left\{ {Q_j^k} \right\}$ is the Calderón-Zygmund cubes of $f_i$ at height $a^k$, for all integers $k \geq k_0=[\log _a \lambda_0]$. Then 
$$
X = X_{\eta ,{a^{{k_0}}}}^{\mathcal D}\bigcup {{{\left( {X_{\eta ,{a^{{k_0}}}}^{\mathcal D}} \right)}^c}}  = \left( {\bigcup\limits_{k = {k_0}}^\infty  {X_{\eta ,{a^k}}^{\mathcal D}\backslash X_{\eta ,{a^{k + 1}}}^{\mathcal D}} } \right)\bigcup {{{\left( {X_{\eta ,{a^{{k_0}}}}^{\mathcal D}} \right)}^c}},
$$
where $X_{\eta,a^{k_0}}^{\mathcal D}:=\left\{x \in X: \M_\eta^{\mathcal{D}} \vec{f}(x)>\lambda_0\right\} \subseteq \left\{ {Q_j^k} \right\}$.

From {\bf step 3} as before. It is sufficient to demonstrate that
\begin{align*}
    \int_{X}\M_\eta^{\d}(h_{s}\sigma_{\rho(s)},h_{t}\sigma_{\rho(t)})(x)^{q(x)}u(x)\,d\mu \lesssim 1,
\end{align*}
and it follows immediately from the derivation of \eqref{case2_1} that
\begin{align*}
&\int_{X}\M_\eta^{\d}(h_{s}\sigma_{\rho(s)},h_{t}\sigma_{\rho(t)})(x)^{q(x)}u(x)\,d\mu \\ 
=&\int_{{{{\left( {X_{\eta ,{a^{{k_0}}}}^{\mathcal D}} \right)}^c}}} \M_\eta^{\mathcal{D}} (h_{s}\sigma_{\rho(s)},h_{t}\sigma_{\rho(t)})^{q(\cdot)} u +\sum_{k=k_0}^{\infty} \int_{{X_{\eta ,{a^k}}^{\mathcal D}\backslash X_{\eta ,{a^{k + 1}}}^{\mathcal D}}} \M_\eta^{\mathcal{D}} (h_{s}\sigma_{\rho(s)},h_{t}\sigma_{\rho(t)})^{q(\cdot)} u \\
 \lesssim & \lambda_0 u(X)+ \sum_{k,j}\int_{E_{j}^{k}}\prod_{l=s,t}
 \bigg(\int_{\Q}{h_{l}\sigma_{\rho(l)}d \mu}\bigg)^{q(x)}
 \mu(Q_j^k)^{(\eta-2)q(x)}u(x)\,d\mu,
\end{align*} 
where $s=1,2$ and $t=3,4$. 

The first term is obviously bounded by a constant. 
When $(s,t) = (1,3),(1,4),(2,4)$, ones can follows the methods of $I_1,P_1,M_1$ respectively to get the results that the second term is bounded by a constant.

We finally finish this proof.

\section{\bf The proof of Theorem \ref{Text.condi.}}\label{sec.5.}
Similar to before, for the linear case, we decompose a nonnegative $f=h_1+h_2$, where $h_1=f \chi_{\left\{f>1\right\}}$ and $h_2=$ $f \chi_{\left\{f \leq 1\right\}}$. It can be deduced from Lemma \ref{p.omega} that
\begin{equation*}\label{Suff.ineq_0}
\left\|h_i \sigma\right\|_{p(\cdot)} \leq \left\|f \sigma\right\|_{p(\cdot)} = 1, \quad i=1,2.
\end{equation*}
By Lemma \ref{p.omega} and the sublinearity of $M_\eta^{\mathcal{D}}$, it suffices to prove that 
\begin{equation}\label{Suff.ineq_1}
H_i:=\int_X(M_\eta^{\mathcal{D}} (h_i\sigma)(x)^{q(x)} \omega(x)^{q(x)} d \mu \lesssim \left(1+[\omega ]_{C_{p( \cdot ),q( \cdot )}^1}\right)[\omega, v]_{C_{p(\cdot), q(\cdot)}^2}, \quad i=1,2,
\end{equation}
From \eqref{zdkz_}, we observe that
\begin{align*}
    M_\eta^{\d}(f \sigma)(x) \approx \sum_{k,j}\left(\mu(\Q)^{\eta - 1}\int_{\Q}f \sigma d \mu \right)\chi_{E_j^k}(x).
\end{align*}

\textbf{Estimate for $H_1$:}

It follows from  \eqref{Bound_function_geq1} that 
\begin{align*}
&\quad \int_X M_\eta^{\mathcal{D}} (h_1 \sigma)(x)^{q(x)} u(x) d \mu \\
& \approx \sum_{k,j}\int_{E_j^k}(\int_{\Q}h_1\sigma d\mu)^{q(x)}(\mu(\Q)^{\eta - 1})^{q(x)}u(x) d \mu \\
& \le [\omega ]_{C_{p( \cdot ),q( \cdot )}^1}\sum_{k,j}\langle h_1\rangle_{\sigma,\Q}^{q_-} \int_{E_j^k} \sigma(\Q)^{q(x)}(\mu(\Q)^{\eta - 1})^{q(x)}u(x)d\mu\\
&=: [\omega ]_{C_{p( \cdot ),q( \cdot )}^1}\sum_{k,j}\langle h_1\rangle_{\sigma,\Q}^{q_-} \Gamma_j^k
\end{align*}

\textbf{Estimate for $H_2$:}

\begin{align*}
&\quad \int_X M_\eta^{\mathcal{D}} (h_2 \sigma)(x)^{q(x)} u(x) d \mu \\
& \approx \sum_{k,j}\int_{E_j^k}(\int_{\Q}h_2\sigma d\mu)^{q(x)}(\mu(\Q)^{\eta - 1})^{q(x)}u(x) d \mu \\
& \lesssim \sum_{k,j}\langle h_2\rangle_{\sigma,\Q}^{q_-} \int_{E_j^k} \sigma(\Q)^{q(x)}(\mu(\Q)^{\eta - 1})^{q(x)}u(x)d\mu.\\
& =:\sum_{k,j}\langle h_2\rangle_{\sigma,\Q}^{q_-} \Gamma_j^k.
\end{align*}

Now let $\Omega:=\{(k,j)\}_{k,j\in \Z}$, for any $E \subseteq \Omega$, we define the measure $\Gamma$ by $\Gamma(E):=\sum_{(k,j) \in E}\Gamma_j^k$.

Define
$$
S:\left(L^1+L^{\infty}\right)\left(X, \sigma d \mu\right) \rightarrow L^{\infty}(\Omega, d\Gamma)
$$
by
$$
S(g)(k, j):=\langle g\rangle_{\sigma, \Q}, \quad g \in\left(L^1+L^{\infty}\right)\left(X, \sigma d \mu\right) .
$$
Obviously, ${\left\| S \right\|_{{L^\infty }(\sigma) \to {L^\infty }(\Omega,\Gamma)}} \le 1$. 
Let $\lambda>0$ and $\{J_i\}_i$ denotes the maximal cubes relative to the collection
\[
  \{Q_j^k: \langle g\rangle_{\sigma, \Q}>\lambda\}.
\]

It follows from the definition of $C_{\vec{p}(\cdot),q(\cdot)}(X)$ that
\begin{align*}
    &\quad \quad \quad \Gamma\{(k,j):S(f)(k,j)>\lambda\}\\
    &=\sum_{\{(j,k):\, S(f)(k,j)>\lambda\}} \Gamma_j^k =\sum_i \sum_{Q_j^k\subseteq J_i}\Gamma_j^k\\
    &\leq \sum_i \sum_{Q_j^k\subseteq J_i} \int_{E_{j}^{k}}M_{\eta}^{\d}(\sigma \chi_{J_i})^{q(x)}u(x) d\mu\\
    &\leq \sum_i  \int_{J_i}M_{\eta}^{\d}(\sigma \chi_{J_i})^{q(x)}u(x) d\mu\\
    &\leq [\omega, v]_{C_{p(\cdot), q(\cdot)}^2}^{q_-} \sum_{i}\left(\int_{\Q}\sigma \chi_{J_i} d\mu\right)^{\frac{q_-}{{{p_t}}}}\\
& \leq [\omega, v]_{C_{p(\cdot), q(\cdot)}^2}^{q_-} {\left( {{\lambda ^{ - 1}}\int_X {f\sigma } } \right)} ^{\frac{q_-}{{{p_t}}}}
\end{align*}


This shows that ${\left\| S \right\|_{{L^1(\sigma)} \to W{L^{{{({p_t})}^{ - 1}}{q_ - }}(\Omega,d\Gamma)}}} \le [\omega ,v]_{C_{p( \cdot ),q( \cdot )}^2}^{{p_t}}$.
It is derived from Marcinkiewicz interpolation theorem (\cite{249}, Theorem 1.3.2) that ${\left\| S \right\|_{{L^{p_t}(\sigma)}\to {L^{q_-}(\Omega,d\Gamma)}}} \le [\omega ,v]_{C_{p( \cdot ),q( \cdot )}^2}$.
This implies that
\begin{align*}
    \|S(h_t)\|_{L^{q_-}(\Omega,d\Gamma)} \lesssim [\omega,v]_{C_{p(\cdot), q(\cdot)}^2} \|h_t\|_{L^{p_t}(\sigma)}.
\end{align*}
Then
\begin{align*}
    H_1 + H_2 &\lesssim \left(1+[\omega ]_{C_{p( \cdot ),q( \cdot )}^1}\right)\sum_{t = 1,2} \sum_{k,j}\langle h_t\rangle_{\sigma,\Q}^{q_-} \Gamma_j^k\\
    & \lesssim \left(1+[\omega ]_{C_{p( \cdot ),q( \cdot )}^1}\right)[\omega, v]_{C_{p(\cdot), q(\cdot)}^2}\sum_{t=1,2}\|h_t\|_{L^{p_t}(\sigma)}\\
    & \lesssim \left(1+[\omega ]_{C_{p( \cdot ),q( \cdot )}^1}\right)[\omega, v]_{C_{p(\cdot), q(\cdot)}^2}\left((\int_{\Q}|h_1|^{p(\cdot)}\sigma d\mu)^{\frac{1}{p_-}} + (\int_{\Q}|h_2|^{p(\cdot)}\sigma d\mu)^\frac{1}{p_+}\right)\\
    & \leq 2\left(1+[\omega ]_{C_{p( \cdot ),q( \cdot )}^1}\right)[\omega, v]_{C_{p(\cdot), q(\cdot)}^2}.
\end{align*}
Hence, we have
\begin{align*}
    \|M_{\eta}\|_{L^{p(\cdot)}(X,\omega)\rightarrow L^{q(\cdot)}(X,v)} \lesssim \sum\limits_{\theta  = \frac{1}{{{p_{\rm{ - }}}}},\frac{1}{{{p_{\rm{ + }}}}}} {{{\left( {{{[\omega ,v]}_{C_{p( \cdot ),q( \cdot )}^2(X)}} + {{[\omega ]}_{C_{p( \cdot ),q( \cdot )}^1(X)}}{{[\omega ,v]}_{C_{p( \cdot ),q( \cdot )}^2(X)}}} \right)}^\theta }}.
\end{align*}
All of the implicit constants mentioned above are independent on $(\omega,v)$.

\vspace{1cm}
\noindent{\bf Acknowledgements } 
The authors would like to thank the editors and reviewers for careful reading and valuable comments, which lead to the improvement of this paper. 

\medskip 

\noindent{\bf Data Availability} Our manuscript has no associated data.

\medskip 
\noindent{\bf\Large Declarations}
\medskip 

\noindent{\bf Conflict of interest} The authors state that there is no conflict of interest.

\end{document}